\documentclass[10pt]{article}
\usepackage{amssymb, fancyhdr,amsmath,amsthm,amsbsy,amsfonts, titletoc,  mathrsfs, mathabx, verbatim,extarrows, wasysym,enumerate}
\usepackage[textwidth=7.2in,textheight=10in]{geometry}

\theoremstyle{plain}
\newtheorem{theorem}{Theorem}[section]
\newtheorem{lemma}[theorem]{Lemma}
\newtheorem{definition}[theorem]{Definition}

\newtheorem{prop}[theorem]{Proposition}

\numberwithin{equation}{section} 

\newcommand{\beq}{\begin{equation}}
\newcommand{\eeq}{\end{equation}}

\newcommand{\ena}{\end{eqnarray}}
\newcommand{\bea}{\begin{eqnarray*}}
\newcommand{\eea}{\end{eqnarray*}}

\newcommand{\p}{{\partial}}

\newcommand{\h}{\hspace}

\newcommand{\overbar}[1]{\mkern 1.5mu\overline{\mkern-1.5mu#1\mkern-1.5mu}\mkern 1.5mu}
\newcommand{\interior}[1]{%
  {\kern0pt#1}^{\mathrm{o}}%
}

\newcommand{\normmm}[1]{{\left\vert\kern-0.25ex\left\vert\kern-0.25ex\left\vert #1
    \right\vert\kern-0.25ex\right\vert\kern-0.25ex\right\vert}}
\newtheorem*{theorem 2.1}{Theorem 2.1}

\def\Xint#1{\mathchoice
  {\XXint\displaystyle\textstyle{#1}}%
  {\XXint\textstyle\scriptstyle{#1}}%
  {\XXint\scriptstyle\scriptscriptstyle{#1}}%
  {\XXint\scriptscriptstyle\scriptscriptstyle{#1}}%
  \!\int}
\def\XXint#1#2#3{{\setbox0=\hbox{$#1{#2#3}{\int}$}
  \vcenter{\hbox{$#2#3$}}\kern-.5\wd0}}

\def\dashint{\Xint-}

\begin{document}
\begin{center}  \textbf{ \large DISCLINATIONS IN LIMITING LANDAU-DE GENNES THEORY\vspace{0.5pc}} \end{center}
\begin{center} \textbf{ \normalsize YONG YU \vspace{0.3pc}\\ \textit{The Chinese University of Hong Kong}}\end{center}\
\\
\textbf{\small ABSTRACT:} In this article we study the low-temperature limit of a Landau-de Gennes theory. Within all $\mathbb{S}^2$-valued $\mathscr{R}$-axially symmetric maps (see Definition 1.1), the limiting energy functional has at least two distinct energy minimizers. One minimizer has biaxial torus structure, while another minimizer has split-core segment structure on the $z$-axis. \vspace{0.5pc}\\
\setcounter{section}{1}
\setcounter{theorem}{0}
\setcounter{equation}{0}

\noindent \textbf{\large I. \small INTRODUCTION}\vspace{0.5pc}\\
Theory of Landau-de Gennes [17] uses tensor-valued order parameter to explain transitions between different phases. Usually the order parameter in the Landau-de Gennes theory is denoted by $\mathscr{Q}$ and the theory is also referred as $\mathscr{Q}$-tensor theory in literatures. Values of $\mathscr{Q}$ are real $3 \times 3$ symmetric matrices with zero trace. Therefore the matrix $\mathscr{Q}$ has three real eigenvalues. Due to quantitative relationship between the three eigenvalues, $\mathscr{Q}$ admits three structures at a point in domain. It is called isotropic if all eigenvalues are zero. It is uniaxial if two of the three eigenvalues are identical and non-zero. It is biaxial if all three eigenvalues are different from each other. There are many different liquid crystal theories, e.g. Ericksen's theory and Oseen-Frank theory. Extensive research works have already been devoted  to the study of Ericksen's theory in [2]-[5], [40] and references in [42]. As for Oseen-Frank theory, readers may refer to [11] and [28]-[29]. Within various theories for liquid crystal materials, Landau-de Gennes theory is a general and unified one. It has been shown in [45] that Ericksen's theory and Oseen-Frank theory can be derived from the Landau-de Gennes theory under uniaxial ansatz and appropriate limiting process, respectively.  \vspace{0.3pc}

The Landau-de Gennes theory and its variant [7] have attracted significant attentions from both physicists and mathematicians. Many research works, e.g. [1], [6], [8]-[10], [13]-[15], [22]-[25], [32], [34]-[38], [46]-[47], [52], [55], [59], focus on equilibrium solutions of the theories. These equilibrium solutions (particularly their core structures near disclinations) play key roles in  understanding phases of liquid crystal materials at different temperatures. In 1988 a radial hedgehog configuration has been considered by the authors in [59]. The configuration is unaxial on all points except the origin at where an isotropic core appears. Its mathematical properties have been addressed in [37], while in [24], [34], [36], [47] and [55], the stability and instability of the radial hedgehog configuration are extensively studied under various parameter regimes and different perturbations. Besides the radial hedgehog configuration, in [52] (see also [36] and [38]), it was found that Landau-de Gennes theory may have a solution with half-degree ring disclination. The order parameter $\mathscr{Q}$ is uniaxial on the ring disclination. Meanwhile the singular core of the radial hedgehog solution is removed by escaping to biaxial phase. It was until 2000 that the third core structure was found by Gartland and Mkaddem. It is called split-core configuration in [25]. Roughly speaking, the isotropic core of the radial hedgehog solution can be split into a line segment on $z$-axis along which equilibrium solution is uniaxial. Moreover the solution is isotropic at the two end-points of the line segment. It exhibits a biaxial phase for points near the line segment disclination and off $z$-axis. The above three fundamental core structures in the Landau-de Gennes theory have also been numerically confirmed in [6], [16], [25], [33], [61] and many references therein. Now it is believed that, under suitable regimes, the radial symmetry of the hedgehog solution can be broken into axial symmetry. There have at least two axially symmetric equilibrium solutions in the Landau-de Gennes theory bifurcating from the hedgehog solution. One admits half-degree ring disclination. Another one has split-core line segment disclination on rotation axis.
\vspace{1.5pc}\\
\textbf{\large I.1. \small AXIALLY SYMMETRIC FORMULATION OF LANDAU-DE GENNES EQUATION} \vspace{1pc}\\
 In the one-constant approximation and when the temperature is below the super-cooling temperature (see [24]), the energy functional of Landau-de Gennes theory can be written as:  \begin{eqnarray} \int_{B_R} \dfrac{1}{2} \h{1pt}\big| \h{1pt}\nabla \mathscr{Q}\h{1pt}\big|^2 - \dfrac{a^2}{2} \h{1pt}\big| \h{1pt}\mathscr{Q}\h{1pt}\big|^2 - \sqrt{6} \h{2pt}\mathrm{tr}\big( \mathscr{Q}^3 \big)  + \dfrac{1}{2}\h{1pt} \big| \h{1pt}\mathscr{Q}\h{1pt}\big|^4.
\end{eqnarray}Here $- a^2$ is the so-called reduced temperature. $B_R$ is the ball in $\mathbb{R}^3$ with center $0$ and radius $R$. $| \h{0.5pt}\mathscr{Q}\h{0.5pt} |^2 := \mathrm{tr}\big( \mathscr{Q}^2 \big)$ is the standard norm of real matrices. Suppose that $\mathbf{I}_3$ is the $3 \times 3$ identity matrix. We can write the Euler-Lagrange equation of  (1.1) as follows:
\begin{eqnarray}
 - \Delta \mathscr{Q}=  a^2 \mathscr{Q} + 3 \sqrt{6} \left( \mathscr{Q} ^2-\frac{1}{3} \h{1pt}| \h{0.5pt}\mathscr{Q} \h{0.5pt}|^2 \h{1pt}\mathbf{I}_3\right)  - 2\h{1pt}|\h{0.5pt}\mathscr{Q}\h{0.5pt}|^2 \mathscr{Q} \h{20pt}\text{in $B_R$.}\vspace{0.2pc}
\end{eqnarray}Equation (1.2) has five degrees of freedom. With axial symmetry, the degree of freedom can be reduced from five to three. Firstly let us introduce some notations. In this article $x = (x_1, x_2, z)$ denotes a $3$-vector in $\mathbb{R}^3$ space.  $\big(r, \varphi, \theta\big)$ are the spherical coordinates of $\mathbb{R}^3$ with respect to the center $0$. Conventionally $r$ is the radial variable. $\varphi \in [\h{0.5pt}0, \pi \h{0.5pt}]$ is the polar angle. $\theta \in [\h{1pt}0, 2 \pi \h{0.5pt} )$ is the azimuthal angle. We also need the radial variable, denoted by $\rho$, in the $x_1$-$x_2$ plane.  In the following $e_r$ is the radial direction $x/r$.  $\big\{e_{\rho}, e_{\theta}, e_z\big\}$ is an orthonormal basis of $\mathbb{R}^3$ with \begin{eqnarray*}e_{\rho} =  \left( \frac{x_1}{\rho}, \frac{x_2}{\rho}, 0 \right), \h{20pt}e_{\theta} =  \left( - \frac{x_2}{\rho}, \frac{x_1}{\rho}, 0 \right), \h{20pt}e_z = \big( 0, 0, 1 \big).
\end{eqnarray*}Supposing that $\displaystyle M_1 := e_r \otimes e_r - \frac{1}{3}\h{2pt}\mathbf{I}_3$, $\displaystyle M_2 := e_z \otimes e_z - \frac{1}{3}\h{1pt}\mathbf{I}_3$, $\displaystyle M_3 := e_\rho \otimes e_z + e_z \otimes e_\rho$, we put  $\mathscr{Q}$ under the ansatz:\begin{eqnarray} \mathscr{Q}= \sum_{j = 1}^3 q_j \h{0.5pt}M_j, \h{20pt}\text{where for  $j = 1, 2, 3$, $q_j = q_j(\rho, z)$ \h{2.5pt} are real-valued functions.}
\end{eqnarray} By applying the change of variables: \begin{eqnarray} a^{-1} \mathscr{Q}\left(R x \right) = \overline{\mathscr{Q}}\left[ u \right] := u_1 \left( e_{\rho} \otimes e_{\rho} - \dfrac{1}{2} \h{0.5pt}\mathbf{I}_2 \right) + \dfrac{\sqrt{3}}{2} \h{1pt}u_2\h{1pt} M_2 + \dfrac{1}{2} \h{0.5pt}u_3 \h{1pt} M_3,  \h{10pt}\text{where $\mathbf{I}_2 = \left(\begin{array}{ccc}
1 & 0 & 0 \\
0 & 1 & 0 \\
0 & 0 & 0\\
\end{array}
\right), $}
\end{eqnarray}the order paramenter $\mathscr{Q}$ in (1.3) solves (1.2) if and only if $u = u(\rho, z) = (u_1, u_2, u_3)$ is a solution to   the system \begin{small}\begin{eqnarray}\left\{\begin{array}{lcl}
 -\Delta u_1+\dfrac{4}{\rho^2}u_1 \h{2pt}= a^2 R^2 u_1+\dfrac{3}{\sqrt{2}}\h{0.5pt}a R^2  \left(- 2 \h{0.5pt} u_1 u_2+\dfrac{\sqrt{3}}{2}u_3^2\right)- a^2 R^2 \h{0.5pt} |\h{0.5pt}u\h{0.5pt}|^2 \h{0.5pt}u_1, &&\text{in $\mathrm{B}_1$;}\vspace{0.5pc}\\
 -\Delta u_2\h{33pt}= a^2R^2 u_2+ \dfrac{3}{\sqrt{2}} \h{0.5pt}aR^2 \left(- u_1^2+ u_2^2+\dfrac{1}{2} \h{0.5pt}u_3^2\right)\h{4pt}- a^2 R^2 \h{0.5pt}|\h{0.5pt}u\h{0.5pt}|^2 \h{0.5pt}u_2, &&\text{in $\mathrm{B}_1$;}\vspace{0.5pc}\\
 -\Delta u_3+\dfrac{1}{\rho^2}u_3\h{2pt}= a^2 R^2 u_3+ \dfrac{3}{\sqrt{2}} \h{0.5pt} a R^2 \h{0.5pt}    \bigg(  \sqrt{3 } \h{1pt}u_1u_3+  u_2 u_3\bigg)-  a^2 R^2\h{0.5pt}|\h{0.5pt}u\h{0.5pt}|^2 \h{0.5pt}u_3, &&\text{in $\mathrm{B}_1$.}\vspace{0.5pc}\\
\end{array}\right.
\end{eqnarray}\end{small}\noindent Equivalently $u$ solves the   Euler-Lagrange equation of \begin{eqnarray*}E_{a, \h{0.5pt} \mu} \left[ \h{0.5pt}v \h{0.5pt}\right] := \int_{B_1}  \big|\h{0.5pt} \nabla v \h{1pt}\big|^2 + \frac{1}{\rho^2} \left( 4 v_1^2 + v_3^2 \h{0.5pt} \right) + \dfrac{\mu}{a}\h{1pt} F_a(\h{0.5pt}v\h{0.5pt}).
\end{eqnarray*}Here $v_j$ is the $j$-th component of $v$. $\mu$ equals to $a R^2$ which is assumed to be a fixed constant in the following course. The nonlinear potential function $F_a$ is read as   \begin{eqnarray}F_a(\h{0.5pt}v\h{0.5pt}) := D_a  - 3 \sqrt{2} \h{1pt} a \h{1pt} P\left(\h{1pt}v\h{1pt}\right) + \dfrac{a^2}{2} \left[ \h{1pt}|\h{1pt} v \h{1pt}|^2 - 1 \h{1pt}\right]^2 \h{5pt}\text{with $P\left(\h{1pt}v\h{1pt}\right) := -v_2 \h{0.5pt}v_1^2 + \frac{\sqrt{3}}{2} \h{0.5pt} v_1 \h{0.5pt} v_3^2 + \frac{1}{3} \h{0.5pt} v_2^3 + \frac{1}{2} \h{0.5pt} v_2 \h{0.5pt} v_3^2$.}
\end{eqnarray}Note that we choose $D_a$ to be a constant depending on $a$ so that $0$ is the absolute minimum value of $F_a$. Particular interests in both mathematics and physics are focused on the case when (1.2)  is assigned with a radial hedgehog boundary condition on $\p B_R$. Therefore in this article  we supply the system (1.5) with the Dirichlet boundary condition:  \begin{eqnarray}u =  \dfrac{\sqrt{2} H_+}{a}\h{1pt} U^*  \h{20pt}\text{on $\p \mathrm{B}_1$, \h{4pt}where $H_+ := \frac{3 + \sqrt{9 + 8 a^2} }{4}$\h{3pt}\text{and}\h{3pt} $U^* := \left( \dfrac{\sqrt{3}}{2}  \h{1pt} \dfrac{\rho^2 }{r^2}\h{1pt}, \h{3pt}\dfrac{3}{2} \h{1pt}\left( \dfrac{z^2}{r^2} - \dfrac{1}{3} \right), \h{3pt}\sqrt{3} \h{2pt} \dfrac{\rho \h{0.5pt} z }{r^2} \right).$}\end{eqnarray}
It can be shown that in the low-temperature limit  $a \rightarrow \infty$ (possibly up to a subsequence), minimizers of $E_{a, \h{0.5pt}\mu}$ within the space $\mathscr{F}_{a, \mu} := \Big\{ u = u(\rho, z) \in \mathbb{R}^3: E_{a, \mu}[\h{1pt}u\h{1pt}] < \infty \h{3pt}\text{and}\h{3pt} \text{$u$ satisfies (1.7)}\Big\}$ converges strongly in $H^1(B_1)$ to a minimizer of \begin{eqnarray} E_{\mu}\left[\h{0.5pt}u\h{0.5pt}\right] := E\left[\h{0.5pt}u\h{0.5pt}\right] + \sqrt{2} \h{1pt}\mu \int_{B_1} 1 - 3 P\left( u \right) := \int_{B_1} \big|\h{0.5pt}\nabla u\h{1.5pt}\big|^2 + \dfrac{1}{\rho^2} \left( 4 u_1^2 + u_3^2 \h{1.5pt}\right) + \sqrt{2}\h{1pt}\mu\h{1.5pt}\big(\h{0.5pt} 1 - 3 \h{1pt}P\big(\h{0.5pt}u\h{0.5pt}\big)\h{0.5pt} \big)
\end{eqnarray}in $\mathscr{F}$. Here  $\mathscr{F}$ is the configuration space \begin{eqnarray} \mathscr{F} := \Big\{ v = v(\rho, z) \in \mathbb{S}^2 : E_{\mu}[\h{1pt}v\h{1pt}] < \infty \h{3pt}\text{and}\h{3pt} \text{$v = U^*$  on $\p B_1$}\Big\}.\end{eqnarray}In other words $E_{a, \mu}$-energy defined on $\mathscr{F}_{a, \mu}$ space $\Gamma$-converges to the $E_\mu$-energy defined on $\mathscr{F}$ space. In the remaining of this article, $U^*$ is referred as hedgehog map. The energy functional $E_\mu$ is called energy functional of limiting Landau-de Gennes theory. We will study in details critical points of the $E_\mu$-energy and their core structures near disclinations. 
\vspace{1.5pc}\\
\textbf{\large I.2. \small SOME DEFINITIONS} \vspace{1pc}\\
Before stating our main results, let us introduce some definitions. \begin{definition}Suppose that $u = (u_1, u_2, u_3)$ is a $3$-vector field on $B_1$. $u$ is called $\mathscr{R}$-axially symmetric if for all $(\rho, z)$ with $\rho^2 + z^2 \leq 1$, it satisfies  $\mathrm{(i).} \h{2pt}u = u(\rho, z)$; $\mathrm{(ii).} \h{2pt} u_1(\rho, z) = u_1(\rho, - z)$; $\mathrm{(iii).} \h{2pt} u_2(\rho, z) = u_2(\rho, - z)$; $\mathrm{(iv).} \h{2pt} u_3(\rho, z) = - u_3(\rho, - z).$  Let $L$ be the augment operator given by \begin{eqnarray}L\left[\h{1pt}u \h{1pt}\right] := \big(\h{1pt}u_1 \cos 2 \theta, \h{2pt}u_1 \sin 2\theta, \h{2pt}u_2, \h{2pt}u_3 \cos \theta, \h{2pt}u_3 \sin \theta\h{1pt}\big).\end{eqnarray} A $5$-vector field on $B_1$  is  called $\mathscr{R}$-axially symmetric if it equals to  $L\left[ u \right]$ for some $\mathscr{R}$-axially symmetric $3$-vector field $u$ on $B_1$. Moreover we denote by $\mathscr{F}^{s}$ the subspace of $\mathscr{F}$ which contains all $\mathscr{R}$-axially symmetric maps in $\mathscr{F}$.
\end{definition}Given a $\mathscr{R}$-axially symmetric vector, we use the following definition to describe its regularity at a point:  \begin{definition} Let $u$ be a vector field on $\overline{B_1}$. $u$ is called regular at a point in $\overline{B_1}$ if there is a neighborhood of this point so that $\nabla u$ is essentially bounded in this neighborhood. This point is also called a regular point of $u$. If a point in $\overline{B_1}$ is not a regular point of $u$, then it is called singularity of $u$.
\end{definition}
Eigenvalues and eigenvectors of the order parameter $\mathscr{Q}$ play key roles in the study of Landau-de Gennes theory. In light of our ansatz (1.4), the eigenvalues of $\mathscr{Q}$ can be represented in terms of the variable $u$. Therefore we introduce \begin{definition}Given a $3$-vector field  $u = (u_1, u_2, u_3)$ on $B_1$, the eigenvalues  of $\overline{\mathscr{Q}}\left[\h{0.5pt}u\h{0.5pt}\right]$ in (1.4) equal to \begin{small}\begin{eqnarray*}  - \dfrac{1}{2} \left( u_1 + \dfrac{1}{\sqrt{3}} \h{0.5pt}u_2 \right); \h{5pt} \dfrac{1}{4} \bigg\{ \left( u_1 + \dfrac{1}{\sqrt{3}} \h{1pt}u_2 \right) - \sqrt{\big( u_1 - \sqrt{3} \h{0.5pt}u_2 \big)^2 + 4 u_3^2 } \h{1pt} \bigg\}; \h{5pt} \dfrac{1}{4} \bigg\{ \left( u_1 + \dfrac{1}{\sqrt{3}} \h{1pt}u_2 \right) + \sqrt{\big( u_1 - \sqrt{3} \h{0.5pt}u_2 \big)^2 + 4 u_3^2 } \h{1pt} \bigg\}.
\end{eqnarray*}\end{small}These three eigenvalues are denoted by $\lambda_1$, $\lambda_2$, $\lambda_3$, respectively and will be referred as eigenvalues determined by $u$ in the following. At a location $x \in B_1$, we call \begin{enumerate}[(1).]
\item $u$ is isotropic at $x$ if $x$ is a point singularity of $u$ (see Definition 1.2);
\item $u$ is uniaxial  at $x$ if two of the three eigenvalues determined by $u$ are identical and different from zero at $x$;
\item $u$ is biaxial at $x$ if all the three eigenvalues determined by $u$ are different at $x$.\end{enumerate}
Moreover a vector field is called director field determined by $u$ if its values are normalized eigenvectors corresponding to the largest eigenvalue in $\lambda_1$, $\lambda_2$, $\lambda_3$.
\end{definition} With the above definitions, we can define the so-called biaxial torus and split-core line segment. \begin{definition}Suppose that $\mathscr{T}$ is an open torus in $B_1$ and $u$ is a given $3$-vector field on $\mathscr{T}$. $\mathscr{T}$ is called biaxial torus with uniaxial core determined by $u$ if $u$ is uniaxial on a closed simple loop in $\mathscr{T}$. Meanwhile $u$ is biaxial in $\mathscr{T}$ except the points on the closed simple loop. In this case the torus $\mathscr{T}$ is also called a biaxial torus of the augmented map $L\left[u\right]$. Given a segment, we have a straight line, denoted by $l$, passing across the segment. The segment is called split-core segment associated with a $3$-vector field $u$ if there is a neighborhood, denoted by $O$, containing the segment so that $u$ is uniaxial  on $O \cap l$ except the two end-points of the segment. Meanwhile $u$ is isotropic at the two end-points of the segment and biaxial on $O \sim l$. Here and in what follows, we use $\sim$ to denote the set minus. In this case $L\left[ u\right]$ has split-core segment structure on $l$.
\end{definition}\
\\
\noindent \textbf{\large I.3. \small MAIN RESULTS AND STRUCTURE OF THE ARTICLE}\vspace{1pc}
\\
For any $\mu \geq 0$,  $U^* $ defined in (1.7) is a critical point of $E_\mu$. However it is not an $E_\mu$-energy minimizer within the configuration space $\mathscr{F}$ (see (1.9)).  For any smooth radial function $f$ with $f(1) = 0$, we let $\mathrm{P}_f$ be the orthogonal projection of $\big(\h{1pt}0, \h{0.5pt}f, \h{0.5pt}0 \h{1pt}\big)$ to the tangent plane of $\mathbb{S}^2$ at $U^*$. The Hessian of $E_{\mu}$ at $U^*$ then satisfies \begin{eqnarray*}\mathrm{Hess}_{U^*} \big[ \h{0.5pt}\mathrm{P}_f, \mathrm{P}_f \h{0.5pt}\big] =  \dfrac{32 \h{1pt}\pi}{5} \left[\h{1.5pt} \int_0^1 \big( f' \big)^2 r^2 - 3 \int_0^1 f^2 \h{1pt}\right] + \dfrac{72 \h{1pt}\sqrt{2}}{5} \h{1pt}\pi\h{1pt}\mu \h{1pt}\int_0^1 f^2 \h{1pt}r^2.
\end{eqnarray*} In light of Lemma 1.3 in [58], there is a smooth radial function $g$ compactly supported on $(0, 1)$ so that  \begin{eqnarray} \mathrm{Hess}_{\h{1pt}U^*} \big[ \h{1pt}\mathrm{P}_{g}, \mathrm{P}_{g} \h{0.5pt}\big] < 0. \end{eqnarray} Therefore $U^*$ is linearly unstable in $\mathscr{F}$.  There exists at least one map in $\mathscr{F}^s$ (see Definition 1.1) whose $E_\mu$-energy is strictly less than the $E_\mu$-energy of $ U^*$. We are now concerned about energy-minimizers of $E_\mu$ in $\mathscr{F}^s$. \begin{theorem} There exist at least two distinct energy minimizers of $E_\mu$ in $\mathscr{F}^s$. One minimizer has biaxial torus structure, while another minimizer has split-core segment structure on the $z$-axis. Moreover if we let $u_\mu$ be one energy minimizer of $E_\mu$ in $\mathscr{F}^s$, then as $\mu \rightarrow \infty$, $L\left[\h{1pt}u_\mu\h{1pt}\right]$ converges to $L\left[\h{1pt}U^*\h{1pt}\right]$ strongly in $H^1(B_1)$. Here $L$ is the augment operator defined in (1.10).
\end{theorem}

Our strategy to prove Theorem 1.5 relies on a novel bifurcation argument. In the following $B_1^+$ ($B_1^-$ resp.) denotes the open upper-half (lower-half resp.) part of $B_1$. $T$ is the flat boundary of $B_1^+$. Moreover we let $\mathrm{I}_- := \big(-1, \h{1pt}- 1\h{0.5pt}/\h{0.5pt}2\h{1.5pt}\big]$ and $\mathrm{I}_+ := \big[- 1\h{0.5pt}/\h{0.5pt}2, \h{1pt}1\h{1.5pt}\big) $. For any $b \in \mathrm{I}_-$ and $c \in \mathrm{I}_+$, we define $\mathscr{F}_{\h{1pt}b, +} := \Big\{ u \in \mathscr{F}^s : u_2 \h{2pt}\geq\h{2pt} b \h{6pt}\text{$\mathrm{on} \h{3pt}T$} \h{1pt} \Big\}$  and $\mathscr{F}_{\h{1pt}c, -} := \Big\{ \h{1pt}u \in \mathscr{F}^s : u_2 \h{2pt}\leq\h{2pt} c \h{6pt}\text{$\mathrm{on} \h{3pt} T$} \h{1pt} \Big\}$. Note that for $g$ satisfying  (1.11), $| \h{1pt}g \h{1pt}|$ also satisfies (1.11). For real number $\sigma$ with sufficiently small absolute value, the normalized map of  $U^* + \sigma \h{1pt}\mathrm{P}_{|\h{0.5pt} g \h{0.5pt}| }$, denoted by $U_{\sigma}^*$, has strictly less $E_{\mu}$-energy than  $U^*$. Since $U_{\sigma}^* \h{1pt}\in\h{1pt} \mathscr{F}_{\h{1pt}b, +}$ if $\sigma \geq 0$ and $U_{\sigma}^* \in \mathscr{F}_{\h{1pt}c, -}$ if $\sigma \leq 0$, it then follows  \begin{eqnarray} \mathrm{Min} \h{2pt} \bigg\{ \h{1pt}E_{\mu}\left[\h{1pt}u\right] : u \in \mathscr{F}_{\h{1pt}b, +} \bigg\} \h{2pt}< \h{2pt}E_{\mu}\h{1pt}\big[\h{1pt}U^*\h{1pt}\big], \h{30pt}\mathrm{Min} \h{2pt} \bigg\{ \h{1pt} E_{\mu}\left[\h{1pt}u\right] : u \in \mathscr{F}_{\h{1pt}c, -} \h{1pt} \bigg\} \h{2pt}< \h{2pt}E_{\mu}\h{1pt}\big[\h{1pt}U^*\h{1pt}\big],
\end{eqnarray}for any $b \in \mathrm{I}_-$ and $c \in \mathrm{I}_+$. By direct method of calculus of variations, for any $b\in \mathrm{I}_-$, the first minimization problem in (1.12) can be solved by a minimizer $u_b^+$. In light of $P$ defined in (1.6),  we have $ P\big(\h{0.5pt} v_1, v_2, v_3 \h{0.5pt}\big) \h{2pt}\leq\h{2pt}P\big(\h{0.5pt} |\h{0.5pt}v_1 \h{0.5pt}|, v_2, |\h{0.5pt}v_3\h{0.5pt}| \h{0.5pt}\big)$ for any $(v_1, v_2, v_3) \in \mathbb{R}^3$. Hence we can make the first and third components of $u_b^+$ non-negative on $B_1^+$. Similarly for any $c \in \mathrm{I}_+$, we can  find a minimizer, denoted by $u_{\h{1pt}c}^-$, of the second minimization problem in (1.12). Moreover the first and third components of $u_{\h{1pt}c}^-$ are non-negative throughout $B_1^+$. By (1.12), $u_b^+$ and $u_c^-$ solve some Signorini-type problems for $\mathbb{S}^2$-valued mappings (see [12] and [53] for scalar case).  \vspace{0.3pc}

Regularity of the two minimizers (i.e. $u_b^+$ and $u_c^-$) are crucial while we study biaxial torus and split-core segment structures. By the augment operator $L$, we reduce the regularity problem for $u_b^+$ and $u_c^-$ to the associated regularity problem of $L\h{1pt}\big[\h{1pt} u_b^+ \h{1pt}\big]$ and $L\h{1pt}\big[\h{1pt} u_c^-\h{1pt}\big]$, respectively. In the remaining $L\h{1pt}\big[\h{1pt} u_b^+ \h{1pt}\big]$ and $L\h{1pt}\big[\h{1pt} u_c^-\h{1pt}\big]$ are referred as bifurcation solutions. To study regularity of these two  bifurcation solutions, we split $\overline{B_1}$ into several sub-regions. By H\'{e}len's trick and arguments of Rivi$\grave{\text{e}}$re-Struwe in [54] (see also [49]), the bifurcation solutions are smooth in $B_1 \sim \left( l_3 \cup T\right)$. Here $l_3$ denotes the $z$-axis.  Moreover by [49], these two bifurcation solutions  are continuous up to the points on $\p B_1 \sim \left(l_3 \cup T\right)$. As for the south and north poles, we can flatten $\p B_1$ near these two poles and extend symmetrically the definition of these bifurcation solutions in the new coordinates from a half ball to a whole ball. By  Rivi$\grave{\text{e}}$re-Struwe arguments  in [54] and Schoen-Unlenbeck's geometric Lemma 2.5 in [57], these two bifurcation solutions are continuous up to these two poles. Hence we only need to study regularity of bifurcation solutions on $B_1 \cap l_3$ and the free boundary $T$. In Sect.II, we consider a Signorini-type problem for harmonic maps with potential.  We show that the bifurcation solutions are smooth on $T^{\circ} := T \sim \Big\{ (0, 0), (1, 0)\Big\}$. Without ambiguity we still use $T$ to denote the set $\Big\{ \left(\rho, 0 \right) : \rho \in \big[\h{0.5pt}0, 1 \h{0.5pt}\big]\h{1pt} \Big\}$. By this regularity result, the bifurcation solutions solve weakly the following Dirichlet boundary value problem:\begin{eqnarray}\left\{\begin{array}{lcl} - \Delta w - \dfrac{3 \sqrt{2}}{2} \h{1pt}\mu\h{1pt}\nabla_w S = \left\{ \h{1pt} \h{1pt}\big|\h{1pt}\nabla w \h{1pt}\big|^2 - \dfrac{9 \sqrt{2}}{2}\h{1pt}\mu\h{1pt}S \left[w\right] \h{1pt}\right\} w \h{20pt}&&\text{in $B_1$;}\vspace{0.3pc}\\
\h{78pt}w = L\left[ U^*\right] &&\text{on $\p B_1$.} \end{array}\right.
\end{eqnarray}Here $S\h{0.5pt}[\cdot]$ is defined by  \begin{eqnarray*} S\left[\h{1pt} w \h{1pt}\right] := - w_3 \left( w_1^2 + w_2^2 \h{0.5pt}\right) +  \sqrt{3} \h{1.5pt}w_2 \h{1pt}w_4 \h{1pt}w_5 + \dfrac{1}{2} \h{1pt}w_3\h{1pt}\left( w_4^2 + w_5^2 \h{0.5pt}\right) + \dfrac{1}{3} \h{1pt}w_3^3 + \dfrac{\sqrt{3}}{2} \h{1pt}w_1 \h{1pt}\left( w_4^2 - w_5^2 \h{0.5pt}\right).
\end{eqnarray*} Results in [49] therefore infer the continuity of the bifurcation solutions up to $\p B_1 \cap T$. In addition an $\epsilon$-regularity result, i.e. Proposition 3.1, also holds for the bifurcation solutions at the origin. In Sect.III, we show that the bifurcation solutions are indeed smooth at the origin. Hence if the bifurcation solutions have singularities, then singularities must be on $l_3$ and different from the origin and two poles. Sect.IV is devoted to studying asymptotic behaviour of bifurcation solutions near their singularities on $l_3$. With the preparations in Sections II-IV, in Sect.V, we study in details the biaxial torus and split-core segment structures of these bifurcation solutions. Our second main result is about specific properties of $L\h{1pt}\big[\h{1pt} u_b^+ \h{1pt}\big]$, particularly its biaxial torus structure. \begin{theorem}For any constant $b \in \mathrm{I}_-$, the followings hold for $u_{\h{1pt}b}^+$ and its augmented map $L\h{1pt}\big[\h{1pt} u_b^+ \h{1pt}\big]$:
\begin{enumerate}[(1).]
\item The augmented map $L\h{1pt}\big[\h{0.5pt}u_{\h{1pt}b}^+\big]$ is a weak solution to the Dirichlet boundary value problem (1.13). It  is smooth on $B_1$, except possibly a finite number of singularities on $l_3$.  $L\h{1pt}\big[\h{0.5pt}u_{\h{1pt}b}^+\big]$ is smooth at the origin, south pole and north pole. If $L\h{1pt}\big[\h{1pt} u_b^+\h{1pt}\big]$ has singularities, then it must have even number of singularities on $l_{\h{1pt}3}^+ := B_1^+ \cap l_3$. By symmetry, $L\h{1pt}\big[\h{0.5pt}u_{\h{1pt}b}^+\big]$ has the same number of singularities on $l_3^- := B_1^- \cap l_3$ as the number of singularities that it has on $l_3^+$;
\item If $L\h{1pt}\big[\h{0.5pt}u_{\h{1pt}b}^+\h{0.5pt}\big]$ has singularities, then at each singularity on $l_{\h{1pt}3}^+$, tangent map of $L\h{1pt}\big[\h{0.5pt}u_{\h{1pt}b}^+\h{0.5pt}\big]$  must equal to  \begin{eqnarray} \text{ $\Lambda^+ := \big(\h{0.5pt}0, \h{2pt}0, \h{2pt}  \cos \varphi, \h{2pt}\sin \varphi  \cos \theta, \h{2pt}\sin \varphi  \sin \theta \h{0.5pt}\big)$ \h{2pt}or\h{2pt} $\Lambda^- := \big(\h{0.5pt}0, \h{2pt}0, \h{2pt}  - \cos \varphi, \h{2pt}\sin \varphi \cos \theta, \h{2pt}\sin \varphi \sin \theta \h{0.5pt}\big)$.}\end{eqnarray} The tangent map of $L\h{1pt}\big[\h{0.5pt}u_{\h{1pt}b}^+\h{0.5pt}\big]$ at its lowest singularity on $l_{\h{1pt}3}^+$ is given by $\Lambda^-$ in (1.14). Moreover on $l_3^+$, tangent maps of two consecutive singularities of $L\h{1pt}\big[\h{0.5pt}u_{\h{1pt}b}^+\h{0.5pt}\big]$ must be different maps in (1.14);
\item In the sense of Definition 1.2, $u_{\h{1pt}b}^+$ is regular away from singularities of $L\h{1pt}\big[\h{0.5pt}u_{\h{1pt}b}^+\h{0.5pt}\big]$. The singularities of $L\h{1pt}\big[\h{0.5pt}u_{\h{1pt}b}^+\h{0.5pt}\big]$, if it has, must also be singularities of $u_{\h{1pt}b}^+$.  If we understand $u_{\h{1pt}b}^+$ as a map defined on the $\rho$\h{0.5pt}-\h{0.5pt}$z$ plane, then $u_{\h{1pt}b}^+$ is continuous on $T$ and analytic on its interior part $T^\circ$. $u_{\h{1pt}b}^+$ takes the value $\left(\sqrt{3}\big/2,  1\big/2, 0\right)$ at finitely many points on $T^\circ$. At these points, $u_{\h{1pt}b}^+$ is uniaxial;
\item There is a $\rho_0 \in (0, 1)$ so that for an $\epsilon > 0$ sufficiently small, the second component of $u_{\h{1pt}b}^+$ satisfies $ u_{\h{1pt}b; \h{1pt}2}^+ > 1 \big/ 2$  on $\Big\{ \left(\rho, 0 \right) : \rho \in \left( \rho_0 - \epsilon, \rho_0 \right) \Big\}$ and $u_{\h{1pt}b;\h{1pt} 2}^+ < 1\big/2$  on $\Big\{ \left(\rho, 0 \right) : \rho \in \left( \rho_0 , \rho_0 + \epsilon \right) \Big\}$. Here and in what follows  $u_{\h{1pt}b; \h{0.5pt}j}^+$ denotes the $j$-th component of  $u_{\h{1pt}b}^+$. Suppose that $D_{\epsilon}\left(\h{0.5pt}\rho_0, 0\h{0.5pt}\right)$ is the disk in the $\rho$\h{0.5pt}-\h{0.5pt}$z$ plane with center $\big(\rho_0, 0\big)$ and radius $\epsilon$. Then $u_{\h{1pt}b}^+$ is biaxial on $D_{\epsilon}\left(\h{0.5pt}\rho_0, 0\h{0.5pt}\right) \sim \Big\{\big(\rho_0, 0\big)\Big\}$, provided $\epsilon$ is sufficiently small. More precisely the three eigenvalues in Definition 1.3 computed in terms of $u_{\h{1pt}b}^+$  satisfy the  quantitative relationship $ \lambda_3 \h{2pt}> \h{2pt}\lambda_2\h{2pt}>\h{2pt}\lambda_1$ on $D_{\epsilon}\big(\rho_0, 0 \big) \sim \Big\{\big(\rho_0, 0\big)\Big\}$;
\item  For any given $3$-vector field $u = (u_1, u_2, u_3)$, it determines a vector field \begin{eqnarray}\kappa\left[\h{0.5pt}u\h{0.5pt}\right] := \dfrac{\sqrt{2}}{2} \left\{ 1 + \dfrac{u_{1} - \sqrt{3}\h{1pt}u_{2}}{\sqrt{\left( u_{1} - \sqrt{3}\h{1pt}u_{2} \right)^2 + 4 u_{3}^2}} \right\} e_{\rho} + \dfrac{\sqrt{2}\h{1pt}u_{3}}{\sqrt{\left(u_{1} - \sqrt{3}\h{1pt}u_{2} \right)^2 + 4 u_{3}^2}} e_z.
\end{eqnarray} Then on $D_{\epsilon}\left(\h{0.5pt}\rho_0, 0\h{0.5pt}\right) \sim \Big\{\big(\rho_0, 0\big)\Big\}$, the normalized vector field of $\kappa\h{1pt}\big[\h{1pt} u_b^+\h{1pt}\big]$, denoted by $\kappa^\star$ for simplicity, is a director field determined by $u_b^+$ (see Definition 1.3).  It is the normalized eigenvector corresponding to $\lambda_3$ in Definition 1.3 computed in terms of $u_{\h{1pt}b}^+$.  $\kappa^{\star}$ is continuous on $D_{\epsilon}\left(\h{0.5pt}\rho_0, 0\h{0.5pt}\right) \sim \Big\{ \big(\rho, 0 \big) : \rho \in \big(\h{0.5pt}\rho_0 - \epsilon, \rho_0\h{0.5pt}\big] \h{1pt}\Big\}.$ However it is  not continuous up to the branch cut $\Big\{ \big(\rho, 0\big) : \rho \in \big(\h{0.5pt}\rho_0 - \epsilon, \rho_0\h{0.5pt}\big] \h{1pt}\Big\}$. Suppose that  $\epsilon'$ is an arbitrary number in $(0, \epsilon)$ and denote by $x'$ the point $\big(\rho_0 - \epsilon', 0\big)$. $\kappa^{\star}$ converges to $- e_z$ if we approach $x'$ from the lower-half part of $\p D_{\epsilon'}\left(\rho_0, 0 \right)$. Meanwhile $\kappa^{\star}$ converges to  $e_z$  if we approach $x'$ from the upper-half part of $\p D_{\epsilon'}\left(\rho_0, 0 \right)$. During the process when we rotate counter-clockwisely from $x'$ to itself along $\p D_{\epsilon'}\left(\rho_0, 0\right)$, the image of $\kappa^{\star}$ varies continuously from $- e_z$ to $e_z$ and keeps on the right-half part of $\rho$\h{0.5pt}-\h{0.5pt}$z$ plane for all points on $\p D_{\epsilon'}\left(\rho_0, 0 \right) \sim \big\{ x'\big\}$. The angle of $\kappa^{\star}$ is changed by $\pi$ during this process;
\item Let $\varphi^{\star}$ be an angular variable ranging from $\big[ - \pi, \pi \h{0.5pt}\big]$. The value of $u_{\h{1pt}b}^+ \left( \left(\rho_0, 0\right) + \epsilon' \left( \cos \varphi^{\star}, \sin \varphi^{\star} \right) \right)$  at $- \pi$ ($\pi$ resp.) is taken to be the limit of $u_{\h{1pt}b}^+$ at $x'$ along the lower (upper resp.) part of $\p D_{\epsilon'}\left(\rho_0, 0\right)$. Then the tangent map of $\kappa^{\star}$, i.e. the limit of $\kappa^{\star}\left( \left(\rho_0, 0\right) + \epsilon' \left( \cos \varphi^{\star}, \sin \varphi^{\star} \right) \right)$ as $ \epsilon' \rightarrow 0^+$, equals to
\begin{small}\begin{eqnarray*}  \left\{\begin{array}{lcl} - e_z, \h{20pt}&&\text{if $\varphi^{\star} = - \pi$;} \vspace{0.5pc}\\
\dfrac{\sqrt{2}}{2} \left\{ 1 - \dfrac{\varkappa\h{1pt} \mathrm{ctan} \h{0.5pt} \varphi^{\star}}{\sqrt{4 + \varkappa^2 \h{1pt} \mathrm{ctan}^2 \varphi^{\star}}} \right\}^{1/2} e_{\rho} - \sqrt{\dfrac{2}{ 4 + \varkappa^2 \h{1pt} \mathrm{ctan}^2 \varphi^{\star}}}\left\{ 1 - \dfrac{\varkappa\h{1pt} \mathrm{ctan} \h{0.5pt} \varphi^{\star}}{\sqrt{4 + \varkappa^2 \h{1pt} \mathrm{ctan}^2 \varphi^{\star}}} \right\}^{- 1/2} e_z, &&\text{if $\varphi^{\star} \in (- \pi, 0)$;}\vspace{0.5pc}\\
e_{\rho}, &&\text{if $\varphi^{\star} = 0 $;}\vspace{0.5pc}\\
\dfrac{\sqrt{2}}{2} \left\{ 1 + \dfrac{\varkappa\h{1pt} \mathrm{ctan} \h{0.5pt} \varphi^{\star}}{\sqrt{4 + \varkappa^2 \h{1pt} \mathrm{ctan}^2 \varphi^{\star}}} \right\}^{1/2} e_{\rho} + \sqrt{\dfrac{2}{ 4 + \varkappa^2 \h{1pt} \mathrm{ctan}^2 \varphi^{\star}}}\left\{ 1 + \dfrac{\varkappa\h{1pt} \mathrm{ctan} \h{0.5pt} \varphi^{\star}}{\sqrt{4 + \varkappa^2 \h{1pt} \mathrm{ctan}^2 \varphi^{\star}}} \right\}^{- 1/2} e_z, &&\text{if $\varphi^{\star} \in (0, \pi)$;}\vspace{0.5pc}\\
e_z, &&\text{if $\varphi^{\star} = \pi$.}
\end{array}\right.
\end{eqnarray*}\end{small}\noindent Here $\varkappa$ is a non-negative constant given by $\varkappa = \left. \dfrac{\p_{\rho} \h{0.5pt} u_{\h{1pt}b; 1}^+ - \sqrt{3} \h{1.5pt} \p_{\rho} \h{0.5pt} u_{\h{1pt}b; \h{0.5pt}2}^+}{\p_z \h{0.5pt} u_{\h{1pt}b; 3}^+} \h{1pt}\right|_{\big(\rho_0, 0\big)}$. Note that by Hopf's lemma, $\p_z \h{0.5pt} u_{\h{1pt}b; \h{0.5pt}3}^+ $ is strictly positive at $\big(\rho_0, 0\big)$.
\end{enumerate}
\end{theorem}
\noindent As for $L\left[ u_c^-\right]$, it exhibits split-core segment structure, which we show in the following theorem: \begin{theorem}For any constant $c \in \mathrm{I}_+ \sim \big\{\h{0.5pt}0\h{0.5pt}\big\}$, the followings hold for $u_{\h{1pt}c}^-$ and its augmented map $L\h{1pt}\big[\h{0.5pt}u_{\h{1pt}c}^-\h{0.5pt}\big]$:
\begin{enumerate}[(1).]
\item The augmented map $L\h{1pt}\big[\h{0.5pt}u_{\h{1pt}c}^-\h{0.5pt}\big]$ is a weak solution to the boundary value problem (1.13). It is smooth on $B_1$, except a finite number of singularities on $l_3$. $L\h{1pt}\big[\h{0.5pt}u_{\h{1pt}c}^-\h{0.5pt}\big]$ is smooth at the origin, south pole and north pole. It must have odd number of singularities on $l_{\h{1pt}3}^+$. By symmetry, $L\h{1pt}\big[\h{0.5pt}u_{\h{1pt}c}^-\big]$ has the same number of singularities on $l_3^-$ as the number of singularities that it has on $l_3^+$;
\item Part (2) of Theorem 1.6 still holds for $L\h{1pt}\big[\h{0.5pt}u_{\h{1pt}c}^-\h{0.5pt}\big]$, except that the tangent map of $L\h{1pt}\big[\h{0.5pt}u_{\h{1pt}c}^-\h{0.5pt}\big]$ at its lowest singularity on $l_{\h{1pt}3}^+$ is given by $\Lambda^+$ in (1.14);
\item In ths sense of Definition 1.2, $u_{\h{1pt}c}^-$ is regular away from the singularities of $L\h{1pt}\big[\h{0.5pt}u_{\h{1pt}c}^-\h{0.5pt}\big]$. Singularities of $L\h{1pt}\big[\h{0.5pt}u_{\h{1pt}c}^-\h{0.5pt}\big]$ must also be singularities of $u_{\h{1pt}c}^-$. Moreover on $\overline{B_1} \cap l_3$, $u_{\h{1pt}c}^-$ is uniaxial except at the singularities. Suppose that $x_0^+$ is the lowest singularity on $l_{\h{1pt}3}^+$. $x_0^-$ is the symmetric point of $x_0^+$ with respect to the $x_1$-$x_2$ plane. For two sufficiently small positive constants $\epsilon$ and $\epsilon_1$, we denote by $\mathrm{D}_{\epsilon, \epsilon_1}$ the dumbbell on $x_1$-$z$ plane introduced in Definition 5.1. Then $\mathrm{D}_{\epsilon, \epsilon_1}$ contains properly the segment connecting $x_0^+$ and $x_0^-$. Moreover $u_{\h{1pt}c}^-$ is biaxial on $\mathrm{D}_{\epsilon, \epsilon_1} \sim l_3$. Quantitatively the three eigenvalues in Definition 1.3 computed in terms of $u_{\h{1pt}c}^-$ satisfy $\lambda_3 \h{1pt}>\h{1pt} \lambda_1 \h{1pt}> \h{1pt}\lambda_2$ on $\mathrm{D}_{\epsilon, \epsilon_1} \sim l_3$;
\item The normalized vector field of $\kappa \h{1pt}\big[\h{1pt} u_c^- \h{1pt}\big]$, which we denote by $\kappa_\star$, is a director field determined by $u_c^-$ on $\mathrm{D}_{\epsilon, \epsilon_1} \sim l_3$. Suppose that $\mathscr{C}$ is the boundary of $\mathrm{D}_{\epsilon, \epsilon_1}$ in the $x_1$-$z$ plane. $\mathscr{C}^+$ is the part of $\mathscr{C}$ with non-negative $x_1$-coordinate. From the lowest point on $\mathscr{C}$ to the most top point on $\mathscr{C}$ along $\mathscr{C}^+$, the image of $\kappa_{\star}$ varies continuously from $- e_z$ to $e_{\rho}$ and then to $e_z$. Meanwhile the image of $\kappa_\star$ keeps on the right-half part of the $(\rho, z)$ plane for all points on $\mathscr{C}^+ \sim l_3$;
\item In limiting scenario, the tangent map of $\kappa_{\star}$ at $x_0^+$ is given by \begin{eqnarray*}\left\{ \begin{array}{lcl} e_z, \h{30pt}&&\text{if $\varphi_+^{\star} = 0$;}\vspace{0.5pc}\\
\dfrac{\sqrt{2}}{2} \left\{ 1 - \dfrac{ \sqrt{3}\h{1pt}\cos \varphi_+^{\star}}{\sqrt{3 + \sin^2 \varphi_+^{\star}}} \right\}^{1/2} e_{\rho} + \sqrt{\dfrac{2\h{1pt}\sin^2 \varphi_+^{\star}}{3 + \sin^2 \varphi_+^{\star}}}\left\{ 1 - \dfrac{ \sqrt{3}\h{1pt}\cos \varphi_+^{\star}}{\sqrt{3 + \sin^2 \varphi_+^{\star}}} \right\}^{- 1/2} e_z, && \text{if $\varphi_+^{\star} \in \big(0, \pi\big]$.}\end{array}\right.
\end{eqnarray*}Here  $\varphi_+^{\star}$ is the polar angle of the spherical coordinates with respect to the center $x_0^+$. Along the segment connecting $x_0^+$ and $x_0^-$ (not include the two end-points), $\kappa_{\star}$ equivalently equals to $e_{\rho}$. The tangent map of $\kappa_{\star}$ at $x_0^-$ is given by \begin{eqnarray*}\left\{ \begin{array}{lcl}
\dfrac{\sqrt{2}}{2} \left\{ 1 + \dfrac{ \sqrt{3}\h{1pt}\cos \varphi_-^{\star}}{\sqrt{3 + \sin^2 \varphi_-^{\star}}} \right\}^{1/2} e_{\rho} - \sqrt{\dfrac{2\h{1pt}\sin^2 \varphi_-^{\star}}{3 + \sin^2 \varphi_-^{\star}}}\left\{ 1 + \dfrac{ \sqrt{3}\h{1pt}\cos \varphi_-^{\star}}{\sqrt{3 + \sin^2 \varphi_-^{\star}}} \right\}^{- 1/2} e_z,  \h{5pt}&&\text{if $\varphi_-^{\star} \in \big[\h{0.5pt}0, \pi \h{0.5pt}\big)$;} \vspace{0.5pc}\\
 - e_z, &&\text{if $\varphi_-^{\star} = \pi$.}\end{array}\right.
\end{eqnarray*}Here $\varphi_-^{\star}$ is the polar angle of the spherical coordinates with respect to the center $x_0^-$.
\end{enumerate}
\end{theorem}
In the next step we prove existence result in Theorem 1.5. Note that for any $b \in \mathrm{I}_-$, the energy $E_\mu\h{1pt}\big[\h{1pt} u_b^+ \h{1pt}\big]$ is non-decreasing with respect to $b$. For any $c \in \mathrm{I}_+$, the energy $E_\mu\h{1pt}\big[\h{1pt} u_c^-\h{1pt}\big]$ is non-increasing with respect to $c$. Hence we can use a limiting process to drop the Signorini contraint from the two minimization problems in (1.12). In Sect.VI.1, we show \begin{theorem}As $b \rightarrow - 1$ and $c \rightarrow 1$, $L\h{1pt}\big[ \h{0.5pt}u_b^+ \h{0.5pt}\big]$ and $L\h{1pt}\big[\h{0.5pt} u_c^- \h{0.5pt}\big]$ converge strongly in $H^1(B_1)$ to some $w^+$ and $w^-$, respectively.  It satisfies $w^+ = L\h{1pt} \left[\h{0.5pt} u_{b_\star}^+\h{0.5pt}\right]$ for some $b_\star \in \mathrm{I}_-$ and  $w^- = L\h{1pt} \left[\h{0.5pt} u_{c_\star}^- \h{0.5pt}\right]$ for some $c_\star \in \mathrm{I}_+$. Moreover we have \begin{eqnarray*} \lim_{b \h{1pt}\rightarrow \h{1pt}- 1} E_\mu\left[ u_b^+\right] = E_\mu\left[ u_{b_\star}^+\right] = \mathrm{Min}  \h{2pt} \bigg\{ \h{1pt}E_{\mu}\left[\h{1pt}u\right] : u \in \mathscr{F}^s \bigg\}, \h{20pt}\lim_{c \h{1pt} \rightarrow \h{1pt}1} E_\mu\left[ u_c^-\right] = E_\mu\left[ u_{c_\star}^-\right] = \mathrm{Min}  \h{2pt} \bigg\{ \h{1pt}E_{\mu}\left[\h{1pt}u\right] : u \in \mathscr{F}^s \bigg\}.
\end{eqnarray*}
\end{theorem}\noindent Sect.VI.2 is devoted to proving the convergence result in Theorem 1.5. \\

In this article we focus on the energy functional $E_\mu$ of the limiting Landau-de Gennes theory. In our forthcoming paper, we will consider the biaxial and split-core structure of solutions to the original Landau-de Gennes system (1.5).
\vspace{1.5pc}\\
\setcounter{section}{2}
\setcounter{theorem}{0}
\setcounter{equation}{0}
\textbf{\large II. \small  REGULARITY ON THE INTERIOR OF FREE BOUNDARY}\vspace{1pc}

\noindent   Throughout the section  $\mathbb{D} := \Big\{\h{0.5pt} (\rho, z) :  z \in (-1, 1), \h{3pt}  \rho \in \left(0, \sqrt{1 - z^2 } \h{1pt} \right)\Big\}$ and $\mathbb{D}^+ := \mathbb{D} \cap \Big\{\h{0.5pt} (\rho, z) : z > 0 \h{1pt}\Big\}$. The gradient operator on the $\rho$\h{0.5pt}-\h{0.5pt}$z$ plane is denoted by $D = \big( \p_{\rho}, \p_z \big)$. We let $D_r\left(x\right)$  represent  the open disk in the $\left(\rho, z\right)$-plane with center $x$ and radius $r$. $D_r^+\left(x\right)$ is the subset of $D_r\left(x\right)$ above the $\rho$\h{0.5pt}-\h{0.5pt}axis.  Under the $\mathscr{R}$-axial symmetry, we can rewrite the energy functional $E_\mu$ in (1.8) as follows: \begin{eqnarray*} E_{\mu}\left[\h{0.5pt}u\h{0.5pt}\right] = 2 \pi E_{\h{1pt}\mathbb{D}}\left[ \h{0.5pt}u \h{0.5pt} \right] := 2 \pi \int_{\mathbb{D}} \left\{ \h{1.5pt}\big|\h{1pt} \p_{\rho} u \h{1pt}\big|^2 + \big|\h{1pt}\p_z u \h{1pt}\big|^2 + \dfrac{1}{\rho^2} \big(\h{1pt} 4 u_1^2 + u_3^2 \h{1pt}\big) + \sqrt{2}\h{1pt} \mu \h{1pt}\big( \h{1pt} 1 - 3 \h{1pt} P\h{0.5pt}\big(\h{1pt} u \h{1pt}\big) \h{1pt} \big)  \h{1.5pt}\right\} \rho \h{2pt}\mathrm{d}\rho \h{2pt}\mathrm{d} z.
\end{eqnarray*}The Euler-Lagrange equation of $E_{\h{1pt}\mathbb{D}}$ can be read as \begin{eqnarray} - \dfrac{1}{\rho} \h{1pt}D \cdot  \big( \rho \h{1pt}   D u \big) + \h{2pt} \dfrac{1}{\rho^2} \h{1pt} \left( \begin{array}{lcl} 4 u_1 \\
\h{4pt}0 \\
\h{3pt}u_3 \end{array} \right) - \dfrac{3 \sqrt{2}}{2}\h{1pt}\mu\h{1pt}\nabla_u P  =  \left\{ \h{2pt} \h{1pt}\big|\h{1pt}D u\h{1pt}\big|^2 + \dfrac{1}{\rho^2} \left( 4 u_1^2 + u_3^2 \right) - \dfrac{9 \sqrt{2}}{2}\h{1pt}\mu\h{1pt}P\h{0.5pt}\big(\h{1pt}u \h{1pt}\big) \h{2pt}\right\} u \h{15pt}\text{in $\mathbb{D}$.}
\end{eqnarray} Note that the second components of $u_{\h{1pt}b}^+$ and $u_{\h{1pt}c}^-$ satisfy the obstacle conditions: \begin{eqnarray}u_{\h{1pt}b; \h{0.5pt} 2}^+ \h{2pt}\geq \h{2pt}b \h{20pt}\text{and}\h{20pt}u_{\h{1pt}c; \h{0.5pt}2}^- \h{2pt}\leq\h{2pt}c \h{10pt}\text{on $T$, \h{3pt}respectively.}
  \end{eqnarray}By the $\mathscr{R}$-axial symmetry and (2.2), the images of $u_{\h{1pt}b}^+$ and $u_{\h{1pt}c}^-$ on $T$ are supported on two circular arcs. Moreover these two arcs are properly contained in the equator of $\mathbb{S}^2$. Different from [30], the supporting manifold in our problem has non-empty boundary.  Our problem is a Signorini-type free boundary problem for $\mathbb{S}^2$-valued mappings. The main result of this section is \begin{prop}For any $b \in \mathrm{I}_-$ and $c \in \mathrm{I}_+$ with $c \neq 0$, $u_{\h{1pt}b}^+$ and $u_{\h{1pt}c}^-$ are smooth on $T^\circ$. Moreover $u_{\h{1pt}b}^+$ and $u_{\h{1pt}c}^-$ are continuous up to $(1, 0)$ in the $(\rho, z)$-plane.
 \end{prop}
Due to (2.2), $u_{\h{1pt}b}^+$ and $u_{\h{1pt}c}^-$ do not a-priorily solve weakly the equation (2.1). Standard regularity theory for harmonic maps can not be applied directly. To tackle this difficulty, we need  \begin{lemma}Let $x_* = \left(\rho_*, 0\right)$ be an arbitrary point on $T^\circ$ and denote by $\sigma_*$ the number $\min \left\{ \dfrac{\rho_*}{4}, \dfrac{1 - \rho_*}{4}\right\}$. For any $\sigma \in \left(0, \sigma_*\right)$, we define \begin{eqnarray*} \mathscr{I}_{\sigma} := \bigg\{ \left(x, r\right) \in  D^+_{\sigma}\left(x_*\right) \times \mathbb{R}_+ :  D_r\left(x\right) \subset D_{\sigma}^+\left(x_*\right) \h{1pt}\bigg\} \h{3pt}\bigcup\h{3pt} \bigg\{ \left(x, r\right) \in T \times \mathbb{R}_+ : D_r^+\left(x\right) \subset D_{\sigma}^+\left(x_*\right) \h{1pt}\bigg\}.
  \end{eqnarray*}Moreover for any $D^+_r\left(x\right) \subset \mathbb{D}^+$, we let \begin{eqnarray*} E_{x, \h{1pt} r}^+\left[ \h{0.5pt}u \h{0.5pt} \right] := \int_{D_r^+\left(x\right)} \left\{ \h{1.5pt}\big|\h{1pt} D u \h{1pt}\big|^2  + \dfrac{1}{\rho^2} \big(\h{1pt} 4 u_1^2 + u_3^2 \h{1pt}\big) + \sqrt{2}\h{1pt} \mu \h{1pt}\big( \h{1pt} 1 - 3 \h{1pt} P\big(\h{1pt} u \h{1pt}\big) \h{1pt} \big)  \h{1.5pt}\right\} \rho \h{2pt}\mathrm{d} \rho \h{1pt}\mathrm{d} z.
\end{eqnarray*}

There exist two constants $\theta_1 \in \left(0, 1/2 \right)$ and $\epsilon > 0$ so that if $\sigma \in \left(\h{0.5pt}0, \sigma_*\h{0.5pt}\right)$ and satisfies $ E^+_{x_*, \h{1pt}\sigma} \h{1pt}\big[\h{0.5pt}u_{\h{1pt}b}^{+} \h{0.5pt} \big] < \epsilon$, then for any $\left(x, r\right) \in \mathscr{I}_{\sigma}$, either one or the other of the followings holds
     \begin{eqnarray} (1).  \h{5pt}E^+_{x, \h{1pt}\theta_1 r} \left[\h{0.5pt}u_{\h{1pt}b}^{+}\h{0.5pt}\right] \h{2pt}\leq\h{2pt} r^{3/2};\h{30pt} (2). \h{5pt} E^+_{x, \h{1pt}\theta_1 r} \left[\h{0.5pt}u_{\h{1pt}b}^{+}\h{0.5pt}\right] \h{2pt}\leq\h{2pt} \dfrac{1}{2} \h{1.5pt} E_{x, r}^+\left[\h{0.5pt}u_{\h{1pt}b}^{+}\h{0.5pt}\right].
\end{eqnarray}Here $\theta_1$ is an universal constant. The same result holds for $u_{\h{1pt}c}^-$.
 \end{lemma}
Firstly we prove Proposition 2.1 with Lemma 2.2.
\begin{proof}[\bf Proof of Proposition 2.1] For simplicity we use  $u$ to denote either $u_{\h{1pt}b}^+$ or $u_{\h{1pt}c}^-$. Letting $x_*$, $\sigma_*$ and $\epsilon$ be as in Lemma 2.2,  we choose $r_* > 0$ sufficiently small so that \begin{eqnarray} 16 \h{0.5pt} r_* < \sigma_* \h{10pt}\text{ and } \h{10pt} E_{x_*, \h{1pt}16\h{0.5pt}r_*}^+\left[\h{0.5pt}u\h{0.5pt}\right] < \epsilon. \end{eqnarray} Moreover we denote by $D^+_*$ the upper-half disk $D_{\h{1pt}\theta_1 r_*}^+\left(x_*\right)$, where $\theta_1$ is given in Lemma 2.2. Now we fix an arbitrary $x \in  \overbar{D^+_*}$. It can be shown that  \begin{eqnarray}D_r^+\left(x\right)\subset D_{8\h{0.5pt}r_*}^+\left( x \right) \subset D_{16 r_*}^+\left( x_* \right), \h{10pt}\text{ for any $r \in \big(\h{1pt}0, 8 \h{0.5pt}r_* \theta_1 \big]$.} \end{eqnarray}

Suppose $x \in T$. For any $r \in \big(\h{1pt}0, 8 \h{0.5pt} r_* \theta_1 \big]$ fixed, we let $k$ be the natural number so that \begin{eqnarray} 8 \h{1pt}r_*\h{1pt}\theta_1^{k + 1} \h{1pt}<\h{1pt}  r  \h{1pt}\leq\h{1pt} 8 \h{1pt}r_*\h{1pt}\theta_1^{k}.\end{eqnarray}Taking $\sigma = 16\h{1pt}r_*$ in Lemma 2.2, by (2.4)-(2.5), we can apply (2.3) to get  \begin{eqnarray} E^+_{x, \h{1pt}8\h{0.5pt}r_*\h{0.5pt}\theta_1^{l }} \left[\h{1pt}u\h{1pt}\right] \h{2pt}\leq\h{2pt}\left( 8\h{0.5pt}r_*\h{0.5pt}\theta_1^{l - 1}\right)^{3/2} \h{10pt}\text{or}\h{10pt}E^+_{x, \h{1pt} 8\h{0.5pt}r_*\h{0.5pt}\theta_1^{l}} \left[\h{1pt}u\h{1pt}\right] \h{2pt}\leq\h{2pt} \dfrac{1}{2} \h{1pt}E^+_{x, \h{1pt}8\h{0.5pt}r_*\h{0.5pt}\theta_1^{l - 1}} \left[\h{1pt}u\h{1pt}\right], \h{10pt}\text{ for any $l = 1, 2, ...$}
\end{eqnarray}If the second inequality in (2.7) holds for any $l = 1, ..., k$, then by (2.6),  it follows \begin{eqnarray*} E_{x, \h{1pt}r}^+\left[ \h{1pt}u\h{1pt}\right] \h{2pt}\leq\h{2pt} E^+_{x, \h{1pt}8  \h{0.5pt} r_* \h{0.5pt}\theta_1^{k }} \left[\h{1pt}u\h{1pt}\right] \h{2pt}\leq\h{2pt} \left(\dfrac{1}{2}\right)^k E^+_{x, \h{1pt}8 \h{0.5pt} r_*} \left[\h{1pt}u\h{1pt}\right] \h{2pt}\lesssim_{r_*}\h{2pt} r^{\alpha}, \h{10pt}\text{where $\alpha = - \mathrm{log}_{\h{0.5pt}\theta_1}2$.}
\end{eqnarray*}Otherwise we take $l_*$ to be the largest number in $\Big\{1, ..., k\Big\}$ so that the first inequality in (2.7) holds. Then \begin{eqnarray*}E_{x, \h{1pt}r}^+\left[ \h{1pt}u\h{1pt}\right] \h{2pt}\leq\h{2pt} E^+_{x, \h{1pt}8 \h{0.5pt} r_*\h{0.5pt}\theta_1^{k }} \left[\h{1pt}u\h{1pt}\right] \leq \left( \dfrac{1}{2} \right)^{k  - l_*} E^+_{x, \h{1pt}8 r_*\h{0.5pt}\theta_1^{l_*}} \left[\h{1pt}u\h{1pt}\right] \leq \left( \dfrac{1}{2} \right)^{k  - l_*} \left(8 \h{0.5pt} r_*\h{0.5pt}\theta_1^{l_* - 1}\right)^{3/2} \h{2pt}\lesssim\h{2pt}\left(\dfrac{1}{2}\right)^{k} \h{2pt}\lesssim_{r_*}\h{2pt}r^{\alpha}.
\end{eqnarray*}The above arguments  infer \begin{eqnarray}E_{x, \h{1pt}r}^+\left[\h{0.5pt}u\h{0.5pt}\right]\h{2pt}\lesssim_{\h{0.5pt}r_*}\h{2pt} r^{\alpha}, \h{10pt}\text{provided $x \in T \cap \overbar{D_*^+}$ and $r \in \big(\h{1pt}0, 8 \h{0.5pt}r_* \theta_1\big]$.}
\end{eqnarray}

 Suppose  $x  \in \overbar{D_*^+} \sim T $ and $r \in \big(\h{1pt}0, 4 \h{1pt}r_* \theta_1\big]$. If $T \cap D_r\left(x\right)  \neq \text{\O}$, then \begin{eqnarray}D_r^+\left(x\right) \subset D_{2 r}^+\left(x'\right), \h{10pt}\text{where $x'$ is the projection of $x$ to the $\rho$-axis.}\end{eqnarray}   By (2.8)-(2.9), it holds $E_{x, \h{1pt} r}^+\left[ \h{0.5pt}u\h{0.5pt}\right] \h{2pt}\leq\h{2pt} E_{x', \h{0.5pt}2\h{0.5pt}r}^+\left[\h{0.5pt}u\h{0.5pt}\right] \h{2pt}\lesssim_{r_*}\h{2pt}r^{\alpha}$, which infers \begin{eqnarray}E_{x, \h{1pt}r}^+\left[ \h{0.5pt}u\h{0.5pt}\right] \h{2pt}\lesssim_{r_*}\h{2pt} r^{\alpha}, \h{10pt}\text{provided $\left(x, r\right) \in \big(\overline{D_*^+} \sim T \big) \times \big(\h{1pt}0, 4 \h{1pt}r_* \theta_1\big]$ and $T \cap D_r\left(x\right)   \neq \text{\O}$.}
\end{eqnarray}

 Suppose $x   \in \overline{D_*^+} \sim T$, $r \in \big(\h{1pt}0, 2 \h{1pt}r_* \theta_1\big]$ and $T \cap D_r\left(x\right)  = \text{\O}$. Let $j$ be the  natural number so that \begin{eqnarray}2 \h{0.5pt} r_* \theta_1^{j+1} \h{2pt}<\h{2pt} r \h{2pt}\leq\h{2pt} 2 \h{0.5pt} r_* \theta_1^{ j}.
\end{eqnarray}We assume $j \geq 2$. If $j = 1$, then by (2.11) and the second inequality in (2.4), it holds $E_{x, \h{0.5pt}r}^+\left[\h{0.5pt}u\h{0.5pt}\right]  \h{2pt}\leq\h{2pt}\big( 2\h{0.5pt} r_* \h{0.5pt}\theta_1^2 \big)^{-1} r$.  Denote by  $j_*$ the largest number  in $\big\{ 2, ..., j\h{1pt}\big\}$ so that $T \h{0.5pt}\cap \h{0.5pt}D_{2 \h{1pt}r_*\theta_1^{j_* -1}}\left(x\right)  \neq \text{\O}$. If $j_* = j $, then (2.10)-(2.11) induce \begin{eqnarray*}E_{x, \h{0.5pt}r}^+\left[\h{0.5pt}u\h{0.5pt}\right] \h{2pt}\leq\h{2pt} E_{x, \h{0.5pt}2 r_* \theta_1^{j }}^+\left[\h{0.5pt}u\h{0.5pt}\right] \h{2pt}\leq\h{2pt} E_{x, \h{0.5pt}2 r_* \theta_1^{j_* - 1 }}^+\left[\h{0.5pt}u\h{0.5pt}\right] \h{2pt}\lesssim_{r_*}\h{2pt}\left( 2 \h{0.5pt} r_* \theta_1^{j - 1 } \right)^{\alpha} \h{2pt}\lesssim_{r_*}\h{2pt}r^{\alpha}.
 \end{eqnarray*}Now we assume $j_* < j$. For any natural number $l$ satisfying $ l  \geq  j_*$, we can apply Lemma 2.2 to get \begin{eqnarray} E^+_{x, \h{1pt}2\h{0.5pt}r_*\h{0.5pt}\theta_1^{l+1}} \left[\h{1pt}u\h{1pt}\right] \h{2pt}\leq\h{2pt}\left( 2\h{0.5pt}r_*\h{0.5pt}\theta_1^{l}\right)^{3/2} \h{10pt}\text{or}\h{10pt}E^+_{x, \h{1pt} 2\h{0.5pt}r_*\h{0.5pt}\theta_1^{l+1}} \left[\h{1pt}u\h{1pt}\right] \h{2pt}\leq\h{2pt} \dfrac{1}{2} \h{1pt}E^+_{x, \h{1pt}2\h{0.5pt}r_*\h{0.5pt}\theta_1^{l}} \left[\h{1pt}u\h{1pt}\right], \h{10pt}\text{ for any $l \geq  j_*$.}
 \end{eqnarray}If for any $l = j_*, ..., j - 1$, the second inequality in (2.12) holds, then by (2.11), it turns out \begin{eqnarray*}E_{x, \h{1pt} r}^+\left[\h{0.5pt}u\h{0.5pt}\right] \h{2pt}\leq \h{2pt} E_{x, \h{1pt}2 \h{0.5pt} r_* \theta_1^{j}}^+\left[\h{0.5pt}u\h{0.5pt}\right] \h{2pt}\leq\h{2pt} \left(\dfrac{1}{2}\right)^{j - j_*} E^+_{x, \h{1pt} 2 \h{0.5pt} r_* \theta_1^{j_*}}\left[\h{0.5pt}u\h{0.5pt}\right] \h{2pt}\leq\h{2pt}\left(\dfrac{1}{2}\right)^{j - j_*} E^+_{x, \h{1pt} 2 \h{0.5pt} r_* \theta_1^{j_* - 1}}\left[\h{0.5pt}u\h{0.5pt}\right] .
\end{eqnarray*}Applying (2.10) to the right-hand side above and utilizing (2.11), we obtain \begin{eqnarray*}E_{x, \h{1pt} r}^+\left[\h{0.5pt}u\h{0.5pt}\right] \h{2pt}\lesssim_{r_*} \h{2pt}\left(\dfrac{1}{2}\right)^{j - j_* }  \left(2 r_* \theta_1^{j_* - 1}\right)^{\alpha} \h{2pt}\lesssim_{r_*}\h{2pt}\left(\dfrac{1}{2}\right)^{j} \h{2pt}\lesssim_{r_*}\h{2pt} r^{\alpha}.
\end{eqnarray*}If the first inequality in (2.12) are satisfied by some $l$ in $\Big\{ j_*, ..., j - 1\h{1pt}\Big\}$, then we take $l_*$ to be the largest number in  $\Big\{ j_*, ..., j - 1 \h{1pt}\Big\}$ so that the first inequality in (2.12) holds. It then follows \begin{eqnarray*}E_{x, \h{1pt}r}^+\left[ \h{1pt}u\h{1pt}\right] \h{2pt}\leq\h{2pt} E^+_{x, \h{1pt}2 \h{0.5pt} r_*\h{0.5pt}\theta_1^{j }} \left[\h{1pt}u\h{1pt}\right] \leq \left( \dfrac{1}{2} \right)^{j  - l_* - 1} E^+_{x, \h{1pt}2 \h{0.5pt}r_*\h{0.5pt}\theta_1^{l_* + 1}} \left[\h{1pt}u\h{1pt}\right] \leq \left( \dfrac{1}{2} \right)^{j  - l_* - 1} \left(2 \h{0.5pt} r_*\h{0.5pt}\theta_1^{l_*}\right)^{3/2} \h{2pt}\leq\h{2pt}\left(\dfrac{1}{2}\right)^{j} \h{2pt}\lesssim_{r_*}\h{2pt}r^{\alpha}.
\end{eqnarray*}In summary we have \begin{eqnarray}E_{x, \h{1pt}r}^+\left[ \h{0.5pt}u\h{0.5pt}\right] \h{2pt}\lesssim_{r_*}\h{2pt}r^{\alpha}, \h{10pt}\text{provided $\left(x, r\right) \in \left(\overline{D_*^+} \sim T \right) \times \big(\h{1pt}0, 2 \h{1pt}r_* \theta_1\big]$ and $T \cap D_r\left(x\right)  = \O$.}
\end{eqnarray}

(2.8), (2.10) and (2.13) infer that  $u \in \mathrm{C}^{0, \alpha/2}\big( \overbar{D_*^+}\big)$. Let  $u_{j}$ be the $j$-th component of $u$. If $u_{2}(x_*) = 0$, then there is a $r_1> 0$ sufficiently small so that $D_{r_1}\left(x_* \right) \subset \mathbb{D}$.  Meanwhile it holds \begin{eqnarray}u_{2} > b \h{0.5pt}\big/\h{0.5pt}2 \h{10pt} \text{on $D_{r_1}\left(x_* \right)$, \h{3pt}if $u = u_{\h{1pt}b}^+$;}\h{20pt}u_{2} \h{1pt}<\h{1pt} c \h{0.5pt}\big/\h{0.5pt}2 \h{10pt} \text{on $D_{r_1}\left(x_* \right)$, \h{3pt}if $u = u_{\h{1pt}c}^-$ and $c > 0$.}\end{eqnarray}Notice that if $u = u_{\h{1pt}c}^-$ and $c < 0$, then the second component of $u_{\h{1pt}c}^-$ cannot equal to $0$ at $x_*$. By (2.14) and the fact that $u_b^+$ and $u_c^-$ are minimizers of the two minimization problems in (1.12), respectively,  it can be shown that $u$ solves (2.1) weakly on $D_{r_1}\left(x_*\right)$. Standard elliptic regularity then yields the smoothness of $u$ on $D_{r_1}\left(x_*\right)$. If $u_{2}\left(x_*\right) \neq 0$, then there is a $r_2> 0$ small enough so that $D_{r_2}\left(x_* \right) \subset \mathbb{D}$. Meanwhile $ \big|\h{1pt}u_{2}\h{1pt}\big| \h{2pt}>\h{2pt} 0$ on $D_{r_2}\left(x_* \right)$.  Letting $\eta$ be an arbitrary test function compactly supported on $D_{r_2}\left(x_*\right)$, we have, by   $ D u_{2} = - \dfrac{u_{1}}{u_{2}} \h{1pt}D u_{1} - \dfrac{u_{3}}{u_{2}} \h{1pt}D u_{3}$  on $D_{r_2}\left(x_*\right)$, that \begin{small}\begin{eqnarray*}  \int_{\h{1pt}\mathbb{D}} \rho \h{1pt} D u_{2} \cdot D \eta
  = \int_{\h{1pt}\mathbb{D}} \rho \h{2pt} D\h{0.5pt}u_{1} \cdot \bigg\{\h{1pt} D\left ( - \dfrac{u_{1}}{u_{2}} \h{2pt}\eta \right) + \eta \h{1pt}\dfrac{u_{2}\h{0.5pt} D u_{1} - u_{1}\h{0.5pt} D u_{2}}{u_{2}^2} \h{1pt} \bigg\}+ \int_{\h{1pt}\mathbb{D}} \rho \h{2pt} D\h{0.5pt}u_{3} \cdot \bigg\{\h{1pt} D\left ( - \dfrac{u_{3}}{u_{2}} \h{2pt}\eta \right) + \eta \h{1pt}\dfrac{u_{2}\h{0.5pt} D u_{3} - u_{3}\h{0.5pt} D u_{2}}{u_{2}^2} \h{1pt} \bigg\}.
\end{eqnarray*}\end{small}\noindent Since $u$ solves the equations of $u_1$ and $u_3$  in (2.1) weakly on $D_{r_2}\left(x_*\right)$, then by the above equality, $u$ also solves the equation of $u_2$ in (2.1) weakly on $D_{r_2}\left(x_*\right)$. This induces the smoothness of $u$ on $D_{r_2}\left(x_*\right)$.
\end{proof}

The remaining of this section is devoted to proving Lemma 2.2. We only consider the case when $x \in T$. The  case when $x$ is off the free boundary $T$ can be similarly proved. The proof  relies on a blow-up process of Luckhaus type (see [43]-[44]). As in the proof of Proposition 2.1, we still use $u$ to denote either $u_{\h{1pt}b}^+$ or $u_{\h{1pt}c}^-$.

Suppose that we can find two sequences $\Big\{\h{0.5pt} x_n = \left(\rho_n, 0\right) \h{0.5pt}\Big\} \subset T^\circ$ and $\Big\{ r_n \Big\} \subset \mathbb{R}^+$ so that the followings hold for any $n$: \begin{eqnarray}(1). \h{3pt}D_{r_n}^+\left(x_n\right) \subset D_{\sigma_*}^+\left(x_*\right);  \h{10pt}(2). \h{3pt}E^+_{x_n, \h{1pt}r_n \theta_1} \left[\h{0.5pt}u\h{0.5pt}\right] \h{2pt}>\h{2pt} r_n^{3/2}; \h{10pt} (3). \h{3pt} E^+_{ x_n, \h{1pt}r_n \theta_1} \left[\h{0.5pt}u\h{0.5pt}\right] \h{2pt}>\h{2pt} \dfrac{1}{2} \h{1.5pt} E_{x_n, \h{1pt}r_n}^+\left[\h{0.5pt}u\h{0.5pt}\right].
\end{eqnarray}Here $\sigma_*$ is given in Lemma 2.2. Meanwhile the $E_{\h{1pt}\mathbb{D}}$\h{0.5pt}-\h{0.5pt}energy of $u$ on $D^+_{r_n}(x_n)$ satisfies \begin{eqnarray} E_{x_n, \h{1pt}r_n}^+\left[\h{0.5pt}u \h{0.5pt}\right] \longrightarrow 0, \h{20pt}\text{as $n \rightarrow \infty$.}
\end{eqnarray}This  limit induces \begin{eqnarray} r_n \longrightarrow0, \h{20pt}\text{as $n \rightarrow \infty$.}
\end{eqnarray}Otherwise we can extract a subsequence, still denoted by $\big\{n\big\}$, so that for some $r_0 > 0$ and $x_0 \in T^\circ$, $r_n \rightarrow r_0 $ and $x_n \rightarrow x_0$ as $n \rightarrow \infty$. By (2.16), it follows $ E_{x_0, \h{1pt} r_0}^+ \left[\h{0.5pt}u \h{0.5pt}\right] = 0$ which gives $u_1 \equiv 0$ on $D_{r_0}^+(x_0)$. Utilizing this result, the non-negativity of $u_1$ on $\mathbb{D}^+$ and the equation satisfied by  $u_1$ (see (2.1)), we can apply the strong maximum principle to $u_1$ and obtain $u_1 \equiv 0$ on $\mathbb{D}^+$. This result  violates the boundary condition of $u$. Letting  \begin{eqnarray} y_n := \dashint_{D_{r_n}^+\left(x_n\right)} u\left(\rho, z\right) \h{20pt}\text{and}\h{20pt} \epsilon_n^2 := r_n^2 +  \int_{D^+_{r_n}(x_n)} |\h{1pt}D u \h{1pt}|^2, \end{eqnarray} we also have \begin{lemma}The following limits hold for $\epsilon_n$ and $r_n$: \begin{eqnarray} (1). \h{5pt}\epsilon_n \longrightarrow 0; \h{40pt}(2). \h{5pt} \dfrac{r_n}{\epsilon_n} \longrightarrow 0, \h{20pt}\text{as $n \rightarrow \infty$.}
\end{eqnarray}
  \end{lemma}\begin{proof}[\bf Proof]By (1) in (2.15), it holds \begin{eqnarray}\rho \geq \rho_*\h{1pt}\big/ \h{1pt}2, \h{10pt}\text{ for any $\left(\rho, z\right) \in D_{r_n}^+\left(x_n\right)$.} \end{eqnarray} Using (2.16) and (2.20), we obtain \begin{eqnarray*} \dfrac{\rho_*}{2} \int_{D_{r_n}^+\left(x_n\right)} \big| \h{1pt}D \h{0.5pt}u\h{1pt}\big|^2 \h{2pt}\leq\h{2pt} \int_{D_{r_n}^+\left(x_n\right)} \rho \h{1.5pt}\big| \h{1pt}D \h{0.5pt}u\h{1pt}\big|^2 \h{2pt}\leq\h{2pt}E_{x_n, \h{0.5pt}r_n}^+\left[\h{0.5pt}u\h{1pt}\right] \longrightarrow 0, \h{10pt}\text{as $n \rightarrow \infty$.}
  \end{eqnarray*}The above estimate and (2.17) induce (1) in (2.19).

 Utilizing (2.20)  and (2) in (2.15), we can find a positive constant $c_{\h{0.5pt}\mu, \rho_*}$ depending on $\mu$ and $\rho_*$ so that   \begin{eqnarray*} c_{\h{0.5pt}\mu, \rho_*}\h{2pt}r_n^2 + \int_{D_{r_n}^+\left(x_n\right)} \big|\h{1pt}D\h{0.5pt}u\h{1pt}\big|^2 \h{2pt}\geq\h{2pt}E_{x_n, r_n}^+\left[\h{0.5pt}u\h{0.5pt}\right]\h{2pt}\geq\h{2pt}E_{x_n, \h{0.5pt}r_n \theta_1}^+\left[\h{0.5pt}u\h{0.5pt}\right]\h{2pt}>\h{2pt}r_n^{3/2}.
  \end{eqnarray*}By (2.17) and the above estimate, we get $\displaystyle \int_{D_{r_n}^+\left(x_n\right)} \big|\h{1pt}D\h{0.5pt}u\h{1pt}\big|^2 \h{2pt}>\h{2pt}\dfrac{r_n^{3/2}}{2}$, provided $n$ is large. Recall the definition of $\epsilon_n^2$ in (2.18). The last lower bound infers \begin{eqnarray*} \left(\dfrac{r_n}{\epsilon_n}\right)^2 \h{2pt}\leq\h{2pt}\dfrac{2\h{1pt}r_n^{1/2}}{1 + 2 \h{1pt}r_n^{1/2}}, \h{10pt}\text{provided $n$ is large.}
  \end{eqnarray*}(2) in (2.19) then follows by this upper bound and (2.17).
  \end{proof}

 Now we rescale domain and value of $u$ by letting \begin{eqnarray} u_n\left(\xi\right) := \dfrac{u\left(x_n + r_n \h{1pt}\xi\right) - y_n}{\epsilon_n}, \h{20pt}\forall \h{3pt}\xi \in D_1^+.
\end{eqnarray}Here we still use $D_r\left(x \right)$ to denote the disk in the $\xi$-space with center $x$ and radius $r$. Moreover $D_r^+\left( x \right)$ is still the upper part of $D_r\left(x\right)$ where the second coordinate $\xi_2 > 0$. For simplicity if the center is $0$, then we use $D_r$ and $D_r^+$ to represent $D_r(0)$ and $D_r^+(0)$, respectively. With aid of Poincar\'{e}'s inequality and the definition of $\epsilon_n$ in (2.18), $u_n$ is uniformly bounded in $H^1(D_1^+)$. Then there exists a $u_{\infty} \in H^1(D^+_1)$ so that up to a subsequence, \begin{eqnarray} (1). \h{5pt}u_n \longrightarrow u_{\infty}, \h{10pt}\text{weakly in $H^1(D^+_1)$}; \h{20pt}(2). \h{5pt}u_n \longrightarrow u_{\infty}, \h{10pt}\text{strongly in $L^2(D^+_1)$, as $n \rightarrow \infty$.} \end{eqnarray}Moreover for any $p \in \left(1, 2\right)$ we also have uniform boundedness of $u_n$ in $W^{1, p}\left(D_1^+\right)$. Therefore by Theorem 6.2 in [50], we can keep extracting a subsequence, still denoted by $\big\{ u_n\big\}$, so that \begin{eqnarray}u_n \longrightarrow u_{\infty}, \h{20pt}\text{strongly in $L^q \left( \p D_1^+ \right)$, for any $q$ satisfying $1 \h{2pt}\geq\h{2pt} \dfrac{1}{q}\h{2pt} > \h{2pt} \dfrac{1}{p} - \dfrac{1}{2}$.}
\end{eqnarray}The limit $u_{\infty}$ is harmonic in $D_1^+$. That is  \begin{lemma}$u_{\infty}$ is smooth in $D_1^+$ and  satisfies $\Delta_{\xi} \h{0.5pt}u_{\infty} = 0$ in $D_1^+$.
\end{lemma}The proof of Lemma 2.4 is similar to proof of (3.13) in [21]. We omit it here. From Lemma 2.4, we can not imply the regularity of $u_{\infty}$ up to the flat boundary $T_1 := \Big\{ \left(\h{0.5pt}\xi_1, 0 \h{0.5pt}\right): \xi_1 \in \left( - 1, 1 \right)\Big\}.$ To achieve the boundary regularity of $u_{\infty}$, we need to figure out the range of $u_{\infty}$ on $D_1^+$ and $T_1$. Before proceeding, let us introduce some notations. We use $\Gamma$ to denote the supporting manifold of $u$ on $T$. More precisely $\Gamma$ equals to either \begin{eqnarray} \Gamma^+ := \Big\{ \left(\h{1pt}v_1, v_2, 0 \h{1pt}\right) \h{2pt}\in\h{2pt} \mathbb{S}^2 : v_2 \geq b \h{1pt}\Big\} \h{20pt}\text{or}\h{20pt}\Gamma^- := \Big\{ \left(\h{1pt}v_1, v_2, 0 \h{1pt}\right) \h{2pt}\in\h{2pt} \mathbb{S}^2 : v_2 \leq c \h{1pt}\Big\}.
\end{eqnarray}Let $U\left(\Gamma\right) \subset \mathbb{R}^3$ be a small neighborhood of $\Gamma$ in $\mathbb{R}^3$. For any $x \in U\left(\Gamma\right)$, we denote by $\Pi_{\Gamma}\left(x\right)$ the unique point on $\Gamma$ so that $\big|\h{1pt}x - \Pi_{\Gamma}\left(x\right)\h{1pt}\big| = \min\Big\{\h{1pt} |\h{1pt}x - y \h{1pt}| : y \in \Gamma \h{1pt}\Big\}$. The operator $\Pi_{\Gamma}$ is Lipschitz continuous from $U\left(\Gamma\right)$ to $\Gamma$. In the remaining arguments, we need to keep extracting subsequences. Without mentioning, for a sequence in one of lemmas below, it should satisfy its associated limits obtained before that lemma. \vspace{0.3pc}

  Now we begin to consider the range of $u_{\infty}$ on $D_1^+$ and $T_1$. Firstly we have \begin{lemma}If $n$ is large enough, then $y_n \in U\left(\Gamma\right)$ and $\Pi_{\Gamma}\left(y_n\right)$ is well defined. Moreover $Y_n := \dfrac{\Pi_{\Gamma}(y_n) - y_n}{\epsilon_n}$  is uniformly bounded for all $n$ large enough.
\end{lemma}
\begin{proof}[\bf Proof] For $\mathscr{H}^1$\h{0.5pt}-\h{0.5pt}almost all  $\xi \in T_1$, it holds $u\left(x_n + r_n \h{1pt}\xi\right) \in \Gamma$. Therefore  $\mathrm{d}\h{0.5pt}\big(y_n, \Gamma \big) \leq \big| \h{1pt}u\left(x_n + r_n \h{1pt}\xi\right) - y_n \h{1pt}\big|$,  for $\mathscr{H}^1$\h{0.5pt}-\h{0.5pt}almost all $\xi \in T_1$. Here $\mathrm{d}\left(x, \Gamma\right)$ is the shortest distance between a point $x \in \mathbb{R}^3$ and $\Gamma$. Integrating with respect to $\xi \in T_1$ yields \begin{eqnarray*} 2 \h{1pt}\mathrm{d}^2\big( y_n, \Gamma \big) \leq \int_{T_1} \big| \h{1pt}u(x_n + r_n \h{1pt}\xi) - y_n \h{1pt}\big|^2 \leq \int_{\p D_1^+} \big| \h{1pt}u(x_n + r_n \h{1pt}\xi) - y_n \h{1pt}\big|^2.
\end{eqnarray*}By trace theorem on Lipschitz domain and Poincar\'{e} inequality, the above estimate can be reduced to  $\mathrm{d}^2\big(y_n, \Gamma \big) \h{2pt}\lesssim\h{2pt}\epsilon_n^2$. Here we also have used (2.18). The proof then follows.
\end{proof}By this lemma, up to a subsequence, still denoted by $n$, we have \begin{eqnarray} \lim_{n \rightarrow \infty} Y_n = Y_{*}, \h{20pt}\text{for some $Y_* \in \mathbb{R}^3$.} \end{eqnarray}The next lemma concerns about the range of $u_{\infty}$ when it is restricted on $D_1^+$.
 \begin{lemma}There is a $y_* \in \Gamma$ so that up to a subsequence $\Pi_{\Gamma}(y_n) \longrightarrow y_*$, as $n \rightarrow \infty$. Moreover it holds \begin{eqnarray*} u_{\infty} \in Y_* + \mathrm{Tan}_{y_*} \h{1pt}\mathbb{S}^2, \h{20pt}\text{for all $\xi \in D_1^+$.}\end{eqnarray*}Here $\mathrm{Tan}_{y_*} \h{1pt}\mathbb{S}^2$ is the $2$-plane in $\mathbb{R}^3$ which contains all vectors perpendicular to $y_*$. \end{lemma}
\begin{proof}[\bf Proof]  We decompose $u_n$ as follows: \begin{eqnarray} u_n = \dfrac{u\left(x_n + r_n\h{1pt}\xi\right) - \Pi_{\Gamma}\left(y_n\right)}{\epsilon_n} + \dfrac{\Pi_{\Gamma}\left(y_n\right) - y_n}{\epsilon_n}.
\end{eqnarray}In light of (2.25), the second term on the right-hand side of (2.26) converges to $Y_*$.  By (2) in (2.22), $u_n$ converges to $u_{\infty}$ for  almost all $\xi \in D_1^+$. Therefore the limit of \begin{eqnarray}\dfrac{u\left(x_n + r_n\h{1pt}\xi\right) - \Pi_{\Gamma}\left(y_n\right)}{\epsilon_n} \end{eqnarray} exists for almost all $\xi \in D_1^+$. Since $\Pi_{\Gamma}\left(y_n\right) \in \Gamma$, up to a subsequence, there is a $y_* \in \Gamma$ so that $\Pi_{\Gamma}\left(y_n\right) \rightarrow y_*$ as $n \rightarrow \infty$. Then for almost all $\xi \in D_1^+$, $u\left(x_n + r_n \h{1pt}\xi\right)$ converges to $y_*$ also as $n \rightarrow \infty$. Here we have used (1) in (2.19). Hence the limit of the ratio in  (2.27) is  perpendicular to $y_*$ for almost all $\xi \in D_1^+$. This infers $u_{\infty} \in Y_* + \mathrm{Tan}_{y_*}\mathbb{S}^2$ for almost all $\xi \in D_1^+$. The proof follows since by  Lemma 2.4, $u_{\infty}$ is smooth in $D_1^+$.
\end{proof}As for the boundary values of $u_{\infty}$ on $T_1$, we need to consider the following three cases: \begin{eqnarray} I.  \h{3pt}y_* \in \Gamma \sim \p\h{1pt}\Gamma; \h{5pt}\textit{II.} \h{3pt} y_* \in \p\h{1pt}\Gamma \h{2pt}\text{and} \h{2pt}\liminf_{n \rightarrow \infty} \h{2pt}\dfrac{\big|\h{1pt}\Pi_{\Gamma}(y_n) - y_* \h{1pt}\big|}{\epsilon_n} = \infty; \h{5pt}\textit{III.} \h{3pt}y_* \in \p\h{1pt}\Gamma \h{2pt}\text{and} \h{2pt}\liminf_{n \rightarrow \infty} \h{2pt}\dfrac{\big|\h{1pt}\Pi_{\Gamma}(y_n) - y_* \h{1pt}\big|}{\epsilon_n} < \infty.
\end{eqnarray} For cases \textit{I} and \textit{II}, we have
\begin{lemma} Suppose that  I or II in (2.28) holds. Then   $u_{\infty}$ satisfies \begin{eqnarray} u_{\infty} \in Y_* + \mathrm{Tan}_{y_*}\h{1pt}\Gamma, \h{20pt}\text{for $\mathscr{H}^1$-almost all $\xi \in T_1$.}
\end{eqnarray}  Here $\mathrm{Tan}_{y_*} \h{1pt}\Gamma$ is the $1$-dimensional line in $\mathbb{R}^3$ which contains all vectors perpendicular to $y_*$ and $\big(0, 0, 1\big)$.
\end{lemma}\begin{proof}[\bf Proof]
In light that $u\left(x_n + r_n \h{1pt}\xi\right) \in \Gamma$ for $\mathscr{H}^1$-almost all $\xi \in T_1$, this lemma can then be  proved by similar method used in the proof of Lemma 2.6. Here we just need to use (2.23) and (2.25)-(2.26).
\end{proof} In fact (2.29) still holds for \textit{III} in (2.28). But in this case we can have more detailed result. That is  \begin{lemma}If III in (2.28) holds, then there exists a $v_- \in \mathrm{Tan}^-_{y_*} \h{1pt}\Gamma$ so that \begin{eqnarray} u_{\infty} \h{2pt}\in\h{2pt} Y_* + v_- + \mathrm{Tan}_{y_*}^+\h{1pt}\Gamma, \h{20pt}\text{for $\mathscr{H}^1$-almost all $\xi \in T_1$.}
\end{eqnarray}Here for $y_* \in \p \h{2pt} \Gamma$, we use $\mathrm{Tan}^{+}_{y_*} \h{1pt}\Gamma$ to denote the half line in $\mathrm{Tan}_{y_*} \h{1pt}\Gamma$ so that for any $v \in \mathrm{Tan}_{y_*}^{+} \h{1pt}\Gamma$, it holds \begin{eqnarray*} \dfrac{y_* + \epsilon \h{1pt}v}{\big|\h{1pt}y_* + \epsilon \h{1pt}v\h{0.5pt}\big|} \h{2pt}\in\h{2pt} \Gamma, \h{20pt}\text{provided $\epsilon > 0$ and sufficiently small.}
 \end{eqnarray*}Similarly $\mathrm{Tan}^-_{y_*} \h{1pt} \Gamma$ denotes the half line in $\mathrm{Tan}_{y_*} \h{1pt} \Gamma$ so that for any $v \h{1pt}\in\h{1pt} \mathrm{Tan}^-_{y_*} \h{1pt} \Gamma$,  it holds $- v \h{1pt}\in\h{1pt} \mathrm{Tan}^+_{y_*} \h{1pt} \Gamma$. \end{lemma}\begin{proof}[\bf Proof] If \textit{III} in (2.28) holds, then up to a subsequence \begin{eqnarray} \dfrac{y_* - \Pi_{\Gamma}\left(y_n\right)}{\epsilon_n} \longrightarrow v_-, \h{20pt}\text{for some $v_- \in  \mathrm{Tan}_{y_*}^-\h{1pt}\Gamma$.}
\end{eqnarray}Here we have used $\Pi_{\Gamma}\left(y_n\right) \in \Gamma$ and $\Pi_{\Gamma}\left(y_n\right) \rightarrow y_*$ as $n \rightarrow \infty$. By (2.26), we can further decompose $u_n$ into \begin{eqnarray}u_n = \dfrac{\Pi_{\Gamma}\left(y_n\right) - y_n}{\epsilon_n} + \dfrac{y_* - \Pi_{\Gamma}\left(y_n\right)}{\epsilon_n} + \dfrac{u\left(x_n + r_n\h{1pt}\xi\right) - y_*}{\epsilon_n}.
\end{eqnarray}In light of (2.23), (2.25), (2.31), as $n \rightarrow \infty$, the limit of the last term in (2.32) exists for $\mathscr{H}^1$-almost all $\xi \in T_1$. Hence for $\mathscr{H}^1$-almost all $\xi \in T_1$, $u\left(x_n + r_n\h{0.5pt}\xi\right) \in \Gamma$ and converges to $y_*$ as $n \rightarrow \infty$. This infers $$ \lim_{n \rightarrow \infty} \dfrac{u\left(x_n + r_n\h{1pt}\xi\right) - y_*}{\epsilon_n} \h{1pt}\in\h{1pt} \mathrm{Tan}_{y_*}^+ \Gamma, \h{20pt}\text{for $\mathscr{H}^1$-almost all $\xi \in T_1$.}$$ Using this result, (2.23), (2.25) and (2.31), we can take $n \rightarrow \infty$ in (2.32) to get (2.30).   \end{proof}

Motivated by Lemmas 2.6-2.8, we introduce the following two configurations spaces: \begin{eqnarray} \mathscr{F}_{A, \h{1pt}\sigma} := \left\{ u \in H^1\big( D_{\sigma}^+ \big) \h{2pt}\Bigg|\h{2pt}  \begin{array}{lcl}
u \in Y_* + \mathrm{Tan}_{y_*} \h{1pt}\mathbb{S}^2 \h{5pt}&& \text{on \h{4pt}$D_{\sigma}^+$}\vspace{0.5pc}\\
 u \in Y_* + \mathrm{Tan}_{y_*} \h{1pt}\Gamma \h{5pt}&&\text{on\h{8pt}$ T_{\sigma}$} \end{array}\right\}\end{eqnarray} and \begin{eqnarray} \mathscr{F}_{B, \h{1pt}\sigma} := \left\{ u \in H^1\big( D_{\sigma}^+ \big) \h{2pt}\Bigg|\h{2pt}  \begin{array}{lcl}
u \in Y_*  + \mathrm{Tan}_{y_*} \h{1pt}\mathbb{S}^2 \h{5pt}&& \text{on \h{4pt}$D_{\sigma}^+$}\vspace{0.5pc}\\
 u \in Y_* + v_- + \mathrm{Tan}^+_{y_*} \h{1pt}\Gamma \h{5pt}&&\text{on\h{8pt}$ T_{\sigma}$} \end{array}\right\}.\end{eqnarray} Here $\sigma$ is an arbitrary number in $\left(\h{0.5pt}0, 1 \h{0.5pt}\right]$. $T_{\sigma}$ denotes the set  $\Big\{ \left( \h{1pt}\xi_1, 0 \h{1pt}\right) : \xi_1 \in \left( - \sigma, \sigma \right) \h{1pt}\Big\}.$ We claim that if \textit{I} or \textit{II} (\textit{III} resp.) in (2.28) holds, then $u_{\infty}$ is an energy minimizing map in $\mathscr{F}_{A, \h{1pt}1}$ ($\mathscr{F}_{B, \h{1pt}1}$ resp.). See the definition of energy minimizing map in Definition 2.10. This claim will be shown in Lemma 2.11. In the following we consider the uniform convergence of $u_n$ on $S_{\sigma}^+$, where for any $\sigma \in \left( 0, 1 \right)$, $S_{\sigma}^+ := \Big\{ \h{1pt}\xi = \left(\h{1pt}\xi_1, \xi_2 \h{1pt}\right) : |\h{1pt}\xi\h{1pt}| = \sigma \h{4pt}\text{and} \h{4pt}\xi_2 \geq 0 \h{1pt}\Big\}.$
\begin{lemma}There exists an increasing positive sequence $\Big\{ \sigma_k \Big\}$ which takes limit $1$ as $k \rightarrow \infty$ and a subsequence of $\Big\{ u_n \Big\}$, still denoted by $\Big\{ u_n \Big\}$,  so that the followings hold:\begin{enumerate}[(1).]
 \item For any $k$, there exists a positive constant $b_k$ independent of $n$ so that \begin{eqnarray*} \sup_{n} \int_{S_{\sigma_k}^+} \big| \h{1pt}D_{\xi} u_n \h{1pt}\big|^2 + \int_{S_{\sigma_k}^+} \big|\h{1pt}D_{\xi} u_{\infty} \h{1pt}\big|^2  \h{2pt}\leq\h{2pt} b_k;
\end{eqnarray*}
\item For any $n$ and $k$, $u_n$ is continuous on $S_{\sigma_k}^+$. For any $k$, $u_n $ converges to $u_{\infty}$ uniformly on $S_{\sigma_k}^+$ as $n \rightarrow \infty$;
\item For any $n$ and $ k$, $y_n + \epsilon_n \h{0.5pt} u_n $ takes value in $ \Gamma$ when it is evaluated at $\left(\h{1pt} \pm \h{1.5pt} \sigma_k, 0 \h{1pt}\right)$;
\item If I or II in (2.28) holds, then for any $k$, $u_{\infty}\left(\h{1pt} \pm \h{1.5pt} \sigma_k, 0 \h{1pt}\right) \in Y_* + \mathrm{Tan}_{y_*} \h{1pt}\Gamma$;
\item If III in (2.28) holds, then for any $k$, $u_{\infty} \left(\h{1pt} \pm \h{1.5pt} \sigma_k, 0\h{1pt}\right) \in Y_* + v_- + \mathrm{Tan}_{y_*}^+ \h{1pt}\Gamma$.
\end{enumerate}
\end{lemma}
\begin{proof}[\bf Proof] By (2) in (2.22), as $n \rightarrow \infty$, it holds $ \displaystyle \int_{S_{\sigma}^+} \big|\h{1pt}u_n - u_{\infty} \h{1pt}\big|^2 \longrightarrow 0$, for $\mathscr{H}^1$-almost all $\sigma \in (0, 1)$. Utilizing Fatou's lemma, we have \begin{eqnarray*} \int_0^1  \liminf_{n \rightarrow \infty}  \int_{S_{\rho}^+} \big|\h{1pt}D_{\xi} u_n \h{1pt}\big|^2  \h{2pt}\leq\h{2pt} \liminf_{n \rightarrow \infty} \int_{D_1^+} \big|\h{1pt}D_{\xi} u_n \h{1pt}\big|^2\h{2pt}<\h{2pt}\infty.
\end{eqnarray*}Therefore  it turns out $\displaystyle \liminf_{n \rightarrow \infty} \int_{S_{\sigma}^+} \big|\h{1pt}D_{\xi} u_n \h{1pt}\big|^2 < \infty$,  for $\mathscr{H}^1$-almost all $\sigma \in (0, 1)$. In light that $u_{\infty} \in H^1(D^+_1)$, we also have $\displaystyle \int_{S_{\sigma}^+} \big|\h{1pt}D_{\xi} u_{\infty} \h{1pt}\big|^2 < \infty$,  for $\mathscr{H}^1$-almost all $\sigma \in (0, 1)$.  By trace theorem, $y_n + \epsilon_n \h{0.5pt}u_n \in \Gamma $ for $\mathscr{H}^1$-almost all $\xi \in T_1$. In light of all the arguments above and Lemmas 2.7-2.8, we can find an increasing positive sequence, denoted by  $\big\{\h{0.5pt}\sigma_k \h{0.5pt}\big\}$, which takes limit $1$ as $k \rightarrow \infty$ and a subsequence of $\big\{ \h{0.5pt}u_n \h{0.5pt}\big\}$, still denoted by $\big\{ u_n \big\}$, so that (1) and (3)-(5) in Lemma 2.9 hold. Moreover  it satisfies \begin{eqnarray} \lim_{n \rightarrow \infty}\int_{S_{\sigma_k}^+} \big|\h{1pt}u_n - u_{\infty} \h{1pt}\big|^2 = 0, \h{20pt}\text{for any $k$.}
\end{eqnarray}

In the remaining we consider (2) in Lemma 2.9. Fix $n$ and $k$. By (1) in Lemma 2.9, $u_n$ is absolute continuous on $S_{\sigma_k}^+$. Meanwhile there holds, for all $\alpha_1, \alpha_2 \in \big[\h{0.5pt}0, \pi\h{0.5pt}\big]$, that \begin{eqnarray*} \Big|\h{1pt} u_{n}\big(\sigma_k\cos \alpha_2, \h{1pt} \sigma_k \sin \alpha_2\big) - u_{n}\big(\sigma_k\cos \alpha_1, \h{1pt} \sigma_k \sin \alpha_1\big) \h{1pt}\Big|^2 \h{2pt}\leq\h{2pt}\sigma_k\h{1pt} | \h{1pt}\alpha_2 - \alpha_1 \h{1pt}| \h{2pt} \int_{S_{\sigma_k}^+} \big|\h{1pt}D_{\xi} u_n\h{1pt}\big|^2 \h{2pt}\leq\h{2pt} b_k \h{1pt}\sigma_k \h{1pt} \big| \h{1pt}\alpha_2 - \alpha_1 \h{1pt}\big|.
\end{eqnarray*}This infers the equicontinuity of the family $\left\{ u_{n}\h{0.5pt}\big|\h{0.5pt}_{S_{\sigma_k}^+} : n \in \mathbb{N}\h{1pt}\right\}$. Using (2.35), we can find a $\alpha_0 \in \big(0, \pi\big)$ so that $u_{n}$ converges to $u_{\infty}$ at $\big(\sigma_k \cos \alpha_0, \h{1pt}\sigma_k \sin \alpha_0\big)$. Then by the last estimate, the $L^{\infty}$-norm of $u_{n}$ on $S_{\sigma_k}^+$ is uniformly bounded with respect to $n$.  In light of Arzel\`{a}-Ascoli theorem, there exists a subsequence of $\big\{u_n\big\}$ so that the subsequence uniformly converges to $u_{\infty}$ on $S_{\sigma_k}^+$. Applying inductive arguments and diagonal process, we can extract a subsequence from $\big\{u_n\big\}$ so that (2) in Lemma 2.9 holds. The proof finishes.
\end{proof}

 To proceed we need to introduce the concept of energy minimizing map in $\mathscr{F}_{A, \h{0.5pt}1}$ and $\mathscr{F}_{B, \h{0.5pt}1}$. \begin{definition}Recall the configuration spaces $\mathscr{F}_{A, \h{0.5pt}\sigma}$ and $\mathscr{F}_{B, \h{0.5pt}\sigma}$ defined in (2.33)-(2.34). We call $w$ an energy minimizing map in $\mathscr{F}_{A, \h{0.5pt}1}$ ($\mathscr{F}_{B, \h{0.5pt}1}$ resp.) if it holds \begin{eqnarray*} \int_{D_r^+} | \h{1pt}D_{\xi}\h{0.5pt} w \h{1pt}|^2 \h{2pt}\leq\h{2pt}\int_{D_r^+} |\h{1pt}D_{\xi} \h{0.5pt} v\h{1pt}|^2, \h{20pt}\text{for any $r \in \left(\h{0.5pt}0, 1\h{0.5pt}\right)$ and $v \in \mathscr{F}_{A, \h{0.5pt}r} $ ($\mathscr{F}_{B, \h{0.5pt}r}$ resp.) with $v = w$ on $S_r^+$.}
 \end{eqnarray*}
 \end{definition}Now we characterize the limiting function $u_{\infty}$ as follows:
 \begin{lemma}If I or II in (2.28) holds, then $u_{\infty}$ is an energy minimizing map in $\mathscr{F}_{A, \h{0.5pt}1}$. If III in (2.28) holds, then $u_{\infty}$ is an energy minimizing map in $\mathscr{F}_{B, \h{0.5pt}1}$. In all cases $u_n \longrightarrow u_{\infty}$ strongly in $H^1\big(D_r^+\big)$ for any $r \in (0, 1)$.
\end{lemma}\begin{proof}[\bf Proof]  We divide the proof into four steps. \vspace{0.5pc} \\
\textbf{Step 1. Comparison map.}  Let $\big\{\h{1pt}\sigma_k\h{1pt}\big\}$ and $\big\{\h{1pt}u_n\h{1pt}\big\}$ be the sequences in  Lemma 2.9. Fixing a $k$, we pick up an arbitrary $v \h{1pt}\in\h{1pt} H^1\big(\h{0.5pt} D_{\sigma_k}^+; Y_* + \mathrm{Tan}_{y_*}\h{1pt}\mathbb{S}^2 \h{1pt}\big)$ with $v = u_{\infty}$ on $S_{\sigma_k}^+$. Moreover, in the sense of trace we let  \begin{eqnarray} v \h{2pt}\in\h{2pt}  \left\{ \begin{array}{lcl} Y_* + \mathrm{Tan}_{y_*} \Gamma \h{30pt}&&\text{on $T_{\sigma_k}$}, \h{20pt} \text{if \textit{I} or \textit{II} in (2.28) holds;}\vspace{0.5pc}\\
 Y_* + v_-  + \mathrm{Tan}_{y_*}^+\h{1pt} \Gamma &&\text{on $T_{\sigma_k}$,} \h{21pt} \text{if \textit{III} in (2.28) holds.} \end{array}\right. \end{eqnarray} Here $v_-$ is given in Lemma 2.8. We cannot compare directly the energies between $v$ and $u_n$ since they have different ranges and boundary values.

Fixing an arbitrary $R > 0$, we define for any $n$ the map \begin{eqnarray}  \mathscr{V}_n := \Gamma_n + R\h{1pt}\epsilon_n\h{1pt}\dfrac{  v - v_* - Y_* }{|\h{1pt} v - v_* - Y_* \h{1pt} |\vee R} \h{5pt}\text{on $D_{\sigma_k}^+$,} \h{10pt}\text{where $\Gamma_n := \left\{ \begin{array}{lcl} \Pi_{\Gamma}\left(y_n\right), &&\text{for  \textit{I} or \textit{II} in (2.28)};  \vspace{0.5pc}\\
 y_*, &&\text{for \textit{III} in (2.28).} \end{array}\right. $}
\end{eqnarray}In (2.37), $a \vee b := \mathrm{max}\big\{a, b\big\}$ for two real numbers $a$, $b$. The vector $v_*$ is $0$ if \textit{I} or \textit{II} in (2.28) holds. It equals to $v_-$  if \textit{III} in (2.28) holds. For large $n$, the length of $\mathscr{V}_n$ is close to $1$. Then we define \begin{eqnarray}v_n := \dfrac{\mathscr{V}_n}{|\h{0.5pt}\mathscr{V}_n\h{0.5pt}|} \h{20pt}\text{on $D_{\sigma_k}^+$}.\end{eqnarray}Recall that  the sequence $\Big\{\h{1pt}\Gamma_n \h{1pt}\Big\} \subset \Gamma$ and converges to $ y_*$, as $n \rightarrow \infty$. By taking $n$ sufficiently large and assuming \textit{I} in (2.28),  it turns out $ v_n \h{2pt}\in \h{2pt}\Gamma$, for all points on $T_{\sigma_k}$ at which $v  \in Y_* + \mathrm{Tan}_{y_*} \Gamma$. Similarly if $n$ is large and \textit{III}  in (2.28) holds, we also have $v_n \h{2pt}\in \h{2pt}\Gamma$, for all points on $T_{\sigma_k}$ at which $v \in Y_* + v_- + \mathrm{Tan}_{y_*}^+ \Gamma$.  As for \textit{II} in (2.28), we  decompose $\mathscr{V}_n$ into \begin{eqnarray*} \mathscr{V}_n = \gamma_n + \Big[\h{1pt} \Pi_{\Gamma}\left(y_n\right) - \gamma_n \h{1pt}\Big] + R\h{1pt}\epsilon_n\h{1pt}\dfrac{v - Y_*}{\big|\h{1pt}v - Y_*\h{1pt}\big| \vee R}, \h{20pt}\text{where  $\gamma_n = y_* \h{1pt}\Big<\h{1pt}y_*, \h{1pt}\Pi_{\Gamma}\left(y_n\right)\h{0.5pt}\Big> $.}
\end{eqnarray*}Denoting by $v^+$ the unit vector in $\mathrm{Tan}_{y_*}^+\Gamma$, we rewrite the above equality by \begin{eqnarray}\mathscr{V}_n = \gamma_n + \epsilon_n \left[\h{1pt} \dfrac{\big|\h{1pt}\Pi_{\Gamma}\left(y_n\right) - \gamma_n\h{1pt}\big|}{\epsilon_n} \h{1.5pt}v^+ + R\h{1pt}\dfrac{v - Y_*}{\big|\h{1pt}v - Y_*\h{1pt}\big| \vee R} \right].
\end{eqnarray}When $n$ is large enough, $\Pi_{\Gamma}\left(y_n\right)$ and $\gamma_n$ are close to $y_*$. Hence it holds $ 2 \h{1pt} \big|\h{1pt}\Pi_{\Gamma}\left(y_n\right) - \gamma_n\h{1pt}\big|  > \big|\h{1pt} \Pi_{\Gamma}(y_n) - y_* \h{1pt}\big|$, for large $n$. By this lower bound and the assumption in \textit{II} of (2.28), we get $$\dfrac{\big|\h{1pt}\Pi_{\Gamma}\left(y_n\right) - \gamma_n\h{1pt}\big|}{\epsilon_n} \longrightarrow \infty, \h{20pt}\text{ as $n \rightarrow \infty$.}$$ In light that $v^+ \in \mathrm{Tan}_{y_*}^+\Gamma$, then for all points on $T_{\sigma_k}$ with $v   \in Y_* +  \mathrm{Tan}_{y_*} \Gamma$, it turns out \begin{eqnarray}\dfrac{\big|\h{1pt}\Pi_{\Gamma}\left(y_n\right) - \gamma_n\h{1pt}\big|}{\epsilon_n} \h{1.5pt}v^+ + R\h{1pt}\dfrac{v - Y_*}{\big|\h{1pt}v - Y_*\h{1pt}\big| \vee R} \h{2pt}\in\h{2pt} \mathrm{Tan}_{y_*}^+ \Gamma, \h{20pt}\text{provided $n$ is large.}
 \end{eqnarray}Utilizing (2.39)-(2.40) yields $v_n \in \Gamma$, for large $n$ and all points on $T_{\sigma_k}$ at where  $v  \in Y_* +  \mathrm{Tan}_{y_*} \Gamma$. Here we also have used the limit  $\big<\h{1pt}\Pi_{\Gamma}\left(y_n\right), \h{0.5pt}y_*\h{1pt}\big> \longrightarrow 1$, as $n \rightarrow \infty$. Since (2.36) holds in the sense of trace for all three cases in (2.28), the above arguments induce \begin{eqnarray} v_n \h{1pt}\in\h{1pt}\Gamma, \h{20pt}\text{for large $n$ and $\mathscr{H}^1$-almost all points on $T_{\sigma_k}$.}
 \end{eqnarray}

 Similarly we define\begin{eqnarray}  v^{\infty}_n = \dfrac{\mathscr{V}_n^{\infty}}{|\h{0.5pt}\mathscr{V}_n^{\infty} \h{0.5pt}|} \h{10pt}\text{on $D_1^+$} \h{20pt}\text{with} \h{3pt} \mathscr{V}^{\infty}_n := \Gamma_n + R\h{1pt}\epsilon_n \h{1pt}\dfrac{ u_{\infty} - v_* - Y_* }{|\h{1pt} u_{\infty} - v_* - Y_* \h{1pt} | \vee R}.
\end{eqnarray}In light of (4)-(5) in Lemma 2.9, same arguments for the $v_n$ above infer \begin{eqnarray} v_n^{\infty}\big(\pm \sigma_k, 0 \h{1.5pt}\big) \h{1pt}\in\h{1pt} \Gamma, \h{20pt}\text{for large $n$.} \end{eqnarray}
Utilizing $v_n$ and $v_n^{\infty}$, for arbitrary $s \in \left(0, 1 \right)$, we define \begin{eqnarray} v_n^*\left(\xi\h{0.5pt}\right) = \left\{ \begin{array}{lcl}  v_n \left( \dfrac{\xi}{1 - s} \right), \h{20pt}&&\text{if $\xi \in D_{(1 - s) \h{0.5pt}\sigma_k}^+$;} \vspace{1pc}\\
\dfrac{\sigma_k - | \h{1pt}\xi \h{0.5pt} |}{s\h{1pt}\sigma_k} \h{2pt} v_n^{\infty}\left(\sigma_k \h{1pt}\dfrac{\xi}{|\h{1pt}\xi\h{1pt}|}\right) + \dfrac{|\h{0.5pt} \xi \h{0.5pt}| - (1- s)\h{0.5pt}\sigma_k}{ s \h{0.5pt}\sigma_k} \h{1pt}u\left(x_n + r_n \h{1pt}\sigma_k\h{1pt}\dfrac{\xi}{|\h{1pt}\xi\h{1pt}|}\right), && \text{if  $\xi \in D_{\sigma_k}^+ \sim D_{(1 - s)\h{1pt}\sigma_k}^+$.}
\end{array}\right.
\end{eqnarray}\noindent In light of (2) in Lemma 2.9, (2.25), (1) in (2.19), and the limit $\Pi_{\Gamma}\left(y_n\right) \rightarrow y_*$, the two sequences  $\Big\{ \h{1pt}u\left(x_n + r_n \h{1pt}\cdot \right)\Big\}$ and $\Big\{ \h{1pt}v_n^{\infty}\h{1pt} \Big\}$ uniformly converge to $y_*$ on $S_{\sigma_k}^+$.  Hence we can define \begin{eqnarray} u_n^* := \dfrac{v_n^*}{\left|\h{1pt} v_n^* \h{1pt}\right|} \h{10pt}\text{on $D_{\sigma_k}^+$.}
\end{eqnarray}For large $n$, it holds $u_n^* \in H^1 \left( D_{\sigma_k}^+; \h{1pt}\mathbb{S}^2 \right)$ with $u_n^* = u\left(x_n + r_n\h{1pt}\cdot \right)$ on $S_{\sigma_k}^+$. By (2.43) and  (3) in Lemma 2.9, we can take $n$ large enough and get  \begin{eqnarray} u_n^* \in \Gamma \h{10pt}\text{ on  $T_{\sigma_k} \sim T_{(1 - s) \h{1pt}\sigma_k}$.}\end{eqnarray} Here we also have used the uniform convergence of $\Big\{ \h{1pt}u\left(x_n + r_n \h{1pt}\cdot \right)\Big\}$ and $\Big\{ \h{1pt}v_n^{\infty}\h{1pt} \Big\}$ to $y_*$ on $S_{\sigma_k}^+$. With (2.41), (2.44)-(2.46), it follows $u_n^* \in \Gamma$,  for $\mathscr{H}^1$-almost all points on $T_{\sigma_k}$. \vspace{0.3pc}\\
\textbf{Step 2. Lower bound.} Note that $u$ minimizes the $E_{\mathbb{D}}$-energy within the configuration space: \begin{eqnarray*} \left\{ v \in H^1\big(\h{1pt} D_{\sigma_k\h{0.5pt}r_n}^+\left(x_n\right);\h{1pt}\mathbb{S}^2 \h{1pt}\big) \h{3pt}\Bigg|\h{3pt} \begin{array}{lcl} v = u &&\text{on the closed upper circle contained in $\p\h{0.5pt}D_{\sigma_k\h{0.5pt}r_n}^+\left(x_n\right)$\h{0.5pt};}\vspace{0.3pc}\\
v \h{1pt}\in\h{1pt} \Gamma &&\text{on the flat boundary of $D_{\sigma_k\h{0.5pt}r_n}^+\left(x_n\right)$\h{0.5pt}}
\end{array} \right\}.
\end{eqnarray*}By the minimality of $u$, it turns out \begin{eqnarray} &&\int_{D_{\sigma_k}^+} \left\{\h{1pt} \big| \h{1pt}D_{\xi} \h{1pt} u_n \h{1pt}\big|^2 + \left(\dfrac{r_n}{\epsilon_n}\right)^2F_2\h{0.5pt}\big(\h{0.5pt}x_{n, 1} + r_n \h{1pt} \xi_1, y_n + \epsilon_n\h{0.5pt}u_n \h{1.5pt} \big) \right\} \cdot \big(x_{n, 1} + r_n \h{1pt} \xi_1 \big) \\
\nonumber\\
\nonumber && \h{40pt}\leq \h{2pt}  \int_{D_{\sigma_k}^+} \left\{ \h{2pt}\left| \h{1pt} \h{1pt}D_{\xi} \h{1pt}\left(\dfrac{u_n^{*}}{\epsilon_n}\right) \h{1pt}\right|^2 + \left(\dfrac{r_n}{\epsilon_n}\right)^2 F_2\h{0.5pt}\big(\h{0.5pt}x_{n, 1} + r_n \h{1pt} \xi_1, u_n^{*} \h{1.5pt} \big) \h{1pt} \right\} \cdot \big(x_{n, 1} + r_n \h{1pt} \xi_1 \big).
\end{eqnarray}Here $F_2 = F_2(\rho, u) := \rho^{-2} \left( 4 u_1^2 + u_3^2\right) + \sqrt{2} \mu \left( 1 - 3 P(u)\right)$. Up to a subsequence, as $n \rightarrow \infty$, \begin{eqnarray} x_n \longrightarrow x_0 = \left(\h{1pt}\rho_0, 0 \h{1pt}\right), \h{10pt}\text{ for some $\rho_0 > 0$ }.\end{eqnarray}  By this limit, (2.17), (1) in (2.22) and lower-semi continuity, it follows \begin{eqnarray} \rho_0 \int_{D_{\sigma_k}^+} \big| \h{1pt}D_{\xi} \h{0.5pt} u_{\infty}\h{1pt}\big|^2 \h{1pt}\leq\h{1pt} \liminf_{n \rightarrow \infty} \int_{D_{\sigma_k}^+} \big|\h{1pt}D_{\xi} \h{0.5pt}u_n \h{1pt}\big|^2 \cdot \big(x_{n, 1} + r_n \h{1pt}\xi_1 \big).
\end{eqnarray}In light of  (2) in (2.19), it holds\begin{eqnarray} \left(\dfrac{r_n}{\epsilon_n}\right)^2 \int_{D_1^+} \h{1.5pt} F_2\h{1pt}\big(\h{0.5pt}x_{n, 1} + r_n \xi_1, y_n + \epsilon_n\h{0.5pt}u_n \h{1pt} \big) \h{3pt}\lesssim_{\mu, \h{1pt}\rho_*} \h{3pt}\left(\dfrac{r_n}{\epsilon_n}\right)^2  \longrightarrow 0, \h{10pt}\text{as $n \rightarrow \infty$.}
\end{eqnarray} Applying this limit and (2.49) to the left-hand side of (2.47), we get \begin{eqnarray} \rho_0 \int_{D_{\sigma_k}^+} \big| \h{1pt}D_{\xi} \h{0.5pt} u_{\infty}\h{1pt}\big|^2 \leq \liminf_{n \rightarrow \infty} \int_{D_{\sigma_k}^+} \left\{\h{1pt} \big| \h{1pt}D_{\xi} \h{0.5pt} u_n \h{1pt}\big|^2 + \left( \dfrac{r_n}{\epsilon_n} \right)^2 F_2\h{0.5pt}\big(x_{n, 1} + r_n \h{0.5pt} \xi_1, y_n + \epsilon_n\h{0.5pt}u_n \h{1pt} \big) \right\} \cdot \left(x_{n, 1} + r_n \h{0.5pt}\xi_1 \right).
\end{eqnarray}\\
\textbf{Step 3. Upper bound.} By the definition of $u_n^{*}$ in (2.45), it follows \begin{eqnarray} \int_{D_{\sigma_k}^+} \big|\h{1pt}D_{\xi}\h{0.5pt}u_n^{*} \h{1pt}\big|^2 = \int_{D_{\sigma_k}^+} \big|\h{1pt}D_{\xi} \h{0.5pt}v_n \h{1pt}\big|^2 + \int_{D_{\sigma_k}^+ \sim \h{0.5pt}D_{(1 - s) \h{0.5pt}\sigma_k}^+} \big|\h{1pt}D_{\xi} \h{0.5pt} u_n^{*}\h{1pt} \big|^2.
\end{eqnarray}Firstly we estimate the right-hand side above. In light of  (2.37), on $D_{\sigma_k}^+$, it satisfies\begin{eqnarray*} \big|\h{1pt}\mathscr{V}_n - y_*\h{1pt}\big|\h{1pt}\leq\h{1pt}\big|\h{1pt}\Gamma_n - y_* \h{1pt}\big| + R \h{1pt}\epsilon_n \longrightarrow 0, \h{10pt}\text{as $n \rightarrow \infty$.}
\end{eqnarray*}With the definition of $v_n$ in (2.38) and the last convergence, for $j = 1, 2$, there holds \begin{small} \begin{eqnarray*} \epsilon_n^{-1} \h{1pt}\p_{\h{0.5pt}\xi_j}v_n  &=&  \big| \h{1pt}\mathscr{V}_n \h{1pt}\big|^{-1} \h{1pt}\p_{\h{0.5pt}\xi_j}  \h{1pt} \dfrac{\big(\h{1pt} v - v_* - Y_* \big)\h{1pt}R}{\big|\h{1.5pt}v - v_* - Y_* \h{1.5pt}\big| \vee R } \h{2pt} - \h{2pt} \big| \h{1pt}\mathscr{V}_n \h{1pt}\big|^{-1} \h{1pt} \left< \h{1pt} v_n, \h{1pt} \p_{\h{0.5pt}\xi_j}  \h{1pt}\dfrac{\big( \h{1pt}v - v_* - Y_* \big)\h{1pt}R}{\big|\h{1pt}v - v_* - Y_* \h{1pt}\big| \vee R } \h{1pt}\right> v_n \\
\\
&\longrightarrow & \p_{\h{0.5pt}\xi_j} \h{1pt}  \dfrac{\big( \h{1pt}v - v_* - Y_* \big)\h{1pt}R}{\big|\h{1pt}v - v_* - Y_* \h{1pt}\big| \vee R} \h{2pt}-\h{2pt} \left< \h{1pt}y_*, \h{1pt}  \p_{\h{0.5pt}\xi_j} \h{1pt} \dfrac{\big( \h{1pt}v - v_* - Y_* \big)\h{1pt}R}{\big|\h{1pt}v - v_* - Y_* \h{1pt}\big| \vee R} \right> y_*, \h{30pt} \text{strongly in $L^2\left( D_{\sigma_k}^+\right)$}.
\end{eqnarray*}\end{small}\noindent Since $v - v_* - Y_* \in \mathrm{Tan}_{y_*} \mathbb{S}^2$  for almost all points in $D_{\sigma_k}^+$, the above limit infers \begin{eqnarray}\epsilon_n^{-1} \h{1pt}D_{\h{0.5pt}\xi} \h{1pt}v_n  &\longrightarrow & D_{\h{0.5pt}\xi} \h{1pt}  \dfrac{\big( \h{1pt}v - v_* - Y_* \big)\h{1pt}R}{\big|\h{1pt}v - v_* - Y_* \h{1pt}\big| \vee R}, \h{20pt} \text{strongly in $L^2\left( D_{\sigma_k}^+\right)$}.
\end{eqnarray}

By the polar coordinates $\left(\h{0.5pt}\tau, \Theta\h{0.5pt}\right)$ in the $\xi$-plane, the last integral in (2.52) can be represented as  \begin{eqnarray} \int_{D_{\sigma_k}^+ \sim \h{1pt}D_{(1 - s) \h{0.5pt}\sigma_k}^+} \big|\h{1pt}D_{\xi}\h{0.5pt}u_n^{*}\h{1pt} \big|^2 =  \int_{(1 - s)\h{0.5pt}\sigma_k}^{\sigma_k} \mathrm{d} \tau \int_0^{\pi} \mathrm{d} \h{0.5pt}\Theta \h{1pt}\left\{ \h{1pt} \tau\h{1pt}\big| \h{1pt}\p_{\tau} \h{0.5pt}u_n^{*} \h{1pt}\big|^2 + \dfrac{1}{\tau} \h{1pt}\big|\h{1pt}\p_{\h{0.5pt}\Theta} \h{0.5pt} u_n^{*} \h{1pt}\big|^2 \right\}.
\end{eqnarray}Using (2) in Lemma 2.9 and the definition (2.21), for any $\epsilon > 0$, we can take $n$ large enough and get \begin{eqnarray} \big|\h{1pt}u\left(x_n + r_n \cdot \h{1pt}\right) - y_n - \epsilon_n \h{1pt}u_{\infty}\h{1pt}\big| = \epsilon_n \h{1pt} \left| \h{1pt} u_n - u_{\infty}\h{1pt}\right| \h{2pt}<\h{2pt} \epsilon \h{1pt}\epsilon_n \h{20pt}\text{on $S_{\sigma_k}^+$.}
\end{eqnarray}Still by (2) in Lemma 2.9, $u_{\infty}$ is continuous and uniformly bounded on $S_{\sigma_k}^+$. In light of $\mathscr{V}_n^{\infty}$ defined in (2.42), we can take $R$ large enough so that  \begin{eqnarray} \mathscr{V}_n^{\infty} =  \Gamma_n + \epsilon_n \big( u_{\infty} - v_* - Y_* \big) \h{20pt}\text{on $S_{\sigma_k}^+$.}
\end{eqnarray}For \textit{I} or \textit{II} in (2.28), $\Gamma_n = \Pi_{\Gamma}\left(y_n\right)$ and $v_* = 0$. Applying (2.25) then yields \begin{eqnarray*} \mathscr{V}_n^{\infty} - y_n - \epsilon_n \h{0.5pt}u_{\infty} = \epsilon_n \h{0.5pt}Y_n  - \epsilon_n \h{0.5pt}Y_* = o \left( \epsilon_n \right) \h{10pt}\text{on $S_{\sigma_k}^+$.}
\end{eqnarray*}Here $o\left(\epsilon_n\right)$ denotes a quantity satisfying $o\left(\epsilon_n\right)\big/\h{0.5pt}\epsilon_n \rightarrow 0$, as $n \rightarrow \infty$.
For \textit{III} in (2.28), $\Gamma_n = y_*$ and $v_* = v_-$. By the limits in (2.25) and (2.31), we also have    \begin{eqnarray*} \mathscr{V}_n^{\infty} - y_n - \epsilon_n\h{0.5pt}u_{\infty} =   \epsilon_n \big(\h{1pt} Y_n - Y_* \h{1pt}\big) + \epsilon_n \left[ \dfrac{y_* - \Pi_{\Gamma}\left(y_n\right)}{\epsilon_n} - v_-\right] = o\left(\epsilon_n\right) \h{10pt}\text{on $S_{\sigma_k}^+$.}
\end{eqnarray*}Therefore in all cases, it holds\begin{eqnarray} \mathscr{V}_n^{\infty} = y_n + \epsilon_n \h{0.5pt}u_{\infty} + o\left(\epsilon_n\right) \h{10pt}\text{on $S_{\sigma_k}^+$.}
\end{eqnarray}Still by (2.56), we can calculate \begin{eqnarray*} \big|\h{1pt}\mathscr{V}_n^{\infty}\h{1pt}\big|^2 = 1 + 2 \h{1pt}\epsilon_n \h{0.5pt} \Big< \h{0.5pt}\Gamma_n, \h{1pt} u_{\infty} - v_* - Y_* \h{1pt} \Big> + o\left(\epsilon_n\right) \h{10pt}\text{on $S_{\sigma_k}^+$.}
\end{eqnarray*}Since $u_{\infty} - v_* - Y_* \in \mathrm{Tan}_{y_*}\h{1pt}\mathbb{S}^2$  on $S_{\sigma_k}^+$, then by $\Gamma_n \longrightarrow y_*$ as $n \rightarrow \infty$, the last equality can be reduced to \begin{eqnarray*}  \big|\h{1pt}\mathscr{V}_n^{\infty}\h{1pt}\big|^2 = 1 + 2 \h{1pt}\epsilon_n \h{0.5pt} \Big<\h{1pt} \Gamma_n - y_*, \h{1pt} u_{\infty} - v_* - Y_* \h{1pt}\Big> + o\left(\epsilon_n\right) = 1 + o\left(\epsilon_n\right) \h{10pt}\text{on $S_{\sigma_k}^+$,}
\end{eqnarray*}which furthermore yields \begin{eqnarray} \big|\h{1pt}\mathscr{V}_n^{\infty}\h{1pt}\big| = 1 + o\left(\epsilon_n\right) \h{20pt}\text{on $S_{\sigma_k}^+$.}
\end{eqnarray}By (2.42) and (2.57)-(2.58), we get \begin{eqnarray} v_n^{\infty} = y_n + \epsilon_n\h{0.5pt}u_{\infty} + o \left( \epsilon_n \right) \h{10pt}\text{on $S_{\sigma_k}^+$.}
\end{eqnarray}Hence for any $\epsilon > 0$, the following estimate holds \begin{eqnarray} \big| \h{1pt}v_n^{\infty} - y_n - \epsilon_n \h{0.5pt}u_{\infty} \h{1pt}\big| \h{1pt}\leq\h{1pt} \epsilon \h{0.5pt}\epsilon_n \h{10pt}\text{on $S_{\sigma_k}^+$,} \h{20pt}\text{provided $n$ is large.}\end{eqnarray}In light of this estimate and (2.55), by triangle inequality, we obtain \begin{eqnarray}  \big| \h{1pt}v_n^{\infty} - u(x_n + r_n \cdot \h{1pt}) \h{1.5pt}\big|\h{1pt}\leq\h{1pt} 2 \h{1pt}\epsilon \h{0.5pt}\epsilon_n \h{10pt}\text{on $S_{\sigma_k}^+$,} \h{20pt}\text{provided $n$ is large.}
\end{eqnarray}Utilizing (2.44)-(2.45) and (2.61) above, for any $\epsilon > 0$, we can take $n$ large enough and get \begin{eqnarray*}\big|\h{1pt} \p_{\tau} u_n^{*} \h{1pt}\big| \h{3pt}\lesssim\h{3pt} \dfrac{1}{s \h{0.5pt}\sigma_k} \h{2pt}\big\| \h{2pt} u(x_n + r_n \cdot \h{1pt}) - v_n^{\infty} \h{2pt} \big\|_{\infty; \h{0.5pt}S_{\sigma_k}^+}  \h{3pt}\lesssim\h{3pt} \dfrac{\epsilon \h{0.5pt}\epsilon_n}{s\h{0.5pt}\sigma_k}\h{30pt}\text{on $D_{\sigma_k}^+ \sim \h{1pt}D_{(1 - s)\h{0.5pt}\sigma_k}^+$.}
\end{eqnarray*}Therefore it follows \begin{eqnarray}\int_{(1 - s)\h{0.5pt}\sigma_k}^{\sigma_k} \mathrm{d} \tau \int_0^{\pi} \h{1pt}\big|\h{1.5pt}\p_{\tau} \h{0.5pt}u_n^{*} \h{1pt}\big|^2 \h{1pt}\tau \h{2pt}\mathrm{d} \Theta\h{3pt} \lesssim \h{3pt} \dfrac{\epsilon^2\h{0.5pt}\epsilon_n^2}{s \h{0.5pt}\sigma_k}.
\end{eqnarray}
Using (2.44), (2.21), (2.42) and (2.56), on $D_{\sigma_k}^+ \sim\h{1pt}D_{\left(1 - s \right)\sigma_k}^+$ we can rewrite $v_n^*$ as  \begin{eqnarray*} v_n^*\left(\tau, \Theta\right) = \dfrac{\sigma_k - \tau}{s\h{0.5pt}\sigma_k} \h{1pt} \dfrac{\Gamma_n + \epsilon_n \left(\h{0.5pt} u_{\infty}\big(\h{0.5pt} \sigma_k \h{1pt} e^{i\h{1pt} \Theta} \h{0.5pt}\big) - v_* - Y_* \h{0.5pt}\right)}{\big|\h{2pt}\Gamma_n + \epsilon_n \left(\h{0.5pt} u_{\infty}\big(\h{0.5pt} \sigma_k \h{1pt} e^{i\h{1pt} \Theta} \h{0.5pt} \big) - v_* - Y_* \h{0.5pt} \right) \h{1pt}\big|} + \dfrac{\tau - \left(1 - s \right)\sigma_k}{s \h{0.5pt}\sigma_k} \left(\h{1pt} y_n + \epsilon_n\h{1pt}u_n\big(\h{0.5pt} \sigma_k \h{1pt}e^{i\h{0.5pt} \Theta} \h{0.5pt}\big)\h{1pt}\right).
\end{eqnarray*}
In light of (2.45) and the above representation of $v_n^*$, it then turns out \begin{eqnarray*} \big|\h{1pt}\p_{\h{0.5pt}\Theta} \h{0.5pt}u_n^{*} \h{1pt}\big|\h{2pt}\bigg|_{\tau \h{0.5pt}e^{i\h{0.5pt}\Theta}} \h{2pt}\lesssim\h{2pt}  \sigma_k\h{1pt}\epsilon_n \left[ \h{2pt}\big| \h{1pt}D_{\xi} \h{0.5pt} u_{\infty} \h{1pt}\big|\h{2pt}\bigg|_{ \sigma_k \h{0.5pt}e^{i\h{0.5pt} \Theta}} + \big|\h{1pt}D_{\xi}\h{0.5pt}u_n \h{1pt}\big|\h{2pt}\bigg|_{ \sigma_k \h{0.5pt}e^{i\h{0.5pt}\Theta}} \right] \h{10pt}\text{on $D_{\sigma_k}^+ \sim \h{1pt}D_{(1 - s)\h{0.5pt}\sigma_k}^+$.}
\end{eqnarray*}Therefore by (1) in Lemma 2.9, we have  \begin{eqnarray*}\int_{(1 - s)\h{0.5pt}\sigma_k}^{\sigma_k} \mathrm{d} \tau  \int_0^{\pi} \h{1pt}\big|\h{1pt}\p_{\h{0.5pt}\Theta} \h{0.5pt}u_n^{*} \h{1pt}\big|^2 \h{1pt} \tau^{-1} \h{1pt}\mathrm{d} \Theta\h{1pt} \h{2pt}\lesssim\h{2pt}b_k\h{1.5pt}\epsilon_n^2 \h{1pt}\ln \left[ \dfrac{1}{1 - s} \right].
\end{eqnarray*}
 Applying the above estimate and (2.62) to right-hand side of (2.54) yield \begin{eqnarray} \int_{D_{\sigma_k}^+ \sim \h{1pt}D_{(1 - s) \h{0.5pt}\sigma_k}^+} \big|\h{1pt}D_{\xi} \h{1pt}u_n^{*}\h{1pt} \big|^2 \h{2pt}\lesssim \h{2pt} \dfrac{\epsilon^2\h{0.5pt}\epsilon^2_n}{s\h{0.5pt}\sigma_k} + b_k\h{1pt}\epsilon_n^2 \h{1pt}\ln \left[ \dfrac{1}{1 - s} \right].
\end{eqnarray}

Now we  multiply $\epsilon_n^{-2}$ on both sides of (2.52) and take $n \rightarrow \infty$. By (2.53) and  (2.63), it turns out \begin{eqnarray*} \limsup_{n \rightarrow \infty} \h{2pt} \epsilon_n^{-2} \h{1pt}\int_{D_{\sigma_k}^+} \big| \h{1pt}D_{\xi } \h{1pt}u_n^{*} \h{1pt}\big|^2 \h{2pt}\leq\h{2pt} \dfrac{c}{s \h{0.5pt}\sigma_k} \h{1pt}\epsilon^2 + c\h{1.5pt}b_k \h{1.5pt} \ln \left[ \dfrac{1}{1 - s} \right]  + \int_{D_{\sigma_k}^+} \left| \h{1pt}D_{\xi}  \h{1pt} \dfrac{\big( v - v_* - Y_* \big) \h{1pt}R}{\big|\h{1pt} v - v_* - Y_* \h{1pt}\big| \vee R } \h{1pt}\right|^2.
\end{eqnarray*}Here $c > 0$ is a constant. Moreover by taking $R \rightarrow \infty$, $\epsilon \rightarrow 0$ and $s \rightarrow 0$ successively, the above estimate induces \begin{eqnarray} \limsup_{n \rightarrow \infty} \h{2pt} \epsilon_n^{-2} \h{1pt}\int_{D_{\sigma_k}^+} \big| \h{1pt}D_{\xi } \h{1pt}u_n^{*} \h{1pt}\big|^2 \cdot \big(x_{n, 1} + r_n \h{1pt}\xi_1 \big) \h{2pt}=\h{2pt} \rho_0 \limsup_{n \rightarrow \infty} \h{1pt} \epsilon_n^{-2} \h{1pt}\int_{D_{\sigma_k}^+} \big| \h{1pt}D_{\xi } \h{1pt}u_n^{*} \h{1pt}\big|^2 \h{2pt}\leq\h{2pt} \rho_0 \int_{D_{\sigma_k}^+} \big|\h{1pt}D_{\xi} \h{0.5pt}v \h{1pt}\big|^2.
\end{eqnarray}Similarly as (2.50), it holds \begin{eqnarray*} \left(\dfrac{r_n}{\epsilon_n}\right)^2 \int_{D_1^+} \h{1.5pt} F_2\h{1pt}\big(\h{0.5pt}x_{n, 1} + r_n \xi_1, u_n^* \h{1pt} \big) \h{2pt}\lesssim_{\mu, \h{1pt}\rho_*} \h{2pt}\left(\dfrac{r_n}{\epsilon_n}\right)^2  \longrightarrow 0, \h{10pt}\text{as $n \rightarrow \infty$.}
\end{eqnarray*}Applying this estimate and (2.64) to the right-hand side of (2.47) yield \begin{eqnarray}\limsup_{n \rightarrow \infty} \int_{D_{\sigma_k}^+} \left\{ \h{2pt}\left| \h{1pt}D_{\xi} \h{0.5pt} \left(\dfrac{u_n^{*}}{\epsilon_n}\right) \h{1pt}\right|^2 + \left(\dfrac{r_n}{\epsilon_n}\right)^2\h{2pt} F_2\big(x_{n, 1} + r_n \h{0.5pt}\xi_1, u_n^* \h{1pt} \big) \h{2pt}\right\} \cdot \left(x_{n, 1} + r_n \h{0.5pt} \xi_1 \right) \h{2pt}\leq\h{2pt} \rho_0 \int_{D_{\sigma_k}^+} \big|\h{1pt}D_{\xi} \h{0.5pt}v\h{1pt}\big|^2.
\end{eqnarray}
\textbf{Step 4. Completion of proof.} In light of (2.51) and (2.65), we can take $n \rightarrow \infty$ in (2.47) and obtain \begin{eqnarray*} \int_{D_{\sigma_k}^+}\big|\h{1pt}D_{\xi} \h{1pt}u_{\infty}\h{1pt}\big|^2 \h{2pt}\leq\h{2pt} \int_{D_{\sigma_k}^+}\big|\h{1pt}D_{\xi} \h{1pt}v\h{1pt}\big|^2.  \end{eqnarray*}Since $\sigma_k \uparrow 1$ as $k \rightarrow \infty$, $u_{\infty}$ is an energy minimizing map in $\mathscr{F}_{A, \h{1pt}1}$ if \textit{I} or \textit{II} in (2.28) holds. If \textit{III} in (2.28) holds, then $u_{\infty}$ is an energy minimizing map in $\mathscr{F}_{B, \h{1pt}1}$. Moreover if we take $v$ in (2.65) to be $u_{\infty}$, then by (2.47), (2.50)-(2.51) and (2.65), it holds $D_{\xi} \h{0.5pt}u_n \rightarrow D_{\xi} \h{0.5pt} u_{\infty}$ strongly in $L^2\left(D_{\sigma_k}^+\right)$. Still by the limit $\sigma_k \uparrow 1$ as $k \rightarrow \infty$, the strong $H^1\left(D_r^+\right)$-convergence of $u_n$ follows for any $r \in (0, 1)$.
\end{proof}

With the above compactness result and characterization of $u_{\infty}$, we are ready to prove Lemma 2.2.
\begin{proof}[\bf Proof of Lemma 2.2] We only show the  case when $x \in T$ by continuing the above arguments. With some modifications, the interior case when $x$ is off $T$ can be similarly proved. By the definition of $u_n$ in (2.21), $\epsilon_n$ defined in (2.18) and (2) in (2.19), it can be shown that \begin{eqnarray} \int_{D_1^+}\big|\h{1pt}D_{\xi} \h{0.5pt}u_n\h{1pt}\big|^2 = \dfrac{1}{\epsilon_n^2} \int_{D_{r_n}^+\left(x_n\right)} | \h{1pt}D\h{0.5pt}u\h{1pt}|^2 = 1 - \left(\dfrac{r_n}{\epsilon_n}\right)^2 \longrightarrow 1, \h{10pt}\text{as $n \rightarrow \infty$.}
\end{eqnarray} Moreover utilizing (3) in (2.15), non-negativity of $F_2$ and change of variable, we have \begin{eqnarray}\int_{D_{\theta_1}^+} \big|\h{1pt}D_{\xi} \h{0.5pt}u_n \h{1pt}\big|^2 \cdot \left(x_{n, 1} + r_n\h{0.5pt}\xi_1\right) &+& \left(\dfrac{r_n}{\epsilon_n}\right)^2 \int_{D_{\theta_1}^+} F_2\h{1pt}\big(\h{1pt}x_{n, 1} + r_n\h{0.5pt}\xi_1,  y_n + \epsilon_n\h{0.5pt}u_n \h{1pt}\big)  \cdot \left(x_{n, 1} + r_n\h{0.5pt}\xi_1\right) \nonumber\\
\nonumber\\
\nonumber& > & \dfrac{1}{2} \h{1pt}\int_{D_{1}^+} \big|\h{1pt}D_{\xi} \h{0.5pt}u_n\h{1pt}\big|^2 \cdot \left(x_{n, 1} + r_n\h{0.5pt}\xi_1\right).
\end{eqnarray}In light of the strong $H^1(D_{\theta_1}^+)$-convergence of $u_n$ in Lemma 2.11, (2.50) and (2.66) above, we can take $n \rightarrow \infty$ in the last estimate and obtain \begin{eqnarray} \int_{D_{\theta_1}^+} \big|\h{1pt}D_{\xi}\h{0.5pt}u_{\infty} \h{1pt}\big|^2 \h{2pt}\geq\h{2pt} \dfrac{1}{2}.\end{eqnarray}

For any $v \in H^1\left(D_1^+;\h{1pt}Y_* + \mathrm{Tan}_{y_*}\h{1pt}\mathbb{S}^2 \right)$, there are two $H^1$ functions on $D_1^+$, denoted by $f^v_1$ and $f^v_2$, so that \begin{eqnarray*} v = Y_* + f^v_1 \h{1.5pt}v^+ + f^v_2\h{1.5pt} v^+ \times y_* \h{20pt}\text{ on $D_1^+$.} \end{eqnarray*} Here $v^+$ is the unit vector in $\mathrm{Tan}^+_{y_*} \h{1pt}\Gamma$. Therefore by Lemma 2.11,  if \textit{I} or \textit{II} in (2.28) holds, then it satisfies \begin{eqnarray} \int_{D_{1/2}^+} \left|\h{1pt}D_{\xi}\h{1pt}f_1^{u_{\infty}} \h{1pt}\right|^2 +  \left|\h{1pt}D_{\xi}\h{1pt}f_2^{u_{\infty}} \h{1pt}\right|^2 \h{2pt}\leq\h{2pt} \int_{D_{1/2}^+} \left|\h{1pt}D_{\xi}\h{1pt}f_1 \h{1pt}\right|^2 +  \left|\h{1pt}D_{\xi}\h{1pt}f_2 \h{1pt}\right|^2,
\end{eqnarray} for any $f_1$, $f_2$ $\in H^1\big(D_{1/2}^+ \big)$ with $f_1 = f_1^{u_{\infty}}$  on $S_{1/2}^+$ and $f_2 = f_2^{u_{\infty}}$ on $\p\h{0.5pt}D_{1/2}^+$ \h{0.5pt}. Notice that $f_2 = 0$ on $T_{1/2}$. Taking $f_2 = f_2^{u_{\infty}}$ on the right-hand side of (2.68), we have \begin{eqnarray} \int_{D_{1/2}^+} \left|\h{1pt}D_{\xi}\h{1pt}f_1^{u_{\infty}} \h{1pt}\right|^2  \h{2pt}\leq\h{2pt} \int_{D_{1/2}^+} \left|\h{1pt}D_{\xi}\h{1pt}f_1 \h{1pt}\right|^2, \h{10pt}\text{for any $f_1$ satisfying $f_1 \in H^1\big(\h{1pt}D_{1/2}^+\h{1pt}\big)$ and $f_1 = f_1^{u_{\infty}}$ on $S_{1/2}^+$.}
\end{eqnarray}Extending $f_1^{u_{\infty}}$ to the whole $D_{1/2}$ so that the extension, still denoted by $f_1^{u_{\infty}}$, is an even function with respect to the variable $\xi_2$, by the minimality of $f_1^{u_{\infty}}$ in (2.69), we can show that $f_1^{u_{\infty}}$ is a harmonic function on $D_{1/2}$. Applying Lemma 1.41 in [27] to $D_{\xi}\h{0.5pt}f_1^{u_{\infty}}$ then yields  \begin{eqnarray} \int_{D_{\theta_1}} \big|\h{1pt}D_{\xi}\h{1pt}f_1^{u_{\infty}} \h{1pt}\big|^2 \h{2pt}\lesssim\h{2pt}\theta_1^2 \int_{D_{1/2}} \big|\h{1pt}D_{\xi}\h{1pt}f_1^{u_{\infty}} \h{1pt}\big|^2,\h{10pt}\text{for any $\theta_1 < 1/2$.}
\end{eqnarray} Taking $f_1 = f_1^{u_{\infty}}$ on the right-hand side of (2.68), we have \begin{eqnarray} \int_{D_{1/2}^+} \left|\h{1pt}D_{\xi}\h{1pt}f_2^{u_{\infty}} \h{1pt}\right|^2  \h{2pt}\leq\h{2pt} \int_{D_{1/2}^+} \left|\h{1pt}D_{\xi}\h{1pt}f_2 \h{1pt}\right|^2, \h{2pt}\text{for any $f_2$ satisfying $f_2 \in H^1\big(\h{1pt}D_{1/2}^+\h{1pt}\big)$ and $f_2 = f_2^{u_{\infty}}$ on $\p\h{1pt}D_{1/2}^+$.}
\end{eqnarray}Since $f_2^{u_{\infty}} = 0 $ on $T_{1/2}$,  we can extend $f_2^{u_{\infty}}$ to the whole $D_{1/2}$ so that the extension, still denoted by $f_2^{u_{\infty}}$, is an odd function with respect to the variable $\xi_2$. By the minimality of $f_2^{u_{\infty}}$ given in (2.71), we can show that $f_2^{u_{\infty}}$ is a harmonic function on $D_{1/2}$. By Lemma 1.41 in [27], it turns out \begin{eqnarray} \int_{D_{\theta_1}} \big|\h{1pt}D_{\xi}\h{1pt}f_2^{u_{\infty}} \h{1pt}\big|^2 \h{2pt}\lesssim\h{2pt}\theta_1^2 \int_{D_{1/2}} \big|\h{1pt}D_{\xi}\h{1pt}f_2^{u_{\infty}} \h{1pt}\big|^2, \h{10pt}\text{for any $\theta_1 < 1/2$.}
\end{eqnarray} Utilizing (2.70), (2.72) and (2.66), we then obtain \begin{eqnarray} \int_{D_{\theta_1}^+} \big|\h{1pt}D_{\xi}\h{0.5pt}u_{\infty}\h{1pt}\big|^2 \h{2pt}\lesssim\h{2pt}\theta_1^2 \h{1pt}\int_{D_{1/2}^+} \big|\h{1pt}D_{\xi}\h{0.5pt}u_{\infty}\h{1pt}\big|^2 \h{2pt}\lesssim\h{2pt}\theta_1^2, \h{20pt}\text{for any $\theta_1 < 1/2$.}
\end{eqnarray}Therefore we can take $\theta_1$ suitably small so that (2.67) does not hold.

If \textit{III} in (2.28) holds, then for any $v \in H^1\left(D_1^+;\h{1pt}Y_* + \mathrm{Tan}_{y_*}\h{1pt}\mathbb{S}^2 \right)$, there are two $H^1$ functions on $D_1^+$, denoted by $f^v_1$ and $f^v_2$, so that \begin{eqnarray*} v = Y_* + v_- + f^v_1 \h{1.5pt}v^+ + f^v_2\h{1.5pt} v^+ \times y_* \h{20pt}\text{ on $D_1^+$.} \end{eqnarray*} Repeating the similar arguments for (2.69), we obtain  \begin{eqnarray*} \int_{D_{3/4}^+} \left|\h{1pt}D_{\xi}\h{1pt}f_1^{u_{\infty}} \h{1pt}\right|^2  \h{2pt}\leq\h{2pt} \int_{D_{3/4}^+} \left|\h{1pt}D_{\xi}\h{1pt}f_1 \h{1pt}\right|^2, \h{3pt}\text{for any $f_1$ satisfying $f_1 \in H^1\big(\h{1pt}D_{3/4}^+ \h{1pt}\big)$\h{0.5pt}; \h{1pt}$f_1 = f_1^{u_{\infty}}$ on $S_{3/4}^+$\h{0.5pt}; \h{1pt}$f_1 \geq 0$ on $T_{3/4}$.}
\end{eqnarray*}Here we have used the boundary condition on $T_{3/4}$ for an admissible map in $\mathscr{F}_{B,\h{0.5pt}3/4}$ (see (2.34)). Applying Lemma 9.1 in [53] to $f_1^{u_{\infty}}$, by Poincar\'{e} inequality, filling hole argument and standard iteration argument, we can find an universal constant $\alpha \in \big(0, 1\big)$ so that \begin{eqnarray*} \int_{D_{r}^+} \big|\h{1pt}D_{\xi}\h{1pt}f_1^{u_{\infty}}\h{1pt}\big|^2 \h{3pt}\lesssim\h{3pt}r^{\alpha} \h{1pt}\int_{D_{3/16}^+} \big|\h{1pt}D_{\xi}\h{1pt}f_1^{u_{\infty}}\h{1pt}\big|^2, \h{20pt}\text{for any $r \in \big(0, 3/16\big)$.}
\end{eqnarray*}As for $f_2^{u_{\infty}}$, we can also extend it to be an odd function with respect to the variable $\xi_2$. (2.72) still holds in this case. By (2.72) and the last estimate, it turns out \begin{eqnarray*} \int_{D_{\theta_1}^+} \big|\h{1pt}D_{\xi}\h{0.5pt}u_{\infty}\h{1pt}\big|^2 &=&\int_{D_{\theta_1}^+} \big|\h{1pt}D_{\xi}\h{0.5pt}f_1^{u_{\infty}}\h{1pt}\big|^2 + \int_{D_{\theta_1}^+} \big|\h{1pt}D_{\xi}\h{0.5pt}f_2^{u_{\infty}}\h{1pt}\big|^2 \\
\\
  &\lesssim& \theta_1^{\alpha}\h{1pt}\int_{D_{3/16}^+} \big|\h{1pt}D_{\xi}\h{1pt}f_1^{u_{\infty}}\h{1pt}\big|^2  + \theta_1^2 \int_{D_{1/2}^+} \big|\h{1pt}D_{\xi}\h{0.5pt}f_2^{u_{\infty}}\h{1pt}\big|^2 \h{2pt}\lesssim\h{2pt} \theta_1^{\alpha} \int_{D_{1/2}^+} \big|\h{1pt}D_{\xi}\h{0.5pt} u_{\infty} \h{1pt}\big|^2 \h{2pt}\lesssim\h{2pt}\theta_1^{\alpha},
\end{eqnarray*}for any $\theta_1 < 3/16$. Therefore we can take  $\theta_1$ suitably small so that (2.67) still fails in this case.
\end{proof}\
\\
\setcounter{section}{3}
\setcounter{theorem}{0}
\setcounter{equation}{0}
\textbf{\large III. \small REGULARITY AT THE ORIGIN}\vspace{1pc}\\
\noindent Due to Proposition 2.1, $L\h{1pt}\big[ \h{1pt}u_b^+\h{1pt}\big]$ and $L\h{1pt}\big[ \h{1pt}u_c^-\h{1pt}\big]$ solve the equation in (1.13) weakly in a neighborhood of $0$. Applying Rivi$\grave{\text{e}}$re-Struwe arguments in [54], we obtain the following $\epsilon$-regularity near the origin:
  \begin{prop}Suppose that $b$ and $c$ satisfy the same assumptions as in Proposition 2.1. Define $\mathcal{W}_{\h{1pt}b}^+ := L\h{1pt}\big[ \h{1pt}u_b^+\h{1pt}\big]$ and $\mathcal{W}_{\h{1pt}c}^- := L\h{1pt}\big[ \h{1pt}u_c^-\h{1pt}\big]$. There exists an $\epsilon_0 > 0$ so that $\mathcal{W}_{\h{1pt}b}^+$  is smooth at the origin, provided  \begin{eqnarray*} \liminf_{r \rightarrow 0} \h{2pt}r^{-1} \int_{B_r} \big|\h{1pt}\nabla \mathcal{W}_{\h{1pt}b}^+ \h{1pt}\big|^2 < \epsilon_0.
 \end{eqnarray*}As a map defined on $B_1$, $u_{\h{1pt}b}^+$ is also regular at the origin in the sense of Definition 1.2,  if the map $\mathcal{W}_{\h{1pt}b}^+$ is smooth at the origin. The same result holds for $\mathcal{W}_{\h{1pt}c}^-$ and $u_{\h{1pt}c}^-$.
\end{prop}\noindent In this section we improve the result in Proposition 3.1 by showing \begin{prop}Suppose $b \in \mathrm{I}_-$ and $c \in \mathrm{I}_+$ with $c \neq 0$. Then $\mathcal{W}_{\h{1pt}b}^+$ and $\mathcal{W}_{\h{1pt}c}^- $ are smooth at the origin.
\end{prop}Throughout the section we use $\mathcal{W}$ to denote either $\mathcal{W}_{\h{1pt}b}^+$ or $\mathcal{W}_{\h{1pt}c}^-$. To prove Proposition 3.2, we need to show that all tangent maps of $\mathcal{W}$ at $0$ are trivial. Since $u_b^+$ and $u_c^-$ are minimizers of the two minimization problems in (1.12), by applying arguments in [21], for almost all $R \in (0, 1)$, it holds \begin{eqnarray*}\int_{B^+_R} \left|\h{1pt} \nabla \mathcal{W} \h{1pt}\right|^2 + 3 \sqrt{2}\h{1pt} \mu \h{1pt} \big( 1 - 3 S\left[\h{1pt} \mathcal{W} \h{1pt}\right] \big) + 2 R \int_{\p^+ B_R} \left|\h{1pt} \p_{\nu} \mathcal{W} \h{1pt}\right|^2 = R \int_{\p^+ B_R}  \left|\h{1pt} \nabla \mathcal{W} \h{1pt}\right|^2 + \sqrt{2} \h{1pt}\mu \h{1pt}\big( 1 - 3 \h{1pt}S\left[\h{1pt} \mathcal{W} \h{1pt}\right] \big).
\end{eqnarray*} Here $\p^+ B_R$ denotes the spherical boundary of $B_R^+$. $\nu$ is the outer normal direction on $\p^+ B_R$. Utilizing the last energy equality then yields \begin{eqnarray}r^{-1} \int_{B^+_r} \left|\h{1pt} \nabla \mathcal{W} \h{1pt}\right|^2 +  \sqrt{2}\h{1pt} \mu \h{1pt} \big( 1 - 3 S\left[\h{1pt} \mathcal{W} \h{1pt}\right] \big) &+& 2 \int_{B_{\sigma}^+ \h{1pt}\sim\h{1pt} B_r^+}\dfrac{1}{|\h{1pt}x\h{1pt}|} \h{1.5pt} \big|\h{1pt}\p_{\nu} \mathcal{W} \h{1pt}\big|^2 \\
 \nonumber\\
\nonumber &\leq& \sigma^{-1} \int_{B^+_{\sigma}} \left|\h{1pt} \nabla \mathcal{W} \h{1pt}\right|^2 +  \sqrt{2}\h{1pt} \mu \h{1pt} \big( 1 - 3 S\left[\h{1pt} \mathcal{W} \h{1pt}\right] \big), \h{10pt}\text{for all $0 < r < \sigma \leq 1$.}
\end{eqnarray}
Supposing that  $\Big\{ k_j \Big\}$ is a positive decreasing sequence which converges to $0$ as $j \rightarrow \infty$, we rescale $\mathcal{W}$ by  \begin{eqnarray} \mathcal{W}^{(j)}\left(x\right) := \mathcal{W}\left( k_j \h{1pt}x \right), \h{20pt}\text{for any $x \in B_1$.}
\end{eqnarray}The map $\mathcal{W}^{(j)}$ is well-defined, provided $j$ is large. Moreover by (3.1), it turns out \begin{lemma}There exists a $\mathcal{W}^{\infty} \in H^1\left( B_1; \mathbb{S}^4\right)$ so that  up to a subsequence,  \begin{eqnarray} (1). \h{3pt} \mathcal{W}^{(j)} \longrightarrow \mathcal{W}^{\infty}, \h{10pt}\text{weakly in $H^1\left(B_1; \mathbb{S}^4\right)$;} \h{20pt}(2). \h{3pt} \mathcal{W}^{(j)} \longrightarrow \mathcal{W}^{\infty}, \h{10pt}\text{strongly in $L^2\left(B_1; \mathbb{S}^4\right)$.}
\end{eqnarray}The limiting map $\mathcal{W}^{\infty}$ is a $0$-homogeneous harmonic map from $B_1$ to $\mathbb{S}^4$ with energy satisfying \begin{eqnarray}\int_{B_1} \left|\h{1pt}\nabla \mathcal{W}^{\infty} \h{1pt}\right|^2  \h{2pt}<\h{2pt} 24 \h{0.5pt}\pi.
\end{eqnarray} Moreover in the sense of trace,  there holds \begin{eqnarray} \mathcal{W}_{3}^{\infty} \h{2pt}\geq\h{2pt}b \h{10pt}\text{on $T$}, \h{10pt}\text{if \h{1pt} $\mathcal{W} = \mathcal{W}_{\h{1pt}b}^+$;} \h{30pt}\mathcal{W}_{3}^{\infty} \h{2pt}\leq\h{2pt}c \h{10pt}\text{on $T$}, \h{10pt}\text{if \h{1pt} $\mathcal{W} = \mathcal{W}_{\h{1pt}c}^-$.}
\end{eqnarray}Here  $T$ is the flat boundary of $ B^+_1$. $\mathcal{W}_{3}^{\infty}$ is the third component of $\mathcal{W}^{\infty}$.
\end{lemma}\begin{proof}[\bf Proof] By (3.1)-(3.2), it follows \begin{eqnarray*} \int_{B_1} \left|\h{1pt}\nabla \mathcal{W}^{(j)} \h{1pt}\right|^2 \h{2pt}=\h{2pt} 2\h{0.5pt}k_j^{-1} \int_{B^+_{k_j}} \big|\h{1pt}\nabla \mathcal{W} \h{1pt}\big|^2   \h{2pt}\leq\h{2pt} 2 \int_{B_{1}^+} \big|\h{1pt}\nabla \mathcal{W} \h{1pt}\big|^2  + \sqrt{2}\h{1pt} \mu \h{1pt} \big( 1 - 3 S\left[\h{1pt} \mathcal{W} \h{1pt}\right] \big).
\end{eqnarray*}Utilizing  this uniform energy bound then yields (3.3). Moreover by lower semi-continuity, we can take $j \rightarrow \infty$ in the last estimate and get \begin{eqnarray*} \int_{B_1} \left|\h{1pt}\nabla \mathcal{W}^{\infty} \h{1pt}\right|^2 \h{2pt}\leq\h{2pt}\liminf_{j \rightarrow \infty} \int_{B_1} \left|\h{1pt}\nabla \mathcal{W}^{(j)} \h{1pt}\right|^2  \h{2pt}\leq\h{2pt} 2 \int_{B_{1}^+} \big|\h{1pt}\nabla \mathcal{W} \h{1pt}\big|^2  + \sqrt{2}\h{1pt} \mu \h{1pt} \big( 1 - 3 S\left[\h{1pt} \mathcal{W} \h{1pt}\right] \big).
\end{eqnarray*} Since $\mathcal{W} = L\left[\h{0.5pt}u\h{0.5pt}\right]$ with $u$ equaling to either $  u_{\h{1pt}b}^+ $ or $ u_{\h{1pt}c}^- $, then by (1.12), the last estimate can be reduced to \begin{eqnarray*}\int_{B_1} \left|\h{1pt}\nabla \mathcal{W}^{\infty} \h{1pt}\right|^2  \h{2pt}\leq\h{2pt}  \int_{B_1} \big|\h{1pt}\nabla \mathcal{W} \h{1pt}\big|^2  + \sqrt{2}\h{1pt} \mu \h{1pt} \big( 1 - 3 S\left[\h{1pt} \mathcal{W} \h{1pt}\right] \big) \h{2pt}=\h{2pt} E_{\mu}\left[\h{0.5pt} u \h{0.5pt}\right]\h{2pt}<\h{2pt}E_{\mu}\left[\h{0.5pt} U^* \h{0.5pt}\right] \h{2pt}=\h{2pt}24\pi.
\end{eqnarray*} (3.4) holds by the last estimate. As for (3.5), we notice that the trace of $\mathcal{W}_{3}$ on $T$ equals to the trace of $u_2$ on $T$. It then turns out \begin{eqnarray} \mathcal{W}_{3}^{(j)} \h{2pt}\geq\h{2pt}b \h{10pt}\text{on $T$}, \h{10pt}\text{if \h{1pt} $\mathcal{W} = \mathcal{W}_{\h{1pt}b}^+$;} \h{30pt}\mathcal{W}_{3}^{(j)}  \h{2pt}\leq\h{2pt}c \h{10pt}\text{on $T$}, \h{10pt}\text{if \h{1pt} $\mathcal{W} = \mathcal{W}_{\h{1pt}c}^-$.}
\end{eqnarray}By Theorem 6.2 in [50], we can keep extracting a subsequence, still denoted by $\left\{ \mathcal{W}^{(j)}\right\}$, so that \begin{eqnarray}\mathcal{W}^{(j)} \longrightarrow \mathcal{W}^{\infty}, \h{20pt}\text{strongly in $L^q \left( \p B_1^+ \right)$, for any $q \in \big[ \h{0.5pt}1, 4\h{0.5pt}\big)$.}
\end{eqnarray}Utilizing (3.7), we can take $j \rightarrow \infty$ in (3.6) and obtain (3.5). The fact that $\mathcal{W}^{\infty}$ is a $0$-homogeneous harmonic map from $B_1$ to $\mathbb{S}^4$ follows by a standard argument in [56].
\end{proof}

Since $\mathcal{W}^{\infty}$ is a $0$-homogeneous harmonic map from $B_1$ to $\mathbb{S}^4$, $\mathcal{W}^{\infty}$ can be represented by $ \mathcal{W}^{\infty} = L \left[\h{0.5pt}v\h{0.5pt}\right]$, where $v = (v_1, v_2, v_3)$ depends only on the polar angle $\varphi$ and  is a smooth $3$-vector field on $\big[\h{0.5pt}0, \pi\h{0.5pt}\big]$ with unit length.   Moreover the three components of $v$  satisfy the following ODE system: \begin{eqnarray} - \big( \sin \varphi \h{2pt} v' \big)' + \dfrac{1}{\sin \varphi} \left(\begin{array}{ccc} 4 v_1 \\
0\\
v_3 \end{array}\right)  \h{2pt}=\h{2pt} \bigg\{ |\h{1pt}v' \h{1pt}|^2 \h{1.5pt}\sin \varphi  + \dfrac{1}{\sin \varphi} \big( 4 v_1^2 + v_3^2 \big) \bigg\} v  \h{20pt}\text{on  $ \left(\h{0.5pt}0, \pi \h{0.5pt}\right)$.}
\end{eqnarray}Here $'$ denotes the derivative with respect to the variable $\varphi$. Utilizing the system (3.8), we have \begin{lemma}If $\mathcal{W}  = \mathcal{W}_{\h{1pt}b }^+$, then the map $\mathcal{W}^{\infty}$ equivalently equals to the constant vector $N^* := (0, 0, 1, 0, 0)$ on $B_1$. If $\mathcal{W} = \mathcal{W}_{\h{1pt}c, }^-$, then the map $\mathcal{W}^{\infty}$ equivalently equals to $- N^*$ on $B_1$.
\end{lemma}
\begin{proof}[\bf Proof] Taking inner product with $v'$ on both sides of (3.8) yields \begin{eqnarray*} \left|\h{1pt}v'\h{1pt}\right|^2 \h{1pt} \sin^2 \varphi - \left( 4 v_1^2 + v_3^2 \right) = \mathrm{const}, \h{20pt}\text{for any $\varphi \in \big[\h{0.5pt}0, \pi \h{0.5pt}\big]$.}
\end{eqnarray*}Since $v$ is smooth on $\big[\h{0.5pt}0, \pi \h{0.5pt}\big]$ and $v_1 = v_3 = 0$ at $\varphi = 0$, $\mathrm{const}$ in the last equality equals to $0$. It turns out \begin{eqnarray}\left|\h{1pt}v'\h{1pt}\right|^2 \h{1pt} \sin^2 \varphi =  4 v_1^2 + v_3^2, \h{20pt}\text{for any $\varphi \in \big[\h{0.5pt}0, \pi \h{0.5pt}\big]$.}
\end{eqnarray}Note that  $v_1\left(\varphi\right) = v_1\left( \pi - \varphi \right)$, $v_2\left(\varphi\right) = v_2\left( \pi - \varphi \right)$, $v_3\left(\varphi\right) = - v_3\left( \pi - \varphi \right)$, for any $\varphi \in \big[\h{0.5pt}0, \pi \h{0.5pt}\big]$. Then  \begin{eqnarray}v_1' = v_2' = 0 \h{20pt}\text{at $\varphi = \dfrac{\pi}{2}$.}
\end{eqnarray}Evaluating (3.9) at $\varphi = \dfrac{\pi}{2}$ and using (3.10), we obtain \begin{eqnarray}v_3' = \pm \h{1.5pt}2 \h{1pt} v_1 \h{20pt}\text{at $\varphi = \dfrac{\pi}{2}$.}
\end{eqnarray} Here we also used the boundary condition $v_3 = 0$ at $\varphi = \pi/ 2$. Since  \begin{eqnarray} v_1 \geq 0, \h{10pt}v_3 \geq 0 \h{20pt}\text{on $\big[\h{1pt}0, \pi/2 \h{1pt}\big]$,}
\end{eqnarray}it holds $v_3' \leq 0$ at $\pi / 2$. The sign in (3.11) can be determined. That is \begin{eqnarray}v_3' = - 2 \h{1pt} v_1 \h{20pt}\text{at $\varphi = \dfrac{\pi}{2}$.}
\end{eqnarray}

Suppose that $v_1 = 0$ at some $\varphi_0 \in (\h{0.5pt}0, \pi / 2 \h{0.5pt})$. By (3.12), $\varphi_0$ is a minimum point of $v_1$ on $(\h{0.5pt}0, \pi / 2 \h{0.5pt})$. It follows $v_1' = 0$ at $\varphi_0$. Applying uniqueness result for ordinary differential equation to the equation of $v_1$ in (3.8) then yields $v_1 \equiv 0$ on $\left[ \h{0.5pt}0, \pi  \h{0.5pt} \right]$. In light of (3.13), $v_3' = 0$ at $\varphi = \pi / 2$. Since we have $v_3 = 0$ at $\varphi = \pi / 2$, we can also apply uniqueness result for ordinary differential equation to the equation of $v_3$ in (3.8) and obtain $v_3 \equiv 0$ on $\left[ \h{0.5pt}0, \pi  \h{0.5pt} \right]$. In this case we have $\mathcal{W}^{\infty}$ equals to either $N^*$ or $- N^*$ on $B_1$.

Similarly if it holds $v_3 = 0$ at some $\varphi_0 \in (\h{0.5pt}0, \pi / 2 \h{0.5pt})$, then we have $v_3 \equiv 0$ on $\left[ \h{0.5pt}0, \pi  \h{0.5pt} \right]$. Still by (3.13), it follows $v_1 = 0$ at $\varphi = \pi / 2$. From (3.10), it satisfies $v_1' = 0$ at $\varphi = \pi / 2$. We then can still infer $v_1 \equiv 0$ on $\left[ \h{0.5pt}0, \pi  \h{0.5pt} \right]$. In this case $\mathcal{W}^{\infty}$ also equals to either $N^*$ or $- N^*$ on $B_1$.

Now we assume $v_1 > 0$ and $v_3 > 0$ on $(\h{0.5pt}0, \pi / 2 \h{0.5pt})$. Since $v_2 = 1 $ or $- 1$ at $\varphi = 0$, by the regularity of $v_2$ and the positivity of $v_1$ and $v_3$ on $(\h{0.5pt}0, \pi / 2 \h{0.5pt})$, we can find a $\varphi_1 \in (\h{0.5pt}0, \pi / 2 \h{0.5pt})$ so that $v_1$, $v_2$ and $v_3$ are all non-zero at $\varphi_1$. Therefore the image of $\mathcal{W}^{\infty}$ restricted on $\big(\varphi_1, \theta \h{0.5pt}\big)$, where $\theta \in \left[\h{0.5pt}0, 2 \pi \h{0.5pt}\right]$, can not be contained in a totally geodesic $\mathbb{S}^3 \subset \mathbb{S}^4$. Here we also used $\mathcal{W}^\infty = L\h{1pt}\left[\h{1pt} v \h{1pt}\right]$ and the definition of the $L$-operator. This infers that $\mathcal{W}^{\infty}$, when restricted on $\mathbb{S}^2$, is a linearly full harmonic map from $\mathbb{S}^2$ to $\mathbb{S}^4$. By Corollary 1.20 in [39], the twistor degree $d$ of $\mathcal{W}^{\infty}$ must satisfy $d \geq 3$. It  follows $\displaystyle \int_{B_1}\left|\h{1pt}\nabla \mathcal{W}^{\infty}\h{1pt}\right|^2 \h{2pt}=\h{2pt}8 \h{0.5pt}\pi\h{0.5pt}d \h{2pt}\geq\h{2pt}24\h{1pt}\pi$. This estimate violates (3.4).

Hence $\mathcal{W}^{\infty} = N^*$ or $- N^*$. The lemma then follows by (3.5).
\end{proof}
\noindent \begin{proof}[\bf Proof of Proposition 3.2] Recall that $u_b^+$ and $u_c^-$ are minimizers of the two minimization problems in (1.12). The proof of Proposition 3.2 can be obtained by Proposition 3.1, Lemma 3.4 and a similar construction as in the proof of convergence theorem 5.5 in [31] (see also Sect.4 in [29]).\end{proof}\
\\
\setcounter{section}{4}
\setcounter{theorem}{0}
\setcounter{equation}{0}
\textbf{\large IV. \small SINGULARITY: ITS TANGENT MAP AND STRUCTURE}\vspace{1pc}\\
Denote by $x_0$ the point $\big(\h{0.5pt} 0, 0, z_0\h{0.5pt}\big)$, where $z_0 \in ( 0, 1)$. Let $\big\{k_j\big\}$ be a positive decreasing sequence which converges to $0$ as $j \rightarrow \infty$. If $j$ is sufficiently large, then the map \begin{eqnarray}w^{(j)}(x) := \mathcal{W}\h{0.5pt}\big(x_0 + k_j\h{1pt} x\big), \h{20pt}\text{ $x \in B_1$}\end{eqnarray} is well-defined. Here we still use $\mathcal{W}$ to denote either $L\h{1pt}\big[\h{0.5pt}u_{\h{1pt}b}^+\h{0.5pt}\big]$ or $L\h{1pt}\big[\h{0.5pt}u_{\h{1pt}c}^-\h{0.5pt}\big]$. Similar arguments as the proof for (3.3) implies that there exists a $w^{\infty} \in H^1\left( B_1; \mathbb{S}^4\right)$ so that up to a subsequence   \begin{eqnarray} (1). \h{3pt} w^{(j)} \longrightarrow w^{\infty} \h{10pt}\text{weakly in $H^1\left(B_1; \mathbb{S}^4\right)$;} \h{20pt}(2). \h{3pt} w^{(j)} \longrightarrow w^{\infty} \h{10pt}\text{strongly in $L^2\left(B_1; \mathbb{S}^4\right)$.}
\end{eqnarray}Moreover  $w^{\infty}$ is a $0$-homogeneous harmonic map from $B_1$ to $\mathbb{S}^4$. By a similar construction as in the proof of convergence theorem 5.5 in [31], it holds, for all $R \in (0, 1)$, that \begin{eqnarray}\limsup_{j \rightarrow \infty}\h{1.5pt} k_j^{-1} \int_{B_{k_j R}(x_0)} \left|\h{1pt}\nabla \mathcal{W} \h{1pt}\right|^2 \h{2pt}\leq\h{2pt} \int_{B_R} \left|\h{1pt}\nabla w^{\infty}\h{1pt}\right|^2.
\end{eqnarray}On the other hand lower semi-continuity and (1) in (4.2) infers \begin{eqnarray*}\liminf_{j \rightarrow \infty} k_j^{-1} \int_{B_{k_j R}(x_0)} \left|\h{1pt}\nabla \mathcal{W} \h{1pt}\right|^2 \h{2pt}=\h{2pt} \liminf_{j \rightarrow \infty} \int_{B_R} \left|\h{1pt}\nabla w^{(j)} \h{1pt}\right|^2 \h{2pt}\geq\h{2pt} \int_{B_R} \left|\h{1pt}\nabla w^{\infty} \h{1pt}\right|^2.
\end{eqnarray*}Combining the last lower bound and (4.3), we get $w^{(j)}\longrightarrow w^{\infty}$,  strongly in $H^1\left(B_R; \mathbb{S}^4\right)$, as $j \rightarrow \infty$. Since $R$ is an arbitrary number in $(0, 1)$, it then holds \begin{eqnarray} w^{(j)}\longrightarrow w^{\infty} \h{20pt}\text{strongly in $H^1_{\mathrm{loc}}\left(B_1; \mathbb{S}^4\right)$, as $j \rightarrow \infty$.}
\end{eqnarray}
\
\\
\textbf{\large IV.1. \small MINIMALITY OF THE LIMITING MAP}\vspace{1pc}\\
In this section we study the minimality of $w^{\infty}$ with respect to the Dirichlet energy. The main result is \begin{prop}For any  $R \in (\h{0.5pt}0, 1 \h{0.5pt})$, $w^{\infty}$ saturates the minimal value of the following minimization problem: \begin{eqnarray} \mathrm{Min} \left\{ \h{2pt} \int_{B_R} \big|\h{1pt}\nabla w \h{1pt}\big|^2 \h{3pt}\Bigg|\h{3pt}\begin{array}{lcl} w \in H^1\big(B_R; \h{1pt}\mathbb{S}^4\big); \h{20pt}w = w^{\infty} \h{5pt}\text{on $\p B_R$;} \vspace{0.5pc}\\
w = L\left[\h{0.5pt}u\h{0.5pt}\right], \h{2pt}\text{for some $3$-vector field $u = u\left(r,  \varphi \right)$} \end{array} \right\}.
\end{eqnarray}
\end{prop}\noindent Before proving Proposition 4.1, we introduce some notations. Given a $3$-vector field $u = \big(u_1, u_2, u_3\big)$ on $\big[\h{0.5pt}0, \pi \h{0.5pt}\big]$, we use $e_2\left[\h{0.5pt}u\h{0.5pt}\right]$ to denote the energy density $ \left| \p_{\varphi} u \h{1pt}\right|^2\sin \varphi   + \dfrac{1}{\sin \varphi} \left( 4 u_1^2 + u_3^2 \right)$. The $\mathrm{E}_2$-energy of $u$ is the integration of $e_2\left[\h{0.5pt}u\h{0.5pt}\right]$ on $\big[\h{0.5pt}0, \pi \h{0.5pt}\big]$. Firstly we control the $L^{\infty}$-norm of $u$ in terms of its $\mathrm{E}_2$-energy. That is \begin{lemma}For any $3$-vector field $u$ on $\big[\h{0.5pt}0, \pi \h{0.5pt}\big]$ with finite $\mathrm{E}_2$-energy, it satisfies \begin{eqnarray}  u^2_1\left(\varphi\right) + u^2_3\left(\varphi\right) \h{2pt}\leq\h{2pt}2 \int_0^{\varphi} e_2\left[u\h{0.5pt}\right], \h{20pt}\text{for any $\varphi \in \big[\h{0.5pt}0, \pi \h{0.5pt}\big]$.}
\end{eqnarray}If in addition $| u \h{0.5pt}| = 1$ on $\big[\h{0.5pt}0, \pi \h{0.5pt}\big]$, then $u_2$ satisfies  \begin{eqnarray}  \bigg|\h{1pt} u_2(\varphi_2) - u_2(\varphi_1) \h{1pt}\bigg|  \h{2pt}\leq \h{2pt} \int_{\varphi_1}^{\varphi_2} e_2\big[\h{0.5pt}u\h{0.5pt}\big], \h{20pt}\text{for any $ 0 \leq \varphi_1 < \varphi_2 \leq \pi$.}
\end{eqnarray}
\end{lemma}The proof is standard, we omit it here. Now we prove Proposition 4.1.
\begin{proof}[\bf Proof of Proposition 4.1] Recalling $w^{(j)}$ defined in (4.1), we can find a $3$-vector field $u^{(j)} = u^{(j)}\left(r, \varphi\right)$ on $B_1$ so that $w^{(j)} = L\big[\h{1pt}u^{(j)}\h{1pt}\big]$. Moreover $w^{\infty} = L\big[\h{1pt}u^{\infty}\h{1pt}\big]$ for some $3$-vector field $u^{\infty} = u^{\infty}\left(\varphi\right)$ on $[\h{1pt}0, \pi\h{1pt}]$. By the strong convergence (4.4), for any $R \in (\h{0.5pt}0, 1)$ fixed, there is a $\sigma_2 \in (R, 1 )$ so that \begin{eqnarray} \int_0^{\pi} e_2\left[u^{(j)} - u^{\infty}\right] \h{1pt}\mathrm{d}\h{0.5pt}\varphi \h{2pt}\bigg|_{r = \sigma_2} \longrightarrow 0, \h{20pt}\text{as $j \rightarrow \infty$.}
\end{eqnarray}Then up to a subsequence, \begin{eqnarray} u^{(j)} \longrightarrow u^{\infty} \h{20pt}\text{pointwisely on $\bigg\{\left(\sigma_2, \varphi\right): \varphi \in (\h{0.5pt}0, \pi\h{0.5pt})$ \bigg\}.}
\end{eqnarray}Utilizing (4.6) and the limit (4.8), we obtain \begin{eqnarray}\left(u_1^{(j)} - u_1^{\infty}\right)^2 + \left(u_3^{(j)} - u_3^{\infty}\right)^2 \longrightarrow 0 \h{20pt}\text{uniformly on $\bigg\{\big(\sigma_2, \varphi\big): \varphi \in \big[\h{0.5pt}0, \pi\h{0.5pt}\big] \h{1pt}\bigg\}$.}
\end{eqnarray}Let  $\varphi_*$ be an arbitrary number in $(\h{0.5pt}0, \pi / 2 )$. By (4.7), it follows \begin{eqnarray*} \left|\h{1pt}u_2^{(j)}\left(\sigma_2, \varphi\right) - u_2^{(j)}\left(\sigma_2, \varphi_*\right) \h{1pt}\right| \h{2pt}\leq\h{2pt}\int_0^{\varphi_*} e_2\left[u^{(j)}\right] \h{2pt}\bigg|_{r = \sigma_2}, \h{20pt}\text{for any $\varphi \in \big[\h{0.5pt}0, \varphi_*\big]$.}
\end{eqnarray*} In light of (4.8)-(4.9) and taking $j \rightarrow \infty$ in the last estimate, we then obtain \begin{eqnarray*} \big|\h{1pt}u_2^{\infty}\left( \varphi\right) - u_2^{\infty}\left(\varphi_*\right) \h{1pt}\big| \h{2pt}\leq\h{2pt}\int_0^{\varphi_*} e_2\left[u^{\infty}\right], \h{20pt}\text{for any $\varphi \in \big(\h{0.5pt}0, \varphi_*\big]$.}
\end{eqnarray*}The last two estimates then yield \begin{small}\begin{eqnarray} \left\|\h{1pt}u_2^{(j)}\left(\sigma_2, \cdot\right) - u_2^{\infty}\left(\cdot\right)\h{1pt}\right\|_{\infty; \h{1pt} \left(\h{0.5pt}0, \h{1pt}\varphi_*\h{0.5pt}\right]}  \h{2pt}\leq\h{2pt}  \left|\h{1pt}u_2^{(j)}\left(\sigma_2, \varphi_*\right) - u_2^{\infty}\left(\varphi_*\right)\h{0.5pt}\right| +  \int_0^{\varphi_*} e_2\left[u^{\infty}\right] + \int_0^{\varphi_*} e_2\left[u^{(j)}\right]\h{2pt}\bigg|_{r = \sigma_2}.
\end{eqnarray}\end{small}\noindent Here and in what follows, for a function defined on an interval $I$, we use $\| \cdot \|_{\infty; \h{1pt}I}$ to denote its $L^{\infty}$-norm on the interval $I$. Similarly it also holds \begin{small}\begin{eqnarray} \left\|\h{1pt}u_2^{(j)}\left(\sigma_2, \cdot\right) - u_2^{\infty}\left( \cdot\right) \h{1pt}\right\|_{\infty;\h{1pt}\left[\h{0.5pt}\pi - \varphi_*, \pi\h{0.5pt}\right)} \leq \left|\h{1pt}u_2^{(j)}\left(\sigma_2, \cdot \right) - u_2^{\infty}\left(\cdot\right)\h{1pt}\right|\h{2pt}\bigg|_{\pi - \varphi_*} + \int_{\pi - \varphi_*}^{\pi} e_2\left[u^{\infty}\right] +  \int_{\pi - \varphi_*}^{\pi} e_2\left[u^{(j)}\right]\h{2pt}\bigg|_{r = \sigma_2}.
\end{eqnarray}\end{small}\noindent As for  $\varphi_1 \in \big[\h{0.5pt}\varphi_*, \h{1pt}\pi - \varphi_*\h{0.5pt}\big]$, we have the lower bound $\sin \varphi_1 \geq \sin \varphi_*$. It induces \begin{small}\begin{eqnarray*}\left|\h{1pt}u_2^{(j)}\left(\sigma_2, \varphi_1\right) - u_2^{\infty}\left(\varphi_1\right)\h{1pt}\right| &\leq&\left|\h{1pt}u_2^{(j)}\left(\sigma_2, \varphi_*\right) - u_2^{\infty}\left(\varphi_*\right)\h{1pt}\right| \h{2pt}+\h{2pt}\dfrac{1}{\sin^{\frac{1}{2}}\varphi_*} \int_{\varphi_*}^{\varphi_1} \left| \h{1pt}\p_{\varphi}u_2^{(j)} - \p_{\varphi} u_2^{\infty} \h{1pt}\right| \sin^{\frac{1}{2}}\varphi \h{2pt}\bigg|_{r = \sigma_2} \\
\\
&\leq&\left|\h{1pt}u_2^{(j)}\left(\sigma_2, \varphi_*\right) - u_2^{\infty}\left(\varphi_*\right)\h{1pt}\right| \h{2pt}+\h{2pt}\left(\dfrac{\pi}{\sin \varphi_*}\right)^{1/2} \left(\int_{0}^{\pi} \left| \h{1pt}\p_{\varphi}u_2^{(j)} - \p_{\varphi} u_2^{\infty} \h{1pt}\right|^2 \sin \varphi \right)^{1/2} \bigg|_{r = \sigma_2}.
\end{eqnarray*}\end{small}\noindent Therefore it holds   \begin{small}\begin{eqnarray*}\left\|\h{1pt}u_2^{(j)}\left(\sigma_2, \cdot\right) - u_2^{\infty}\left(\cdot\right)\h{1pt}\right\|_{\infty; \h{1pt}\left[\h{0.5pt}\varphi_*, \pi - \varphi_*\h{0.5pt}\right]} \leq\left|\h{1pt}u_2^{(j)}\left(\sigma_2, \varphi_*\right) - u_2^{\infty}\left(\varphi_*\right)\h{1pt}\right| \h{2pt}+\h{2pt}\left(\dfrac{\pi}{\sin \varphi_*}\right)^{1/2} \left(\int_{0}^{\pi} e_2\left[ \h{0.5pt}u^{(j)} - u^{\infty} \h{0.5pt}\right] \right)^{1/2} \h{2pt}\bigg|_{r = \sigma_2}.
\end{eqnarray*}\end{small}\noindent Using the last estimate, (4.11)-(4.12) and the limits in (4.8)-(4.9), we get \begin{eqnarray*} \limsup_{j \rightarrow \infty} \left\|\h{1pt}u_2^{(j)}\left(\sigma_2, \cdot \h{0.5pt}\right) - u_2^{\infty}\left(\cdot\right)\h{1pt}\right\|_{\infty; \h{1pt}\left(\h{0.5pt}0, \pi\h{0.5pt}\right)} \h{2pt}\leq\h{2pt} 2  \int_0^{\varphi_*} e_2\left[u^{\infty}\right] + 2  \int_{\pi - \varphi_*}^{\pi} e_2\left[u^{\infty}\right] \longrightarrow 0, \h{10pt}\text{as $\varphi_* \rightarrow 0^+$.}
\end{eqnarray*}By (4.10) and the above limit, it turns out \begin{eqnarray}u^{(j)}\left(\sigma_2, \cdot \h{0.5pt}\right) \longrightarrow u^{\infty}\left(\cdot\right) \h{20pt}\text{uniformly on $\big[\h{0.5pt}0, \pi\h{0.5pt}\big]$, as $j \rightarrow \infty$}.\end{eqnarray}

Suppose that  $w^* = L\left[\h{0.5pt}u\h{0.5pt}\right]$ for some $u = u\left(r, \varphi\right)$ on $B_R$.  Moreover we assume $w^* \in H^1\left( B_R; \h{1pt}\mathbb{S}^4 \right)$ and satisfies, in the sense of trace, $w^* = w^{\infty}$ on $\p B_R$. We need to show \begin{eqnarray} \int_{B_R} \left|\h{1pt}\nabla w^{\infty} \right|^2 \h{2pt}\leq\h{2pt} \int_{B_R} \left|\h{1pt} \nabla w^* \h{1pt}\right|^2.
\end{eqnarray}Firstly we extend $w^*$ from $B_R$ to $B_{\sigma_2}$ by setting \begin{eqnarray}\widehat{w}^* = w^* \h{10pt}\text{on $B_R$;} \h{30pt} \widehat{w}^* = w^{\infty} \h{10pt}\text{ on $B_{\sigma_2} \sim B_R$}.\end{eqnarray}\noindent  Fixing a  $\sigma_1 \in \big(R, \sigma_2 \big)$, we define  \begin{eqnarray}\widetilde{w}^{(j)} = \widehat{w}^* \h{20pt}\text{ on $B_{\sigma_1}$;} \h{30pt}  \widetilde{w}^{(j)} = w^{\infty} + \dfrac{r - \sigma_1}{\sigma_2 - \sigma_1} \left( w^{(j)}(\sigma_2, \cdot) - w^{\infty} \right) \h{20pt}\text{on $B_{\sigma_2} \sim B_{\sigma_1}$.}
\end{eqnarray}In light of (4.13), the distance between $\widetilde{w}^{(j)}$ and $\mathbb{S}^4$ is uniformly small on $B_{\sigma_2}$, provided $j$ is sufficiently large. Therefore we can define our comparison map \begin{eqnarray} \widehat{w}^{(j)} = \dfrac{\widetilde{w}^{\h{1pt}(j)}}{\left|\h{1pt} \widetilde{w}^{\h{1pt}(j)} \h{1pt}\right|} \h{20pt}\text{ on $B_{\sigma_2}$.}\end{eqnarray} Moreover  it holds \begin{eqnarray} \int_{B_{\sigma_2} \sim B_{\sigma_1}} \left| \h{1pt}\nabla \widehat{w}^{(j)} \h{1pt}\right|^2 \h{3pt}\lesssim\h{3pt}  \int_{B_{\sigma_2} \sim B_{\sigma_1}} \left| \h{1pt}\nabla \widetilde{w}^{(j)} \h{1pt}\right|^2, \h{10pt}\text{provided $j$ is sufficiently large.}
\end{eqnarray}By (4.16), the energy of $\widetilde{w}^{(j)}$ on $B_{\sigma_2} \sim B_{\sigma_1}$ can be estimated by \begin{eqnarray*} \int_{B_{\sigma_2} \sim B_{\sigma_1}} \left| \h{1pt}\nabla \widetilde{w}^{(j)} \h{1pt}\right|^2 \h{3pt}\lesssim \h{3pt}\dfrac{1}{\sigma_2 - \sigma_1} \h{1pt} \big\|\h{1pt}w^{(j)} - w^{\infty} \h{1pt}\big\|^2_{\infty;\h{1pt}\p B_{\sigma_2}}  + \big(\sigma_2 - \sigma_1 \big) \int_{0}^{\pi}  e_2 \big[ u^{\infty} \big] +  e_2 \left[ u^{(j)} \right]  \h{2pt}\mathrm{d}\h{0.5pt}\varphi \h{2pt}\bigg|_{r = \sigma_2}.
\end{eqnarray*}Applying the last estimate to (4.18) and taking $j$ large, we obtain \begin{eqnarray} \int_{B_{\sigma_2} \sim B_{\sigma_1}} \left| \h{1pt}\nabla \widehat{w}^{(j)} \h{1pt}\right|^2 \h{3pt}\lesssim \h{3pt}\dfrac{1}{\sigma_2 - \sigma_1} \h{1pt} \big\|\h{1pt}w^{(j)} - w^{\infty} \h{1pt}\big\|^2_{\infty;\h{1pt}\p B_{\sigma_2}}  + \big(\sigma_2 - \sigma_1 \big) \int_{0}^{\pi}  e_2 \big[ u^{\infty} \big] +  e_2 \left[ u^{(j)} \right]  \h{2pt}\mathrm{d}\h{0.5pt}\varphi \h{2pt}\bigg|_{r = \sigma_2}.
\end{eqnarray}Since $u_b^+$ and $u_c^-$ are minimizers of the two minimization problems in (1.12), respectively, \begin{eqnarray*} \int_{B_{\sigma_2}} \left| \h{1pt}\nabla w^{(j)}\h{1pt}\right|^2 + \sqrt{2}\h{1pt}\mu\h{1pt}k_j^2 \left[ 1 - 3 \h{0.5pt}S\left[\h{0.5pt} w^{(j)} \h{0.5pt}\right] \h{0.5pt}\right]  \h{2pt}\leq\h{2pt}  \int_{B_{\sigma_2}} \left| \h{1pt}\nabla \widehat{w}^{(j)}\h{1pt}\right|^2 + \sqrt{2}\h{1pt}\mu\h{1pt}k_j^2 \left[ 1 - 3 \h{0.5pt}S\left[\h{0.5pt} \widehat{w}^{(j)} \h{0.5pt}\right] \h{0.5pt}\right].
\end{eqnarray*}By the boundedness of potential term and (4.15)-(4.17), the last estimate infers
  \begin{eqnarray*}\int_{B_{\sigma_2}} \left| \h{1pt}\nabla w^{(j)}\h{1pt}\right|^2 \h{2pt}\leq\h{2pt} c_{\mu} \h{1pt}k_j^2 + \int_{B_{\sigma_2} \sim B_{\sigma_1}}  \left| \h{1pt}\nabla \widehat{w}^{(j)}\h{1pt}\right|^2 + \int_{B_{\sigma_1} \sim B_{R}}  \left| \h{1pt}\nabla w^{\infty}\h{1pt}\right|^2 + \int_{B_{R}}  \left| \h{1pt}\nabla w^{*}\h{1pt}\right|^2.
\end{eqnarray*}Here $c_\mu$ is a positive constant depending only on $\mu$. Utilizing (4.19) and convergence in (4.4), (4.8) and (4.13), we can take $j \rightarrow \infty $ in the last estimate and get \begin{eqnarray*}\int_{B_{\sigma_2}} \left| \h{1pt}\nabla w^{\infty}\h{1pt}\right|^2 \h{2pt}\leq\h{2pt} c \h{1pt}\big( \sigma_2 - \sigma_1 \big) \int_0^{\pi} e_2\h{0.5pt}\big[\h{0.5pt}u^{\infty}\h{0.5pt}\big] + \int_{B_{\sigma_1} \sim B_{R}}  \left| \h{1pt}\nabla w^{\infty}\h{1pt}\right|^2 + \int_{B_{R}}  \left| \h{1pt}\nabla w^{*}\h{1pt}\right|^2.
\end{eqnarray*}Here $c > 0$ is a universal constant. (4.14) then follows by taking  $\sigma_1 \rightarrow \sigma_2^-$ in the above estimate.
\end{proof}\
\\
\textbf{\large IV.2. \small GENERIC SINGULARITIES}\vspace{1pc}\\
Since $w^{\infty} = L\h{0.5pt}\big[\h{0.5pt}u^{\infty}\h{0.5pt}\big]$ is a $0$-homogeneous harmonic map from $B_1$ to $\mathbb{S}^4$,  the three components of $u^{\infty}$ satisfy the ODE system (3.8). By this ODE system, there holds \begin{lemma} The limiting map $w^{\infty}$ must be one of the following harmonic maps: \vspace{0.5pc}
\\
$\mathbf{Class \h{4pt}I:}$ \h{2pt}Constant mapping $N^*$ or $- N^*$, where $N^* = (0, 0, 1, 0, 0)$; \vspace{0.5pc}\\
$\mathbf{Class \h{4pt}II:}$ \h{2pt}For some $\beta_1 \in \big(\h{0.5pt}0, \pi\h{0.5pt}\big)$, $w^{\infty} = L\h{0.5pt}\big[\h{0.5pt}u^{\infty}\h{0.5pt}\big]$ with $u^{\infty}$ equaling to  \begin{eqnarray*} \text{either }\h{10pt}\left(0, \h{3pt} \dfrac{\cos \beta_1 + \cos \varphi}{1 + \cos \beta_1 \cos \varphi}, \h{3pt}\dfrac{\sin \beta_1 \sin \varphi}{1 + \cos \beta_1 \cos \varphi}\right) \h{10pt}\text{or}\h{10pt}\left(0, \h{3pt} \dfrac{\cos \beta_1 - \cos \varphi}{1 - \cos \beta_1 \cos \varphi}, \h{3pt}\dfrac{\sin \beta_1 \sin \varphi}{1 - \cos \beta_1 \cos \varphi}\right) \h{20pt}\text{on $\big[\h{0.5pt}0, \pi \h{0.5pt}\big]$;}
\end{eqnarray*}
$\mathbf{Class \h{4pt}III:}$ \h{2pt}For some $\beta_2 \in \big(\h{0.5pt}0, \pi\h{0.5pt}\big)$, $w^{\infty} = L\h{0.5pt}\big[\h{0.5pt}u^{\infty}\h{0.5pt}\big]$ with $u^{\infty}$ equaling to   \begin{eqnarray*} \text{either} \h{2pt}\left(\dfrac{\sin \beta_2 \sin J\left(\varphi\right)}{1 + \cos \beta_2 \cos J\left(\varphi\right)}, \h{3pt} \dfrac{\cos \beta_2 + \cos J\left(\varphi\right)}{1 + \cos \beta_2 \cos J\left(\varphi\right)}, \h{3pt}0\right)\h{2pt}\text{or}\h{2pt}\left(\dfrac{\sin \beta_2 \sin J\left(\varphi\right)}{1 - \cos \beta_2 \cos J\left(\varphi\right)}, \h{3pt} \dfrac{\cos \beta_2 - \cos J\left(\varphi\right)}{1 - \cos \beta_2 \cos J\left(\varphi\right)}, \h{3pt}0\right) \h{5pt}\text{on $\big[\h{0.5pt}0, \pi \h{0.5pt}\big]$.}
\end{eqnarray*}Here $J\left(\varphi\right) := \arccos\left( \dfrac{2 \cos \varphi}{1 + \cos^2 \varphi}\right)$.

\end{lemma}
\begin{proof}[\bf Proof] In the proof we use $v$ to denote the $3$-vector field $u^{\infty}$. Since $x_0 = (\h{0.5pt}0, 0, z_0\h{0.5pt})$, where $z_0 \in \big(\h{0.5pt}0, 1\big)$, by the non-negativity of $u_{\h{1pt}b; \h{0.5pt}1}^+$, $u_{\h{1pt}b; \h{0.5pt}3}^+$, $u_{\h{1pt}c; \h{0.5pt}1}^-$ and $u_{\h{1pt}c; \h{0.5pt}3}^-$ on $B_1^+$, we have \begin{eqnarray}v_1 \geq 0, \h{20pt}v_3 \geq 0 \h{30pt}\text{ on $\big(\h{0.5pt}0, \pi\big)$.}\end{eqnarray}

If $v_1 = 0$ at some $\varphi_0 \in \big(0, \pi\big)$, then by (4.20), $\varphi_0$ is a minimum point of $v_1$. It follows $v_1'\left(\varphi_0\right) = 0$. Here and in the following $'$ denotes the derivative with respect to the variable $\varphi$. By applying uniqueness result for ordinary differential equation to the equation of $v_1$ in (3.8), it turns out $v_1 \equiv 0$ on $\big[\h{0.5pt}0, \pi \h{0.5pt}\big]$. If in addition $v_3 = 0$ at some point in   $\big(\h{0.5pt}0, \pi\h{0.5pt}\big)$, then similarly we have $v_3 \equiv 0$ on $\big[\h{0.5pt}0, \pi \h{0.5pt}\big]$. Here we just need to use the equation of $v_3$ in (3.8). In this case, $w^{\infty}$ is the constant mapping $N^*$ or $- N^*$. If $v_1 \equiv 0$ on $\big[\h{0.5pt}0, \pi \h{0.5pt}\big]$ and $v_3 > 0$ on $\big(\h{0.5pt}0, \pi \h{0.5pt}\big)$, then we can represent $v$ by \begin{eqnarray} v = \big(0, v_2, v_3\big) = \big(0, \cos \alpha, \h{1pt}\sin \alpha\big), \h{20pt}\text{ where $\alpha = \arccos v_2$.}\end{eqnarray}Plugging the representation (4.21) into (3.9), we get \begin{eqnarray*}\left( \alpha' \right)^2 = \left(\dfrac{\sin \alpha}{\sin \varphi}\right)^2 \h{20pt}\text{on $\big(\h{0.5pt}0, \pi\h{0.5pt}\big)$.}
\end{eqnarray*}The sign of $\alpha'$ does not change on $\big(\h{0.5pt}0, \pi \h{0.5pt}\big)$. Otherwise there will be a $\varphi_1 \in \big(\h{0.5pt}0, \pi\h{0.5pt}\big)$ so that $\alpha'(\varphi_1) = 0$. Using the last equality yields $\sin \alpha(\varphi_1) = 0$. It violates the assumption that $v_3\left(\varphi_1\right) = \sin \alpha(\varphi_1) > 0$. Therefore we have from the last equality that \begin{eqnarray}\alpha' = \dfrac{\sin \alpha}{\sin \varphi} \h{10pt}\text{on $\big(\h{0.5pt}0, \pi\h{0.5pt}\big)$} \h{20pt}\text{or}\h{20pt}\alpha' = - \dfrac{\sin \alpha}{\sin \varphi} \h{10pt}\text{on $\big(\h{0.5pt}0, \pi\h{0.5pt}\big)$.}
\end{eqnarray}Solving the two ODEs in (4.22), we obtain the harmonic maps given in Class II of the lemma.

In the remaining we suppose $v_1 > 0$ on $\big(\h{0.5pt}0, \pi \h{0.5pt}\big)$. In this case we multiply $v_3$ and $v_1$ on both sides of the equations of $v_1$ and $v_3$ in (3.8), respectively. Then we subtract one of the two resulting equations from another one. By this way, it turns out\begin{eqnarray} \big( \big(v_1 v_3' - v_3 v_1' \big)\sin \varphi  \big)' \h{2pt}=\h{2pt} - \dfrac{3\h{1pt} v_1 v_3}{\sin \varphi} \h{2pt}\leq\h{2pt} 0, \h{20pt}\text{for all $\varphi \in \big(0, \pi\big)$.}
\end{eqnarray}Here we also have used (4.20) in the last inequality above. (4.23) infers that $\big(v_1 v_3' - v_3 v_1'\big)\h{1pt}\sin \varphi$ is non-increasing on $\big(\h{0.5pt}0, \pi\h{0.5pt}\big)$. The quantity $\big(v_1 v_3' - v_3 v_1'\big)\sin \varphi$ equals to $0$ when $\varphi = 0$ and $\pi$. Therefore we have  $\left(v_1v_3' - v_3 v_1'\right)\sin \varphi  \equiv 0$,  for all $\varphi \in \big(0, \pi\big)$.  Since $v_1 > 0$ on $\big(\h{0.5pt}0, \pi\h{0.5pt}\big)$, it follows  \begin{eqnarray*}\left( \dfrac{v_3}{v_1}\right)' \equiv 0, \h{20pt}\text{for all $\varphi \in \big(0, \pi\big)$.}\end{eqnarray*}Therefore $v_3 = c v_1$ on $\big(\h{0.5pt}0, \pi\h{0.5pt}\big)$, where $c$ is a constant. If $c \neq 0$, then by the equations of $v_1$ and $v_3$  in (3.8), it holds $v_1 \equiv 0$ on $\big(\h{0.5pt}0, \pi\h{0.5pt}\big)$. It violates our assumption that $v_1 > 0$ on $\big(\h{0.5pt}0, \pi\h{0.5pt}\big)$. Therefore $c = 0$. Equivalently $v_3 \equiv 0$ on $\big(\h{0.5pt}0, \pi\h{0.5pt}\big)$. Now we can represent $v$ by $v = \big(v_1, v_2, 0\big) = \big(\sin \alpha, \h{1pt}\cos \alpha, 0\big)$,  where $\alpha = \arccos v_2$. Plugging this representation of $v$ into (3.9) then yields \begin{eqnarray*} \left( \alpha' \right)^2 = \left(\dfrac{2 \sin \alpha}{\sin \varphi}\right)^2 \h{20pt}\text{on $\big(\h{0.5pt}0, \pi\h{0.5pt}\big)$.}
\end{eqnarray*}The sign of $\alpha'$ does not change on $\big(\h{0.5pt}0, \pi \h{0.5pt}\big)$. Therefore we have from the last equality that \begin{eqnarray}\alpha' = \dfrac{2 \h{0.5pt}\sin \alpha}{\sin \varphi} \h{10pt}\text{on $\big(\h{0.5pt}0, \pi\h{0.5pt}\big)$} \h{20pt}\text{or}\h{20pt}\alpha' = - \dfrac{2 \h{0.5pt}\sin \alpha}{\sin \varphi} \h{10pt}\text{on $\big(\h{0.5pt}0, \pi\h{0.5pt}\big)$.}
\end{eqnarray}Solving the two ODEs in (4.24), we obtain the harmonic maps given in Class III of the lemma.
\end{proof}

With Proposition 4.1 and Lemma 4.3, we can further determine $w^{\infty}$ at possible singularity $x_0$. \begin{prop} Suppose that $\mathcal{W} = L\h{1pt}\big[\h{0.5pt}u_{\h{1pt}b}^+\h{0.5pt}\big]$ or $L\h{0.5pt}\big[\h{0.5pt}u_{\h{1pt}c}^-\h{0.5pt}\big]$. Moreover we assume that  $x_0 \in B_1^+ \cap \h{1pt} l_3$ is a singular point of $\mathcal{W}$. Then the tangent map $w^{\infty}$ of $\mathcal{W}$ at $x_0$ must  be either  $\Lambda^+$ or $\Lambda^-$ in (1.14).
\end{prop}\begin{proof}[\bf Proof] Since $w^{\infty}$ is $0$-homogeneous, we take $R$ in Proposition 4.1 to be $1$. If $w^{\infty}$ is a constant map, then by (4.4) and standard $\epsilon$-regularity result (similar to Proposition 3.1), $x_0$ is a smooth point of $\mathcal{W}$. This contradicts our assumption that $x_0$ is a singularity of $\mathcal{W}$. If $w^{\infty}$ is a harmonic map in Class II of Lemma 4.3 with $\cos \beta_1 \neq 0$, then we fix a $x_* \in B_1^+ \cap\h{1pt}l_3$. Using the value of $w^{\infty}$ on $\p B_1$, we can construct a $0$-homogeneous map with respect to the center $x_*$. Similarly as  the proof of Theorem 7.3 in [11], the Dirichlet energy on $B_1$ of this constructed $0$-homogeneous map is strictly less than $8 \pi$, provided that $x_*$ is sufficiently close to $0$. The closeness depends on $\cos \beta_1$. Therefore $w^{\infty}$ cannot saturate the minimal energy of (4.5) since the Dirichlet energy of $w^{\infty}$ on $B_1$ equals to $8 \pi$. We get a contradiction to Proposition 4.1. Hence $w^{\infty}$ can only be $\Lambda^+$, $\Lambda^-$  or the harmonic maps in Class III of Lemma 4.3. In the remaining we show that  $w^{\infty}$ cannot be a harmonic map in Class III of Lemma 4.3. For our case, we can use an easier construction than the one given in the proof of Theorem 7.4 in [11]. Suppose that $w^{\infty} = L\h{0.5pt}\left[\h{0.5pt}u^{\infty}\h{0.5pt}\right]$ is a harmonic map in Class III of Lemma 4.3. Denoting by $\alpha_{\star}$ the angular function $\arccos u^{\infty}_2$, we represent $u^{\infty}$ by $\big(\sin \alpha_{\star}, \h{2pt}\cos \alpha_{\star}, \h{2pt}0\big)$. Let $f$ be a radial scalar function with $f(1) = 0$ and $P_*$ be a perturbation which takes the form \begin{eqnarray} P_*  = \big(\h{0.5pt}0, \h{3pt}0, \h{3pt}0, \h{3pt}f(r)\h{0.5pt}\sin \alpha_{\star} \h{0.5pt}\cos \theta, \h{3pt}f(r)\h{0.5pt}\sin \alpha_{\star} \h{0.5pt} \sin \theta\h{0.5pt}\big).
\end{eqnarray}If we denote by $\mathrm{Hess}^{\h{1pt}D}_{\h{1pt}w^{\infty}}\big[\h{0.5pt}\cdot, \cdot \h{0.5pt}\big]$ the Hessian of Dirichlet energy at $w^{\infty}$, it then turns out \begin{eqnarray*} \mathrm{Hess}^{\h{1pt}D}_{\h{1pt}w^{\infty}}\big[\h{0.5pt}P_*,  P_* \h{0.5pt}\big] = 2 \int_{B_1} \big|\h{1pt} \nabla P_* \h{0.5pt}\big|^2 - \big|\h{0.5pt} P_* \h{0.5pt}\big|^2 \h{2pt} \big|\h{1pt} \nabla w^{\infty} \h{0.5pt} \big|^2.
\end{eqnarray*}By (4.25) and in light that $\alpha_{\star}$ satisfies either the first or the second equation in (4.24),  it follows \begin{eqnarray*} \dfrac{1}{2 \pi} \int_{B_1} \big|\h{0.5pt} \nabla P_* \h{0.5pt}\big|^2 - \big|\h{0.5pt} P_* \h{0.5pt} \big|^2 \h{1pt} \big|\h{0.5pt} \nabla w^{\infty}\h{0.5pt} \big|^2 \h{2pt}=\h{2pt}  \int_0^{\pi} \sin^2 \alpha_{\star} \sin \varphi \h{1pt} \int_0^{1} \big(\h{0.5pt}f'\h{0.5pt}\big)^2 \h{0.5pt}  r^2 - 3 \int_0^1 f^2 \h{2pt}\leq\h{2pt}  \int_0^{1} \big(\h{0.5pt}f'\h{0.5pt}\big)^2 \h{0.5pt}  r^2 - 3 \int_0^1 f^2.
\end{eqnarray*}Using Lemma 1.3 in [58], we can find a $f$ compactly supported in the interval $\big(0, 1\big)$ so that the most-right-hand side above is strictly negative. The proof then completes.
\end{proof}\
\\
\textbf{\large IV.3. \small STRUCTURE OF SINGULARITY}\vspace{1pc}\\
By the same arguments as in the proof of Regularity theorem 6.4 in [31], it holds \begin{lemma}The set of singularities of $\mathcal{W}$ is discrete.\end{lemma}\noindent In this section we consider the asymptotic behaviour of $\mathcal{W}$ near each singularity. Firstly we show the uniqueness of tangent map of $\mathcal{W}$ at a possible singularity. That is \begin{lemma}Let $\mathcal{W}$ be either $L\big[\h{0.5pt}u_{\h{1pt}b}^+\h{0.5pt}\big]$ or $L\h{0.5pt}\big[\h{0.5pt}u_{\h{1pt}c}^-\h{0.5pt}\big]$. If $x_0 \in B_1^+ \cap\h{1pt}l_3$ is a singular point of $\mathcal{W}$, then it holds \begin{eqnarray*}   r^{-3} \int_{B_r\left(x_0\right)} \big|\h{1pt} \mathcal{W} -  \Lambda \left( \cdot - x_0 \right) \h{1pt}\big|^2 + r^{-1} \int_{B_r\left(x_0\right)} \big|\h{1pt}\nabla \mathcal{W} - \nabla \Lambda \left( \cdot - x_0 \right) \h{1pt}\big|^2 \longrightarrow 0, \h{20pt}\text{as $r \rightarrow 0^+$,}
\end{eqnarray*}where $\Lambda = \Lambda^+$ or $\Lambda^-$ in (1.14). Moreover $\Lambda$ is uniquely determined by $x_0$, the location of singularity.
\end{lemma}\begin{proof}[\bf Proof] Recalling  the blow-up sequence (4.1), by Proposition 4.4, we know that the limit $w^{\infty}$ in (4.4) satisfies $w^{\infty} = L\h{0.5pt}\big[\h{0.5pt}u^{\infty}\h{0.5pt}\big]$  with $u^{\infty} = \big(0, \h{2pt}\cos \varphi, \h{2pt}\sin \varphi \big)$ or $\big(0, \h{2pt}- \cos \varphi, \h{2pt}\sin \varphi \big)$.  Suppose that $u^{\infty} = \big(0, \h{2pt}\cos \varphi, \h{2pt}\sin \varphi \big)$. Then by the uniform convergence in (4.13), it follows $u^{(j)}\big(\sigma_2, 0\h{0.5pt}\big) \longrightarrow \big(\h{0.5pt}0, \h{2pt}1, \h{2pt}0\h{0.5pt}\big)$  and $u^{(j)}\big(\sigma_2, \pi\big) \longrightarrow \big(\h{0.5pt}0, \h{2pt}- 1, \h{2pt}0\h{0.5pt}\big)$,  as $j \rightarrow \infty$. Here $\sigma_2$ is some constant in $(0, 1)$. Denote by $x_j^+$ the point $ x_0 + k_j \h{1pt}\big(\h{0.5pt}0, \h{1pt}0, \h{1pt}\sigma_2 \h{0.5pt}\big)$ and  $x_j^-$ the point $ x_0 - k_j \h{1pt} \big(\h{0.5pt}0, \h{1pt}0, \h{1pt}\sigma_2 \h{0.5pt}\big)$. The two limits infer  \begin{eqnarray}\mathcal{W}\left(x_j^+\right) = N^* \h{20pt}\text{and} \h{20pt} \mathcal{W}\left(x_j^-\right) = - N^*, \h{20pt}\text{ provided $j$ is large.}\end{eqnarray} Here $N^* = (0, 0, 1, 0, 0)$. In light of Lemma 4.5, when $j$ is large enough, $x_0$ is the only singularity of $\mathcal{W}$ on the segment connecting $x_j^+$ and $x_0$. Therefore $\mathcal{W} \equiv N^*$ on the open segment between $x_j^+$ and $x_0$.  Similarly for $j$ large, $\mathcal{W} \equiv - N^*$ on the open segment between $x_j^-$ and $x_0$. Fix a $j$ sufficiently large. If we have another positive decreasing sequence $\big\{ \rho_k \big\}$ so that up to a subsequence $\mathcal{W}\left(x_0 + \rho_k\h{0.5pt}\cdot \h{1pt}\right)$ converges to $\Lambda^-$ strongly in $H_{\mathrm{loc}}^1\big(B_1; \h{0.5pt}\mathbb{S}^4\big)$ as $k \rightarrow \infty$ (here $\rho_k \rightarrow 0$ as $k \rightarrow \infty$), then the same derivation of (4.26) can be applied to find  a point on the open segment between $x_j^+$ and $x_0$ at which $\mathcal{W}$ equals to $- N^*$. But this is impossible. Therefore for any decreasing sequence $\big\{ \rho_k \big\}$, there is a subsequence, still denoted by $\big\{\h{0.5pt}\rho_k\h{0.5pt}\big\}$, so that $\mathcal{W}\left(x_0 + \rho_k\h{0.5pt}\cdot \h{1pt}\right)$ converges $\Lambda^+$ strongly in $H_{\mathrm{loc}}^1\big(B_1; \h{0.5pt}\mathbb{S}^4\big)$ as $k \rightarrow \infty$. The proof then follows.
\end{proof}

Moreover the following proposition holds by Lemma 4.6, $W^{2, p}$-estimate for elliptic equations and similar arguments as in Sect.3.15 of [60].\begin{prop}Let $\mathcal{W} = L\big[\h{0.5pt}u_{\h{1pt}b}^+\h{0.5pt}\big]$ or $L\h{0.5pt}\big[\h{0.5pt}u_{\h{1pt}c}^-\h{0.5pt}\big]$. Suppose that $x_0 \in B_1^+ \cap\h{1pt}l_3$ is a singular point of $\mathcal{W}$ with $\Lambda$ being its   tangent map  at $x_0$. Then we have \begin{eqnarray*} \sum_{j = 0}^2 \sigma^j \h{2pt}\Big\|\h{1pt}\nabla^j \mathcal{W} - \nabla^j \Lambda\left(\cdot - x_0\right)\h{1pt}\Big\|_{\infty; \h{1pt}\p B_{\sigma}\left(x_0\right)} \longrightarrow 0, \h{20pt}\text{as $\sigma \rightarrow 0^+$.}
\end{eqnarray*}Here $\|\cdot\|_{\infty; \h{1pt}\p B_{\sigma}\left(x_0\right)}$ is the $L^\infty$-norm on $\p B_{\sigma}\left(x_0\right)$.
\end{prop}\
\vspace{0.3pc}\\
\setcounter{section}{5}
\setcounter{theorem}{0}
\setcounter{equation}{0}
\textbf{\large V.  \small BIAXIAL TORUS AND SPLIT-CORE SEGMENT STRUCTURES}\vspace{1pc}\\
In this section we prove Theorems 1.6-1.7. Recall that $\mathcal{W}_{\h{1pt}b}^+ = L \h{1pt}\big[\h{1pt}u_b^+\h{1pt}\big]$ and $\mathcal{W}_{\h{1pt}c}^- = L\h{1pt}\big[\h{1pt}u_c^-\h{1pt}\big]$. \vspace{1pc}\\
\textbf{\large V.1. \small BIAXAIL TORUS STRUCTURE IN $\mathcal{W}_b^+$}\vspace{1pc}\\
To simplify the notation, we use $u^+$ to denote $u_{\h{1pt}b}^+$. In this section $\lambda_1$, $\lambda_2$, $\lambda_3$ are three eigenvalues in Definition 1.3 computed in terms  of $u^+$. Note that $u_{3}^+ = 0$ on $T$. By Definition 1.3, it follows $\lambda_1 = - \dfrac{1}{2} \left( u_{1}^+ + \dfrac{1}{\sqrt{3}} \h{1pt}u_{2}^+\right)$, $\lambda_2 =  \dfrac{1}{4} \bigg\{  u_{1}^+ + \dfrac{1}{\sqrt{3}} \h{1pt}u_{ 2}^+  - \left|\h{1pt}u_{ 1}^+ - \sqrt{3} \h{0.5pt}u_{ 2}^+ \h{1pt}\right| \h{1pt} \bigg\}$, $\lambda_3 =  \dfrac{1}{4} \bigg\{  u_{ 1}^+ + \dfrac{1}{\sqrt{3}} \h{1pt}u_{ 2}^+  + \left|\h{1pt}u_{ 1}^+ - \sqrt{3} \h{0.5pt}u_{2}^+ \h{1pt}\right| \h{1pt} \bigg\}$ on $T$. The three eigenvalues cannot be all $0$ at some point on $T$, since otherwise $u^+ = 0$ at this point. If $\lambda_1 = \lambda_2$ on $T$, then $3 \left( u_{1}^+ + \dfrac{1}{\sqrt{3}} \h{1pt}u_{2}^+ \right) = \left|\h{1pt}u_{1}^+ - \sqrt{3} \h{0.5pt}u_{ 2}^+ \h{1pt}\right|$. If $\lambda_2 = \lambda_3$ on $T$, then  $u_{1}^+ = \sqrt{3}\h{1pt}u_{2}^+$. If $\lambda_1 = \lambda_3$ on $T$, then $3 \left( u_{1}^+ + \dfrac{1}{\sqrt{3}} \h{1pt}u_{2}^+ \right) = - \left|\h{1pt}u_{1}^+ - \sqrt{3} \h{0.5pt}u_{ 2}^+ \h{1pt}\right|$. Therefore by the unit length condition for $u^+$ and the strict positivity of $u_{1}^+$ (this is true by strong maximum principle), the map $u^+$ is uniaxial on $T \sim \big\{\h{0.5pt}0\h{0.5pt}\big\}$ if and only if  $u^+ = \big( \sqrt{3}\big/2, \h{1pt}1\big/2, \h{1pt}0 \big)$ or $u^+ = \big( \sqrt{3}\big/2, \h{1pt}- 1\big/2, \h{1pt}0 \big)$. In light that $u_{2}^+ \geq b$ on $T$ and $u^+$ is regular at the origin (see Proposition 3.2), it follows $u_{2}^+ = 1$ at the origin. Since $u_{2}^+ = - 1/2$ when $\rho = 1$ and $z = 0$, using regularity of $u^+$ on $T$, we can find a point on $T$ at where $u_{2}^+ = 1/2$. Equivalently the vector $ \big(\sqrt{3}\big/2, \h{1pt}1\big/2, \h{1pt}0 \big)$ can be taken by $u^+$ at some point on $T$. $u^+$ is a classical solution to (2.1) on a small neighborhood of $\interior{T}$ in $\mathbb{D}$. Moreover the system (2.1) is an analytic elliptic system. Therefore due to [48], $u^+$ is real analytic on $\interior{T}$, which infers real analyticity of $u_{2}^+$ on $\interior{T}$. It then turns out that $1/2$ can only be taken finitely many times by $u_{2}^+$ on $\interior{T}$. Otherwise $u_{2}^+ \equiv 1/2$ on $T$. This is impossible. Since $u_{2}^+(0) = 1$ and $u_{ 2}^+\left(1, 0\right) = - 1/2$, within all points on $\interior{T}$ at where $u_{2}^+ = 1/2$, there is one, denoted by $x_0 = \left(\rho_0, 0 \right)$, so that for a sufficiently small $\epsilon > 0$, it holds $$u_{2}^+ > \dfrac{1}{2} \h{5pt}\text{ on $\Big\{\left(\rho, 0\right) : \rho \in \left(\h{0.5pt} \rho_0 - \epsilon, \h{1pt}\rho_0\right) \Big\}$} \h{10pt}\text{ and } \h{10pt} u_{2}^+ < \dfrac{1}{2} \h{5pt}\text{ on $\Big\{\left(\rho, 0\right) : \rho \in \left(\h{0.5pt} \rho_0 , \h{1pt}\rho_0 + \epsilon\right) \Big\}$.}$$ For points in $D_{\epsilon}\left(x_0\right) \sim T$,  strong maximum principle  induces that  $u_{3}^+ > 0$ on $ D_{\epsilon}\left(x_0\right) \cap \mathbb{D}^+$  and $u_{3}^+ < 0$  on $D_{\epsilon}\left(x_0\right)  \cap \mathbb{D}^-$.  Here $\mathbb{D}^-$ is the lower-half part of $\mathbb{D}$. By Definition 1.3, we then have  \begin{eqnarray}\lambda_3 > \lambda_2 \h{20pt}\text{ on $D_{\epsilon}\left(x_0\right) \sim T$.}\end{eqnarray} It can also be computed that \begin{small}$\lambda_2 - \lambda_1   \h{2pt}=\h{2pt}  \dfrac{1}{4} \bigg\{ 3 \left( u_{ 1}^+ + \dfrac{1}{\sqrt{3}} \h{1pt}u_{ 2}^+ \right) - \sqrt{\left( u_{ 1}^+ - \sqrt{3} \h{1pt}u_{ 2}^+\right)^2 + 4 \h{1pt}\left(u_{ 3}^+\right)^2 }  \h{2pt} \bigg\}$\end{small}\noindent.  By the regularity of $u^+$ on $D_{\epsilon}\left(x_0\right)$ and the fact that $u^+\left(x_0\right) = \left( \sqrt{3}\big/ 2 , \h{1pt} 1 \big/2, \h{1pt}0 \right)$, it then turns out  \begin{eqnarray} \lambda_2  \h{2pt}>\h{2pt}\lambda_1 \h{20pt}\text{on $D_{\epsilon}\left(x_0\right)$, provided $\epsilon$ is suitably small.}
\end{eqnarray}This inequality and (5.1) infer \begin{eqnarray} \lambda_3 \h{2pt}>\h{2pt} \lambda_2 \h{2pt}>\h{2pt}\lambda_1 \h{20pt}\text{on $D_{\epsilon}\left(x_0\right) \sim T$, provided $\epsilon$ is suitably small.}\end{eqnarray}

On $\Big\{\left(\rho, 0\right) : \rho \in \left(\h{0.5pt} \rho_0 , \h{1pt}\rho_0 + \epsilon\right) \Big\}$, it holds $u_{ 2}^+ \in \left(\h{0.5pt}0, 1/2 \right)$, provided $\epsilon$ is small. Therefore we have $u_{1}^+ > \sqrt{3}\big/2$ on $\Big\{\left(\rho, 0\right) : \rho \in \left(\h{0.5pt} \rho_0 ,\h{1pt} \rho_0 + \epsilon\right) \Big\}$, which  infers \begin{eqnarray} u_{1}^+ - \sqrt{3}\h{1pt}u_{ 2}^+ \h{2pt}>\h{2pt} 0 \h{20pt}\text{ on $\Big\{\left(\rho, 0\right) : \rho \in \left(\h{0.5pt} \rho_0 , \h{1pt}\rho_0 + \epsilon\right) \Big\}$.}\end{eqnarray} It then follows by Definition 1.3 that $\lambda_3 \h{2pt}> \h{2pt}\lambda_2$  on $\Big\{\left(\rho, 0\right) : \rho \in \left(\h{0.5pt} \rho_0 , \h{1pt}\rho_0 + \epsilon\right) \Big\}$, provided $\epsilon $ is small. Utilizing (5.2), we then obtain \begin{eqnarray*}\lambda_3 \h{2pt}>\h{2pt} \lambda_2 \h{2pt}>\h{2pt}\lambda_1 \h{20pt}\text{on $\Big\{\left(\rho, 0\right) : \rho \in \left(\h{0.5pt} \rho_0 , \h{1pt}\rho_0 + \epsilon\right) \Big\}$, provided $\epsilon$ is suitably small.}
\end{eqnarray*}Similarly we also have \begin{eqnarray*}\lambda_3 \h{2pt}>\h{2pt} \lambda_2 \h{2pt}>\h{2pt}\lambda_1 \h{20pt}\text{on $\Big\{\left(\rho, 0\right) : \rho \in \left(\h{0.5pt} \rho_0 - \epsilon, \h{1pt}\rho_0\right) \Big\}$, provided $\epsilon$ is suitably small.}
\end{eqnarray*}
By combining the above arguments with (5.3), on $D_{\epsilon}\left(x_0\right) \sim  \big\{x_0\big\}$, $u^+$ is biaxial with $\lambda_3$ being the largest eigenvalue.

By (5.3),  $\kappa^\star$ in (5) of Theorem 1.6 is a director field determined by $u^+$ on $D_{\epsilon}\left(x_0\right) \sim T$. It is also well-defined and continuous on $\Big\{\left(\rho, 0\right) : \rho \in \left(\h{0.5pt} \rho_0 , \h{1pt} \rho_0 + \epsilon\right) \Big\}$. In fact by (5.4) and $u_{3}^+ = 0$ on $T$, it holds \begin{eqnarray}\kappa^{\star} \h{2pt}\equiv\h{2pt} e_{\rho} \h{20pt}\text{on $\Big\{\left(\rho, 0\right) : \rho \in \left(\h{0.5pt} \rho_0 , \rho_0 + \epsilon\right) \Big\}$.}
\end{eqnarray}However  $\kappa^{\star}$  is not continuous on $\Big\{\left(\h{0.5pt}\rho, 0\h{0.5pt}\right) : \rho \in \big(\h{0.5pt} \rho_0 - \epsilon, \h{1pt} \rho_0  \big] \h{1pt} \Big\}$.
Fix an $\epsilon' \in \left(\h{0.5pt}0, \epsilon \right)$ and denote by $x'$ the point $\left(\h{0.5pt}\rho_0 - \epsilon', \h{1pt}0\h{0.5pt}\right)$. When we approach $x'$ along the lower part of $\p D_{\epsilon'}\left(x_0\right)$, the component $u_{3}^+$ keeps negative. It then follows \begin{small}\begin{eqnarray*}&&\dfrac{\sqrt{2}\h{1pt}u_{3}^+}{\sqrt{\left(u_{1}^+ - \sqrt{3}\h{1pt}u_{2}^+ \right)^2 + 4 \left( u_{3}^+ \right)^2}}\left\{ 1 + \dfrac{u_{1}^+ - \sqrt{3}\h{1pt}u_{2}^+}{\sqrt{\left( u_{1}^+ - \sqrt{3}\h{1pt}u_{2}^+ \right)^2 + 4 \left(u_{3}^+\right)^2}} \right\}^{- 1/2} \\
\\
\\
 && = \h{2pt}- \dfrac{\sqrt{2}}{2} \bigg\{ \left( u_{1}^+ - \sqrt{3}\h{1pt}u_{2}^+ \right)^2 + 4 \left(u_{3}^+\right)^2 \bigg\}^{- 1/4}\left\{ \sqrt{\left( u_{1}^+ - \sqrt{3}\h{1pt}u_{2}^+ \right)^2 + 4 \left(u_{3}^+\right)^2} - \left(u_{1}^+ - \sqrt{3}\h{1pt}u_{2}^+ \right) \right\}^{1/2},
\end{eqnarray*}\end{small}\noindent where the left-hand side above is evaluated  on the lower part of $\p D_{\epsilon'}\left(x_0\right)$. Since we have $u_{2}^+ > 1\big/2$ at $x'$, it follows $u_{1}^+ <  \sqrt{3}\big/2$ at $x'$. Therefore we get $u_{1}^+ - \sqrt{3}\h{1pt}u_{2}^+ < 0$ at $x'$. Utilizing this inequality and the last equality, we can approach $x'$ along the lower part of $\p D_{\epsilon'}\left(x_0\right)$ and  get \begin{small}\begin{eqnarray*}&&\dfrac{\sqrt{2}\h{1pt}u_{ 3}^+}{\sqrt{\left(u_{1}^+ - \sqrt{3}\h{1pt}u_{ 2}^+ \right)^2 + 4 \left( u_{ 3}^+ \right)^2}}\left\{ 1 + \dfrac{u_{1}^+ - \sqrt{3}\h{1pt}u_{2}^+}{\sqrt{\left( u_{1}^+ - \sqrt{3}\h{1pt}u_{ 2}^+ \right)^2 + 4 \left(u_{ 3}^+\right)^2}} \right\}^{- 1/2} \longrightarrow - 1. \end{eqnarray*}\end{small}\noindent Meanwhile it holds \begin{small}\begin{eqnarray*}\dfrac{\sqrt{2}}{2} \left\{ 1 + \dfrac{u_{ 1}^+ - \sqrt{3}\h{1pt}u_{ 2}^+}{\sqrt{\left( u_{ 1}^+ - \sqrt{3}\h{1pt}u_{ 2}^+ \right)^2 + 4 \left(u_{ 3}^+\right)^2}} \right\}^{1/2} \longrightarrow 0, \h{10pt}\text{if $x \rightarrow x'$ along the lower part of $\p D_{\epsilon'}\left(x_0\right)$. }\end{eqnarray*}\end{small}\noindent The last two limits  infer that when we approach $x'$ along the lower part of $\p D_{\epsilon'}\left(x_0\right)$,  the director field $\kappa^{\star}$ approaches to $- e_z$. Similarly when we approach $x'$ along the upper part of $\p D_{\epsilon'}\left(x_0\right)$, the director field $\kappa^{\star}$ approaches to $e_z$. Here we  need the positivity of $u_{3}^+$ on the upper part of $\p D_{\epsilon'}\left(x_0\right)$. By (1.15), $\kappa^\star$ has positive  coefficient in front of $e_{\rho}$ when $x \in \p D_{\epsilon'}\left(x_0\right) \sim \big\{\h{0.5pt}x'\h{0.5pt}\big\}$. Therefore when we start from $x'$ and rotate counter-clockwisely along $\p D_{\epsilon'}\left(x_0\right)$ back to $x'$, the director field $\kappa^{\star}$ continuously varies from $- e_z$ to $e_z$. Meanwhile $\kappa^{\star}$ keeps on the right-half part of $\left(\rho, z\right)$\h{0.5pt}-\h{0.5pt}plane. The angle of $\kappa^{\star}$ is totally changed by $\pi$.  Part $(6)$ in Theorem 1.6 follows by (5.5) and L'H$\mathrm{\hat{o}}$pital's rule.\vspace{1.5pc}\\
\textbf{\large V.2. \small SPLIT-CORE LINE SEGMENT STRUCTURE}\vspace{1pc}\\
In this section we prove Theorem 1.7 by considering the mapping $\mathcal{W}_{\h{1pt}c}^-$, equivalently $u_{\h{1pt}c}^-$. To simplify the notation, we use $u^-$ to denote $u_{\h{1pt}c}^-$. Throughout this section $\lambda_1$, $\lambda_2$, $\lambda_3$ are three eigenvalues in Definition 1.3 computed in terms  of $u^-$. Since $u_{2}^- \leq c$ on $T$, by the smoothness of $\mathcal{W}_{\h{1pt}c}^-$ at the origin (see Proposition 3.2), we get $u^- = \left(0, - 1, 0\right)$ at the origin. Since $u^- = \left(0, 1, 0\right)$ when $\rho = 0$ and $z = 1$, we can find a location on $l_3$ at where $u^-$ is singular. Denote by $x_0^+ = \left(0, z_0\right)$ the lowest singularity of $u^-$ on $l_3^+$. Utilizing Proposition 4.7, for any $\epsilon > 0$ sufficiently small, we have a $\delta\left(\epsilon\right) > 0$ suitably small so that  \begin{eqnarray} \sum_{j = 0}^2 \left|\h{1pt}\p_{\psi}^j\h{1pt}u^- - \p_{\psi}^j\left(\h{0.5pt}0, \h{1pt}\cos \psi, \h{1pt}\sin \psi \h{0.5pt}\right) \h{1pt}\right|_{\infty;\h{1pt}\big\{\sigma\big\} \times \big[\h{0.5pt}0, \h{0.5pt}\pi\h{0.5pt}\big]} \h{2pt}<\h{2pt}\epsilon, \h{20pt}\text{for all $\sigma$ satisfying $\sigma \in \left(\h{0.5pt}0, \h{1pt} \delta\left(\epsilon\right) \right)$.}
\end{eqnarray}Here we use $\left(r_*, \psi, \theta\right)$ to denote the spherical coordinate system with respect to the center $x_0^+$. Without ambiguity, $u^-$ in (5.6) is also understood as a $3$-vector field depending on $r_*$ and $\psi$.

In what follows we define a dumbbell. \begin{definition}Let $x_0^-$ be the symmetric point of $x_0^+$ with respect to the $(x_1, x_2)$\h{0.5pt}-\h{0.5pt}plane.  Let $\epsilon_1$ be another positive constant. It is sufficiently small and satisfies $\epsilon_1 < \delta\left(\epsilon\right)$. In $(x_1, z)$\h{0.5pt}-\h{0.5pt}plane, the horizontal line $z = z_0 - \delta\left(\epsilon\right) + \epsilon_1$ has two intersections  with the circle $\p D_{\delta\left(\epsilon\right)}\left(x_0^+\right)$. Here we also use $D_{\rho}\left(x\right)$ to denote a disk in $(x_1, z)$\h{0.5pt}-\h{0.5pt}plane with center $x$ and radius $\rho$. The intersection point with positive $x_1$-coordinate is denoted by $x_1^+$, while the intersection point with negative $x_1$-coordinate is denoted by $x_2^+$.  The distance between $x_1^+$ ($x_2^+$ resp.) and $l_3$ equals to \begin{eqnarray}\rho_{\epsilon, \epsilon_1} = \epsilon_1^{1/2} \left( \h{1pt}2 \h{0.5pt} \delta\left(\epsilon\right) - \epsilon_1 \right)^{1/2}.\end{eqnarray} Similarly the horizontal line $z = - z_0 + \delta\left(\epsilon\right) - \epsilon_1$ also has two intersections  with $\p D_{\delta\left(\epsilon\right)}\left( x_0^-\right)$. We denote by $x_1^-$ the intersection point with positive $x_1$-coordinate, while $x_2^-$ denotes another intersection point which has negative $x_1$-coordinate. The contour $\mathscr{C}$ in the $(x_1, z)$\h{0.5pt}-\h{0.5pt}plane is then  defined as follows. Firstly we  start from $x_1^+$ and rotate counter-clockwisely  along $\p D_{\delta\left(\epsilon\right)}\left(x_0^+\right)$ to $x_2^+$. Then we connect $x_2^+$ and $x_2^-$ by the vertical segment between them. From $x_2^-$, we rotate counter-clockwisely along $\p D_{\delta\left(\epsilon\right)}\left(x_0^-\right)$ to $x_1^-$. Finally we connect $x_1^-$ and $x_1^+$ by the vertical segment between them. The dumbbell, denoted by $\mathrm{D}_{\epsilon, \epsilon_1}$, is the region in $(x_1, z)$\h{0.5pt}-\h{0.5pt}plane which is enclosed by the contour $\mathscr{C}$.   \end{definition}

On $\Big\{ \left(\h{0.5pt}0, z\right) : z \in \left(z_0, z_0 + \delta\left(\epsilon\right) \right) \cup \left(- z_0 - \delta\left(\epsilon\right), - z_0 \right)\Big\}$, it holds $u^- = \left(\h{0.5pt}0, 1, 0 \h{0.5pt}\right)$. By Definition 1.3, $u^-$ is uniaxial with two eigenvalues equaling to $ - \sqrt{3}\big/6$. The largest eigenvalue is $ \sqrt{3}\big/3$, whose associated eigenspace is given by $\mathrm{span}\Big\{e_z\Big\}$. On $\Big\{ \left(\h{0.5pt}0, z\right) : z \in \left(- z_0, z_0 \right)\Big\}$, it holds $u^- = \left(\h{0.5pt}0, - 1, 0 \h{0.5pt}\right)$. Still by Definition 1.3, $u^-$ is uniaxial with two eigenvalues equaling to $\sqrt{3}\big/6$. The largest eigenvalue is $\sqrt{3}\big/6$, whose associated eigenspace is given by $\mathrm{span}\Big\{ e_{\theta}, \h{1pt}e_{\rho} \Big\}$.

In what follows we show the biaxiality of $u^-$ on $\mathrm{D}_{\epsilon, \epsilon_1} \sim l_3$. We firstly compare the three eigenvalues on $D_{\delta\left(\epsilon\right)}\left(x_0^+\right)$. Due to the symmetry of $u^-$ with respect to the variable $z$, the case for $D_{\delta\left(\epsilon\right)}\left(x_0^-\right)$ can be similarly studied. Suppose that $\sigma$ is an arbitrary number in $\big(\h{1pt}0, \delta\left(\epsilon\right)]$. Using the polar angle $\psi$ in the spherical coordinate $\left(r_*, \psi, \theta\right)$, we have \begin{eqnarray*} u^-\left(\sigma, \psi\right) = u^-\left(\sigma, 0\right) + \p_{\psi}\h{0.5pt} u^- \h{2pt}\Big|_{\left(\sigma, 0\right)} \psi + \int_0^{\psi} \int_0^{\psi_1} \p^2_{\psi}\h{1pt} u^- \h{2pt}\Big|_{\left(\sigma, \h{1pt}\zeta\right)} \h{1pt}\mathrm{d}\h{1pt}\zeta \h{2pt}\mathrm{d}\h{1pt}\psi_1, \h{20pt}\text{for any $\psi \in \big[\h{0.5pt}0, \pi\h{0.5pt}\big].$}
\end{eqnarray*}Here we understand $u^-$ as a mapping depending on $r_*$ and $\psi$. In light of the regularity of $\mathcal{W}_{\h{1pt}c}^-$ at $\left(\h{0.5pt}0, z_0 + \sigma\right)$, it follows $\p_{\psi}\h{0.5pt} u^- \h{2pt}\Big|_{\left(\sigma, 0\right)} = 0$. The last equality is then reduced to \begin{eqnarray*} u^-\left(\sigma, \psi\right) = u^-\left(\sigma, 0\right)   + \int_0^{\psi} \int_0^{\psi_1} \p^2_{\psi}\h{1pt} u^- \h{2pt}\Big|_{\left(\sigma, \h{1pt}\zeta\right)} \h{1pt}\mathrm{d}\h{1pt}\zeta \h{2pt}\mathrm{d}\h{1pt}\psi_1, \h{20pt}\text{for any $\psi \in \big[\h{0.5pt}0, \pi\h{0.5pt}\big],$}
\end{eqnarray*}which furthermore infers \begin{eqnarray*} \sqrt{3}\h{1pt}u_{1}^-\h{1pt}\Big|_{\left(\sigma, \psi\right)} + u_{2}^-\h{1pt}\Big|_{\left(\sigma, \psi \right)} \h{2pt}=\h{2pt} 1 - \left( 1 -  \cos \psi\right) +  \int_0^{\psi} \int_0^{\psi_1} \sqrt{3}  \h{2pt}\p^2_{\psi}\h{1pt} u_{1}^- \h{2pt}\Big|_{\left(\sigma, \h{1pt}\zeta\right)} + \p^2_{\psi}\h{1pt} \left[ u_{2}^- - \cos \psi \right] \h{2pt}\Big|_{\left(\sigma, \h{1pt}\zeta\right)} \h{1pt}\mathrm{d}\h{1pt}\zeta \h{2pt}\mathrm{d}\h{1pt}\psi_1.
\end{eqnarray*}Applying the estimate in (5.6) to the last equality then yields \begin{eqnarray*} \sqrt{3}\h{1pt}u_{1}^-\h{1pt}\Big|_{\left(\sigma, \psi\right)} + u_{2}^-\h{1pt}\Big|_{\left(\sigma, \psi \right)} \h{2pt}\leq\h{2pt} 1 - \left( 1 -  \cos \psi\right) + 2 \h{1pt}\epsilon \h{1.5pt}\psi^2,\h{20pt}\text{for any $\psi \in \big[\h{0.5pt}0, \pi\h{0.5pt}\big]$.}
\end{eqnarray*}Therefore we can fix $\epsilon$ sufficiently small and find a universal constant $\psi_0 \in (\h{0.5pt}0, \pi / 4 \h{0.5pt})$ so that \begin{eqnarray}\sqrt{3}\h{1pt}u_{1}^-\h{1pt}\Big|_{\left(\sigma, \psi\right)} + u_{2}^-\h{1pt}\Big|_{\left(\sigma, \psi \right)} \h{2pt}<\h{2pt} 1, \h{20pt}\text{for any $\sigma \in \big(\h{1pt}0, \delta\left(\epsilon\right) \big]$ and $\psi \in \big(\h{0.5pt}0, \psi_0\h{0.5pt}\big)$.}
\end{eqnarray}Still by (5.6), on the sector with the polar angle $\psi \in \big[\h{0.5pt}\psi_0, \pi \h{0.5pt}\big]$, it holds \begin{eqnarray*}\sqrt{3}\h{1pt}u_{1}^-\h{1pt}\Big|_{\left(\sigma, \psi\right)} + u_{2}^-\h{1pt}\Big|_{\left(\sigma, \psi \right)} \h{2pt}=\h{2pt} \cos \psi + \sqrt{3}\h{1pt}u_{1}^-\h{1pt}\Big|_{\left(\sigma, \psi\right)} + \left[\h{1pt}u_{2}^-\h{1pt}\Big|_{\left(\sigma, \psi \right)} - \cos \psi \right] \h{2pt}\leq\h{2pt}\cos \psi_0 + 4 \h{0.5pt}\epsilon.
\end{eqnarray*}In light of the last estimate, we can keep choosing $\epsilon$ small enough and obtain $\sqrt{3}\h{1pt}u_{1}^-\h{1pt}\Big|_{\left(\sigma, \psi\right)} + u_{2}^-\h{1pt}\Big|_{\left(\sigma, \psi \right)} \h{2pt}<\h{2pt} 1$,  for any $\sigma \in \big(\h{1pt}0, \delta\left(\epsilon\right) \big]$ and $\psi \in \big[\h{0.5pt}\psi_0, \pi\h{0.5pt}\big]$. This estimate and (5.8) infer \begin{eqnarray*}\sqrt{3}\h{1pt}u_{ 1}^-\h{1pt}\Big|_{\left(\sigma, \psi\right)} + u_{2}^-\h{1pt}\Big|_{\left(\sigma, \psi \right)} \h{2pt}<\h{2pt} 1, \h{10pt}\text{for any $\sigma \in \big(\h{1pt}0, \delta\left(\epsilon\right) \big]$ and $\psi \in \big(\h{0.5pt}0, \pi\h{0.5pt}\big)$, provided $\epsilon$ is sufficiently small.}
\end{eqnarray*} When $\psi \in (0, \pi)$, $u_{1}^-$ is strictly positive. It turns out $$\sqrt{3}\h{1pt}u_{ 1}^-\h{1pt}\Big|_{\left(\sigma, \psi\right)} + u_{ 2}^-\h{1pt}\Big|_{\left(\sigma, \psi \right)} \h{2pt}>\h{2pt} - 1, \h{20pt}\text{ for any $\sigma \in \big(\h{0.5pt}0, \delta\left(\epsilon\right) \big]$ and  $\psi \in (0, \pi)$.}$$ Utilizing the last two bounds then yields \begin{eqnarray} \left( \sqrt{3}\h{1pt}u_{ 1}^-\h{1pt}\Big|_{\left(\sigma, \psi\right)} + u_{2}^-\h{1pt}\Big|_{\left(\sigma, \psi \right)} \right)^2 < 1, \h{20pt}\text{for any $\sigma \in \big(\h{1pt}0, \delta\left(\epsilon\right) \big]$ and $\psi \in \big(\h{0.5pt} 0, \pi \h{0.5pt}\big)$.}
\end{eqnarray} Equivalently  we have \begin{eqnarray} \left(u_{1}^- - \sqrt{3}\h{1pt}u_{2}^- \right)^2 + 4 \left(u_{*, 3}^-\right)^2 \h{2pt}> \h{2pt} 9 \left( u_{1}^- + \dfrac{1}{\sqrt{3}} \h{1pt}u_{2}^- \right)^2 \h{20pt} \text{on $\big(\h{1pt}0, \delta\left(\epsilon\right) \big]\times \big(\h{0.5pt} 0, \pi \h{0.5pt}\big)$.}
\end{eqnarray} By Definition 1.3, it holds \begin{eqnarray} \lambda_2 - \lambda_1 = \dfrac{1}{4} \bigg\{ 3\left( u_{1}^- + \dfrac{1}{\sqrt{3}} \h{1pt}u_{2}^- \right) - \sqrt{\big( u_{1}^- - \sqrt{3} \h{0.5pt}u_{2}^- \big)^2 + 4 \left(u_{3}^-\right)^2 } \h{1pt} \bigg\}.
\end{eqnarray} Moreover we also have \begin{eqnarray} \lambda_3 - \lambda_1 = \dfrac{1}{4} \bigg\{ 3\left( u_{1}^- + \dfrac{1}{\sqrt{3}} \h{1pt}u_{2}^- \right) + \sqrt{\big( u_{1}^- - \sqrt{3} \h{0.5pt}u_{2}^- \big)^2 + 4 \left(u_{3}^-\right)^2 } \h{1pt} \bigg\}.
\end{eqnarray}Applying (5.10) to the right-hand sides of (5.11)-(5.12) and utilizing the symmetry of $u^-$ with respect to the variable $z$, we get \begin{eqnarray}\lambda_3 \h{1pt}>\h{1pt} \lambda_1 \h{1pt}> \h{1pt}\lambda_2 \h{20pt}\text{on $\Big[ D_{\delta\left(\epsilon\right)}\left(x_0^+\right) \cup D_{\delta\left(\epsilon\right)}\left(x_0^-\right) \Big] \sim l_3$.}
\end{eqnarray}Now we fix a sufficiently small $\epsilon$ so that (5.13) holds. We are left to compare the three eigenvalues on $\mathrm{R}_{\star}$, where $\mathrm{R}_{\star}$ is the rectangle in $(x_1, z)$\h{0.5pt}-\h{0.5pt}plane with four vertices given by $x_1^{+}$, $x_1^{-}$, $x_2^{+}$, $x_2^-$. Since $u^-$ is regular on $\mathrm{R}_{\star}$, $u^-$ is close to $\left(\h{1pt}0, -1, 0\h{0.5pt}\right)$ on $\mathrm{R}_{\star}$, provided $\epsilon_1$ is small enough. It turns out $\lambda_1 > \lambda_2$ on $\mathrm{R}_{\star}$. Using the strict positivity of $u_{1}^-$, we have $$\sqrt{3}\h{1pt}u_{1}^- + u_{2}^- \h{2pt}>\h{2pt} u_{2}^- \h{2pt}\geq\h{2pt} -1 \h{10pt}\text{ on $\mathrm{D}_{\epsilon, \epsilon_1} \sim l_3$.}$$ Moreover $\sqrt{3}\h{1pt}u_{1}^- + u_{2}^- < 1$ on $\mathrm{R}_{\star}$ in that $u^-$ is close to $\left(\h{1pt}0, -1, 0\h{0.5pt}\right)$ on $\mathrm{R}_{\star}$. It turns out $ \left(\sqrt{3}\h{1pt}u_{1}^- + u_{2}^- \right)^2 \h{2pt}<\h{2pt}1$ on $\mathrm{R}_{\star} \sim l_3$.  Equivalently  we have $\left(u_{1}^- - \sqrt{3}\h{1pt}u_{2}^- \right)^2 + 4 \left(u_{3}^-\right)^2 \h{2pt}> \h{2pt} 9 \left( u_{1}^- + \dfrac{1}{\sqrt{3}} \h{1pt}u_{2}^- \right)^2$  on $\mathrm{R}_{\star} \sim l_3$. Therefore $\lambda_3 > \lambda_1$ on $\mathrm{R}_{\star} \sim l_3$ by this inequality and (5.12). Hence it holds $\lambda_3 \h{1pt}>\h{1pt} \lambda_1 \h{1pt}> \h{1pt}\lambda_2$ on $\mathrm{R}_{\star} \sim l_3$, which together with  (5.13) then infer the biaxiality of $u^-$ on $\mathrm{D}_{\epsilon, \epsilon_1} \sim l_3$. More precisely we have  \begin{eqnarray}\lambda_3 \h{1pt}>\h{1pt} \lambda_1 \h{1pt}> \h{1pt}\lambda_2 \h{20pt}\text{on $\mathrm{D}_{\epsilon, \epsilon_1} \sim l_3$, provided that $\epsilon$ and $\epsilon_1$ are sufficiently small.}
\end{eqnarray}

In the end we describe the variation of director field near the vertical segment connecting $x_0^+$ and $x_0^-$. In light of (5.14), the largest eigenvalue is $\lambda_3$ on $\mathrm{D}_{\epsilon, \epsilon_1} \sim l_3$. Therefore $\kappa_\star$ in (4) of Theorem 1.7 is a director field determined by $u^-$ on  $\mathrm{D}_{\epsilon, \epsilon_1} \sim l_3$.  Fixing $\big(0, 0, z\big)$ with $z \in \big(\h{0.5pt}z_0, \h{1pt} z_0 + \delta\left(\epsilon\right)\h{1pt}\big]$, we have $u^-$ approaching to $(0, 1, 0)$ while $x$ approaches to $(0, 0, z)$. It then turns out \begin{eqnarray}\dfrac{\sqrt{2}}{2} \left\{ 1 + \dfrac{u_{1}^- - \sqrt{3}\h{1pt}u_{2}^-}{\sqrt{\left( u_{1}^- - \sqrt{3}\h{1pt}u_{2}^- \right)^2 + 4 \left(u_{3}^-\right)^2}} \right\}^{1/2} \longrightarrow 0, \h{10pt}\text{as $x \rightarrow (0, 0, z)$.}
\end{eqnarray}Since $u_{3}^-$ is positive when $x$ is close to $(0, 0, z)$, we have \begin{small}\begin{eqnarray*}&&\dfrac{\sqrt{2}\h{1pt}u_{3}^-}{\sqrt{\left(u_{1}^- - \sqrt{3}\h{1pt}u_{2}^- \right)^2 + 4 \left( u_{3}^- \right)^2}}\left\{ 1 + \dfrac{u_{1}^- - \sqrt{3}\h{1pt}u_{2}^-}{\sqrt{\left( u_{1}^- - \sqrt{3}\h{1pt}u_{2}^- \right)^2 + 4 \left(u_{3}^-\right)^2}} \right\}^{- 1/2} \\[3mm]
 && =  \dfrac{\sqrt{2}}{2} \bigg\{ \left( u_{ 1}^- - \sqrt{3}\h{1pt}u_{2}^- \right)^2 + 4 \left(u_{3}^-\right)^2 \bigg\}^{- 1/4}\left\{ \sqrt{\left( u_{1}^- - \sqrt{3}\h{1pt}u_{2}^- \right)^2 + 4 \left(u_{3}^-\right)^2} - \left(u_{1}^- - \sqrt{3}\h{1pt}u_{2}^- \right) \right\}^{1/2},
\end{eqnarray*}\end{small}\noindent which infers \begin{small}\begin{eqnarray*}\dfrac{\sqrt{2}\h{1pt}u_{ 3}^-}{\sqrt{\left(u_{ 1}^- - \sqrt{3}\h{1pt}u_{ 2}^- \right)^2 + 4 \left( u_{ 3}^- \right)^2}}\left\{ 1 + \dfrac{u_{  1}^- - \sqrt{3}\h{1pt}u_{ 2}^-}{\sqrt{\left( u_{ 1}^- - \sqrt{3}\h{1pt}u_{ 2}^- \right)^2 + 4 \left(u_{ 3}^-\right)^2}} \right\}^{- 1/2} \longrightarrow 1, \h{5pt}\text{as $x \rightarrow (0, 0, z)$.}
\end{eqnarray*}\end{small}\noindent Applying this limit and (5.15), we obtain \begin{eqnarray}\kappa_\star \longrightarrow e_z, \h{10pt}\text{as $x \rightarrow (0, 0, z)$, where $z \in \big(z_0, \h{1pt}z_0 + \delta\left(\epsilon\right)\big]$.}
\end{eqnarray}Utilizing the strict negativity of $u_{3}^-$ in the lower-half plane, we also have \begin{eqnarray}\kappa_\star \longrightarrow - e_z, \h{10pt}\text{as $x \rightarrow (0, 0, z)$, where $z \in \big[- z_0 - \delta\left(\epsilon\right), \h{1pt}- z_0 \h{1pt}\big)$.}
\end{eqnarray}As for $(0, 0, z)$ with $z \in (- z_0, z_0)$, $u^-$ is close to $(0, -1, 0) $ when $x $ is close to $(0, 0, z)$. It follows \begin{eqnarray}\kappa_\star \longrightarrow e_{\rho}, \h{10pt}\text{as $x \rightarrow (0, 0, z)$, where $z \in \big( - z_0, z_0 \h{1pt}\big)$.}
\end{eqnarray}Moreover by $u_{3}^- = 0$ on $T$, we also have \begin{eqnarray}\kappa_\star  \h{2pt}=\h{2pt}e_{\rho} \h{20pt}\text{on $\mathrm{D}_{\epsilon, \epsilon_1} \cap \text{$x_1$\h{0.5pt}-\h{0.5pt}axis},$ provided $\epsilon_1$ is small.}
\end{eqnarray}Now we denote by $\mathscr{C}^+$ the part of $\mathscr{C}$ with positive $x_1$\h{0.5pt}-\h{0.5pt}coordinate. By (5.16)-(5.19), the director field $\kappa_\star$ varies from $- e_z$ to $e_{\rho}$ and then to $e_z$ when points vary from $\big(0, 0, - z_0 - \delta\left(\epsilon\right)\big)$ to $\big(0, 0, z_0 + \delta\left(\epsilon\right)\big)$ along the contour $\mathscr{C}^+$. Since  $\kappa_\star $ has positive coefficient in front of $e_{\rho}$ for all points on $\mathscr{C}^+ \sim l_3$, therefore $\kappa_\star $ keeps on the right-half part of $(\rho, z)$-plane for all points on $\mathscr{C}^+ \sim l_3$. When  points are changed from $\big(0, 0, - z_0 - \delta\left(\epsilon\right)\big)$ to $\big(0, 0, z_0 + \delta\left(\epsilon\right)\big)$ along $\mathscr{C}^+$, the angle of $\kappa_\star $ is totally changed by $\pi$.  Part (5) in Theorem 1.7 follows from (5.6), (5.18) and the representation of $\kappa_\star$.\vspace{2pc}
\\
\setcounter{section}{6}
\setcounter{theorem}{0}
\setcounter{equation}{0}
\textbf{\large VI. \small CONVERGENCE RESULTS}\vspace{1pc}\\
In this last section, we prove Theorem 1.8 and the convergence result in Theorem 1.5.\vspace{1pc}\\
\textbf{\large VI.1. \small PROOF OF THEOREM 1.8}\vspace{1pc}\\
$\mathcal{W}_b^+$ is uniformly bounded in $H^1(B_1; \mathbb{S}^4)$ for all $b \in \mathrm{I}_-$. Up to a subsequence, there is a $\mathcal{W}_\star^+ \in H^1(B_1; \mathbb{S}^4)$ so that $\mathcal{W}_b^+ $ converges weakly in $H^1(B_1; \mathbb{S}^4)$ to $\mathcal{W}_\star^+$  as $b \rightarrow - 1$. $\mathcal{W}_\star^+$ is a weak solution to (1.13). Moreover there is a $\mathbb{S}^2$-valued mapping $u_\star^+$ on $B_1$ so that $\mathcal{W}_\star^+ = L\h{1pt}\big[\h{1pt}u_\star^+ \h{1pt}\big]$. Since the energy $E_\mu \big[ \h{1pt}u_b^+\h{1pt}\big]$ is non-decreasing with respect to $b$, it holds by lower semi-continuity that \begin{eqnarray} E_\mu\big[\h{1pt}u_\star^+\h{1pt}\big] \leq E_\mu \big[\h{1pt}u_b^+\h{1pt}\big] < E_\mu \big[\h{1pt}U^*\h{1pt}\big], \h{20pt}\text{for any $b \in \mathrm{I}_-$.}
\end{eqnarray}

Fix $x_0 \in l_3^+$ and a radius $r$ satisfying $B_r(x_0) \subset\subset B_1^+$.  By the fact that $u_b^+$ is the minimizer of the first minimization problem in (1.12), we can apply similar construction as in the proof of convergence theorem 5.5 in [31] to show  $\mathcal{W}_b^+$ converges to $\mathcal{W}_\star^+$ strongly in $H^1(B_r(x_0); \mathbb{S}^4)$. As for points on $\mathbb{D}$, there is no obstruction to prevent $u_b^+$ converging to $u_\star^+$ strongly in $H^1_{\mathrm{loc}}(\mathbb{D}; \mathbb{S}^2)$. Otherwise by standard bubbling analysis (see [51]), there will be a non-trivial harmonic map from $\mathbb{R}^2$ to $\mathbb{S}^2$ which is a bubble map obtained by this bubbling analysis. However this is impossible since for all $b \in \mathrm{I}_-$, $u_{b; 1}^+ > 0$ on $\mathbb{D}$. At the origin, by (3.1) and (1.12), it holds \begin{eqnarray}r^{-1}\int_{B_r}\big|\h{1pt}\nabla \mathcal{W}_\star^+ \h{1pt}\big|^2 \leq \liminf_{b \h{0.5pt}\rightarrow \h{0.5pt}-1} \h{1pt} r^{-1}\int_{B_r}\big|\h{1pt}\nabla \mathcal{W}_b^+ \h{1pt}\big|^2 \leq 24 \pi, \h{20pt}\text{for any $r \in (0, 1)$.}
  \end{eqnarray} Hence there is a universal constant $c > 0$ so that  for any $r_0 \in (0, 1)$ and $\sigma \in (0, r_0)$, $$ \int_{B_{r_0}} \big|\h{1pt}\nabla \mathcal{W}_b^+ - \nabla \mathcal{W}_\star^+ \h{1pt}\big|^2 = \int_{B_{r_0} \h{1pt}\sim\h{1pt}B_{\sigma}} \big|\h{1pt}\nabla \mathcal{W}_b^+ - \nabla \mathcal{W}_\star^+ \h{1pt}\big|^2 + \int_{B_{\sigma}} \big|\h{1pt}\nabla \mathcal{W}_b^+ - \nabla \mathcal{W}_\star^+ \h{1pt}\big|^2 \leq c \h{1pt}\sigma + \int_{B_{r_0} \h{1pt}\sim\h{1pt}B_{\sigma}} \big|\h{1pt}\nabla \mathcal{W}_b^+ - \nabla \mathcal{W}_\star^+ \h{1pt}\big|^2.$$ Here we have used (6.2) and (3.1) to get the last inequality above. By taking $b \rightarrow - 1$ and $\sigma \rightarrow 0$ successively in the last estimate, it turns out $\mathcal{W}_b^+$ converges to $\mathcal{W}_\star^+$ strongly in $H^1(B_{r_0}; \mathbb{S}^4)$. In summary we have $\mathcal{W}_b^+$ converges to $\mathcal{W}_\star^+$ strongly in $H^1_{\mathrm{loc}}(B_1; \mathbb{S}^4)$. Applying this strong convergence and Fatou's lemma to (3.1), we also have (3.1) holds for $\mathcal{W}_\star^+$. Similar monotonicity inequality holds for $\mathcal{W}_\star^+$ at other points in $B_1$. Therefore (6.1) and proof of Proposition 3.2 can be applied to infer the smoothness of $\mathcal{W}_\star^+$ at the origin. More precisely $\mathcal{W}_\star^+$ is smooth in a neighborhood of $0$. By this smoothness result and $H^1_{\mathrm{loc}}(B_1; \mathbb{S}^4)$-convergence of $\mathcal{W}_b^+$ as $b \rightarrow - 1$, for $\epsilon > 0$ sufficiently small, we can find a uniform radius $r_\star$ and $b_1 \in \mathrm{I}_-$ so that $\displaystyle r_\star^{-1}\int_{B_{r_\star}}\big|\h{1pt}\nabla \mathcal{W}_b^+ \h{1pt}\big|^2 < \epsilon$, for all $b \in (-1, b_1)$. By standard $\epsilon$-regularity result and Arzel\`{a}-Ascoli theorem, $\mathcal{W}_b^+$ converges to $\mathcal{W}_\star^+$ uniformly in a neighborhood of $0$. Therefore $\mathcal{W}_\star^+$ is continuous on $T$ with its third component $\mathcal{W}_{\star; 3}^+ = 1$ at $0$. Utilizing maximum principle we also have $u_{\star; 1}^+ > 0$ on $\mathbb{D}$. There then  exists a $b_\star \in \mathrm{I}_-$ so that $u_{\star; 2}^+ = \mathcal{W}_{\star; 3}^+ \geq b_\star$ on $T$. That is $u_\star^+ \in \mathscr{F}_{b_\star, +}$, which induces $E_\mu\big[\h{1pt}u_\star^+\h{1pt}\big] \geq \mathrm{Min} \h{2pt} \bigg\{ \h{1pt}E_{\mu}\left[\h{1pt}u\right] : u \in \mathscr{F}_{\h{1pt}b_\star, +} \bigg\}.$ By this lower bound and (6.1), $u_\star^+$ is a minimizer of $E_\mu$-energy on $\mathscr{F}_{\h{1pt}b_\star, +}$. $u_\star^+$ is then the $u_{b_\star}^+$ in Theorem 1.8. Moreover by lower semi-continuity (or (6.1)), it holds $\displaystyle E_\mu\big[\h{1pt}u_\star^+\h{1pt}\big] \h{2pt}\leq\h{2pt} \liminf_{b \h{1pt}\rightarrow \h{1pt} - 1} E_\mu \big[\h{1pt}u_b^+\h{1pt}\big].$ Utilizling the non-decreasing  monotonicity of $E_\mu\big[\h{1pt}u_b^+\h{1pt}\big]$ yields $\displaystyle \limsup_{b \h{1pt}\rightarrow \h{1pt} - 1} E_\mu \big[\h{1pt}u_b^+\h{1pt}\big] \h{2pt} \leq \h{2pt} E_\mu\big[\h{1pt}u_\star^+\h{1pt}\big].$ Then we induce that $E_\mu\big[\h{1pt}u_b^+\h{1pt}\big]$ converges to $E_\mu\big[\h{1pt}u_\star^+\h{1pt}\big]$ as $b \rightarrow -1$.

Denote by $u_\star$ a minimizer of $E_\mu$-energy on $\mathscr{F}^s$. Moreover we assume $u_{\star; 1} \geq 0$, $u_{\star; 3} \geq 0$ on $B_1^+$. Similar arguments for $u_b^+$ or $u_c^-$ can be applied to show that $u_{\star}$ is regular on $T$. Let $f$ be a smooth radial function compactly supported in $B_1$ and equal to $1$ in $B_{1/2}$. Utilizing $f$ and $U^*$, we construct $u_{\star} + \epsilon f U^*$ and denote by $u_\star^\epsilon$ the normalized vector field of $u_{\star} + \epsilon f U^*$. There exists a $b_\epsilon \in \mathrm{I}_-$ satisfying $b_\epsilon \rightarrow - 1$ as $\epsilon \rightarrow 0$ so that $u_{\star; 2}^\epsilon \geq b_\epsilon$ on $T$ in the sense of trace. It then holds $E_\mu\big[\h{1pt}u_\star^\epsilon\h{1pt}\big] \h{2pt}\geq\h{2pt}E_\mu\big[\h{1pt}u_{b_\epsilon}^+\h{1pt}\big] \h{2pt}=\h{2pt}E_\mu\big[\h{1pt}u_\star^+\h{1pt}\big].$ Taking $\epsilon \rightarrow 0$  yields $E_\mu\big[\h{1pt}u_\star\h{1pt}\big] \geq E_\mu\big[\h{1pt}u_\star^+\h{1pt}\big]$, which furthermore infers $E_\mu\big[\h{1pt}u_\star\h{1pt}\big] = E_\mu\big[\h{1pt}u_\star^+\h{1pt}\big]$. The case when $c \rightarrow 1$ can be similar discussed. The proof is then completed.\\
\\
\textbf{\large VI.2. \small PROOF OF CONVERGENCE RESULT IN THEOREM 1.5}\vspace{1pc}\\
Suppose $u_\mu$ is a minimizer of $E_\mu$-energy within configuration space $\mathscr{F}^s$. It can be shown that up to a subsequence, denoted by $\mu_k$, $u_{\mu_k}$ converges strongly in $H^1(B_1)$ to a limit $u^\star_\infty$ as $k \rightarrow \infty$. Moreover
\begin{lemma}$u^\star_\infty$ is a minimizer of the energy $E$ (see (1.8)) on $\mathscr{F}_1$. Here $\mathscr{F}_1$ is the configuration space  $\bigg\{u \in \mathscr{F}^s :  P\h{0.5pt}\big(\h{0.5pt}u\h{0.5pt}\big) = \dfrac{1}{3} \h{4pt} \text{a.e. in $B_1$} \bigg\}$.
\end{lemma}Let $u^\star_{\infty; j}$ be the $j$-th component of $u^\star_\infty$. On $B_1$, we define $u^\dagger_1 := |\h{1pt}u^\star_{\infty; 1}\h{1pt}|$, $u^\dagger_2 := u^\star_{\infty; 2}$, $u^{\dagger}_3 := u^\star_{\infty; 3} $. It holds \begin{lemma} $u^\dagger := \left(u_1^\dagger, u_2^\dagger, u_3^\dagger\right) \in \mathscr{F}_1$ and is also a minimizer of the energy $E$ in $\mathscr{F}_1$.\end{lemma}
We can also characterize $u^{\dagger}$ in terms of a minimization problem on $2$-vector fields. \begin{lemma} For a given $2$-vector field $v$ on $B_1^+$, we denote by $F$ the energy functional $ \displaystyle F\left[\h{1pt}v\h{1pt}\right] := \int_{B_1^+} \big|\h{1pt}\nabla v \h{1.5pt}\big|^2 + \dfrac{2}{\rho^2} \h{1pt}\big( 1 - v_1 \big)$.  Here $v_1$ is the first component of $v$. Associated with $F$, we introduce the configuration space  \begin{eqnarray}
\mathscr{F}_2 := \left\{v : B_1^+ \rightarrow \mathbb{R}^2 \h{3pt}\Bigg|\h{3pt} \left.\begin{array}{lcl} v = \big(v_1, v_2\big) = v(\rho, z); \h{18pt}F\left[\h{1pt}v \h{1pt}\right] < \infty;  \h{18pt}|\h{0.5pt}v \h{0.5pt}| = 1 \h{4pt} \text{a.\h{1pt}e. in $B_1^+$}; \vspace{0.5pc}\\
v_2 = 0,  \h{4pt}\text{on $T$}; \h{7pt}v_1 = z^2 - \rho^2, \h{4pt}\text{on $\p^+ B_1$;} \h{7pt}
v_2 = 2 \h{1pt}\rho\h{1pt}z, \h{7pt}\text{on $\p^+ B_1$}\end{array}\right. \right\}.
\end{eqnarray}Here $\p^+B_1$ denotes the spherical boundary of $B_1^+$. Then it holds \begin{eqnarray} v^{\dagger} = \big( \h{0.5pt}v_1^{\dagger}, v_2^{\dagger} \h{1pt}\big) \h{1pt}\in\h{1pt} \mathscr{F}_2, \h{20pt} \text{where $ v^{\dagger}_1 = 1 - \dfrac{4}{\sqrt{3}} \h{1pt}u^{\dagger}_1$,  \h{5pt}$v^{\dagger}_2 = \frac{2}{\sqrt{3}} \h{1pt}u_3^{\dagger}$.} \end{eqnarray}Moreover $v^{\dagger}$ is a minimizer of the energy $F$ on $\mathscr{F}_2$.
\end{lemma}\begin{proof}[\bf Proof] Since $\big(u^{\dagger}_3\h{1pt}\big)^2 = 1 - \big(u^{\dagger}_1\h{1pt}\big)^2 - \big(u^{\dagger}_2\h{1pt}\big)^2$ for almost all points on $B_1^+$, it turns out, by  $ P\left(u^\dagger\right) =  1 \big/ 3$ on $B_1^+$, that \begin{eqnarray}\left[\sqrt{3}\h{1pt}u^{\dagger}_1 + u^{\dagger}_2 \h{2pt}\right]^3 - 3 \h{1pt} \left[\sqrt{3}\h{1pt}u^{\dagger}_1 + u^{\dagger}_2 \h{2pt}\right] + 2 = 0 \h{20pt}\text{a.\h{1pt}e. in $B_1^+$.}
\end{eqnarray}Therefore it holds  \begin{eqnarray} \sqrt{3} \h{1.5pt}u^{\dagger}_1 + u^{\dagger}_2 = 1 \h{20pt}\text{a.\h{1pt}e. in $B^+_1$.} \end{eqnarray} Otherwise from (6.5), we have $\sqrt{3}\h{1.5pt}u^{\dagger}_1 + u^{\dagger}_2 = - 2$ on a subset of $B^+_1$ with positive Lebesgue measure. Since $\big(u^{\dagger}_1\h{0.5pt}\big)^2 + \big(u^{\dagger}_2\h{0.5pt}\big)^2 \leq 1$, it then follows $u^{\dagger}_1 = - \sqrt{3}\h{1pt}/\h{1pt}2$ and $u^{\dagger}_2 = - 1\h{1pt}/\h{1pt}2$ on this subset. This is a contradiction to the almost non-negativity of $u^{\dagger}_1$ on $B^+_1$. Plugging $u^{\dagger}_2 = 1 - \sqrt{3} \h{1.5pt} u^{\dagger}_1$ into $\left|\h{1pt} u^{\dagger}\h{1pt}\right|^2 = 1$, we obtain \begin{eqnarray}\left( 1 - \dfrac{4}{\sqrt{3}} \h{1pt}u^{\dagger}_1 \right)^2 + \left( \frac{2}{\sqrt{3}} \h{1pt}u_3^{\dagger}\right)^2 = 1\h{20pt}\text{a.\h{0.5pt}e. in $B^+_1$.}\end{eqnarray}In light of $v_1^{\dagger}$ and $v_2^{\dagger}$ given in (6.4), by (6.7) $v^{\dagger}$ must have unit length for almost all points on $B_1^+$. This result combined with $u^{\dagger} \in \mathscr{F}_1$ induces  $v^{\dagger} \in \mathscr{F}_2$. In terms of $v^{\dagger}$, the energy $E\h{1pt}\big[ \h{1pt}u^{\dagger}\h{1pt}\big]$ (see (1.8)) can be evaluated by \begin{eqnarray}E\h{1pt}\big[\h{1pt}u^{\dagger}\h{1pt}\big] = \dfrac{3}{2} \h{1pt}F\left[ \h{0.5pt}v^{\dagger}\h{1pt}\right], \h{20pt}\text{where $F$ is defined in Lemma 6.3.}
\end{eqnarray}Since $v^{\dagger} \in \mathscr{F}_2$, the above equality infers\begin{eqnarray} E\h{1pt}\big[\h{1pt}u^{\dagger}\h{1pt}\big] \h{2pt} \geq \h{2pt} \dfrac{3}{2} \h{2pt}\mathrm{Min}\h{2pt}\bigg\{ F\left[\h{0.5pt}v\h{0.5pt}\right] : v \in \mathscr{F}_2 \h{1pt} \bigg\}.
\end{eqnarray}

Letting $v = (v_1, v_2) \in \mathscr{F}_2$ be arbitrarily given, we define $u^v_1 := \sqrt{3}\big(1 - v_1\big) \big/ 4$, $u^v_2 := \big(1 + 3 v_1\big)\big/4$ and $u^v_3 :=  \sqrt{3}v_2\big/2$.  Moreover we extend $u^v := \left(u^v_1, u^v_2, u^v_3\right)$ to $B_1$ so that the extension, still denoted by $u^v$, is $\mathscr{R}$-axially symmetric.   $u^v  \h{1pt}\in\h{1pt} \mathscr{F}_1$. By Lemma 6.2, it  holds $ E\big[\h{1pt}u^{\dagger}\h{1pt}\big] \h{2pt}\leq\h{2pt} E\left[\h{1pt}u^v\h{1pt}\right] \h{2pt}=\h{2pt} \dfrac{3}{2} F\left[\h{0.5pt}v\h{0.5pt}\right]$. Taking minimum  over $v \in \mathscr{F}_2$, we obtain from this estimate that $E\big[\h{1pt}u^{\dagger}\h{1pt}\big] \h{2pt}\leq\h{2pt} \dfrac{3}{2} \h{2pt}\mathrm{Min}\h{2pt}\Big\{ F\left[\h{0.5pt}v\h{0.5pt}\right] : v \in \mathscr{F}_2 \h{1pt} \Big\}$. The lemma then follows by this upper bound and (6.8)-(6.9).
\end{proof}  By regularity result in [49], we can obtain the regularity of $v^\dagger$, which infers the following regularity of $u^{\dagger}$. \begin{lemma} $u^{\dagger}$ is smooth on $\mathbb{D}^+$ and continuous up to $\p \h{1pt}\mathbb{D}^+ \sim l_3$. Moreover $u^{\dagger} \equiv \big(  \sqrt{3}\big/2 , \h{1pt}- 1/2, \h{1pt} 0\h{1pt} \big)$ on \h{2pt}$ T^\circ$.
\end{lemma}Here $u^{\dagger} \equiv \big(  \sqrt{3}\big/2 , \h{1pt}- 1/2, \h{1pt} 0\h{1pt} \big)$ on \h{2pt}$ T^\circ$ follows from (6.6), $u_3^\dagger = 0$ on $T$,  the unit length of $u^\dagger$, the continuity of $u^\dagger$ on $T \sim \big\{0\big\}$ and the fact that $u^\dagger(1, 0) = \big(  \sqrt{3}\big/2 , \h{1pt}- 1/2, \h{1pt} 0\h{1pt} \big)$. Now we show \begin{proof}[\bf Proof of convergence result in Theorem 1.5] The proof is divided into three steps.\vspace{0.3pc}
\\
\textbf{Step 1. Lower Bound.} Suppose that  $\beta = \beta(\rho, z)$ is an angular function on $B_1$ satisfying  \begin{eqnarray} \int_{B_1} |\h{0.5pt} \nabla \beta \h{1pt}|^2 + \frac{1}{\rho^2} \sin^2 \beta \h{2pt} < \h{2pt} \infty \h{20pt}\text{and}\h{20pt}\beta = \varphi \h{6pt}\text{on $\p B_1$.}
\end{eqnarray}Associated with $\beta$, we define $w_{\beta} := \big( \sin \beta \cos \theta, \sin \beta \sin \theta, \cos \beta \big)$. The map $w_{\beta}$ is an $\mathbb{S}^2$-valued vector field on $B_1$. It satisfies  $\displaystyle \int_{B_1} |\h{1pt}\nabla w_{\beta}\h{1pt}|^2 = \int_{B_1} |\h{0.5pt} \nabla \beta \h{0.5pt}|^2 + \frac{1}{\rho^2} \sin^2 \beta \h{2pt} < \h{2pt} \infty$. Moreover $w_{\beta} (x) = x$ on $\p B_1$. Using Theorem 7.1 in [11], we have \begin{eqnarray} \int_{B_1} |\h{1pt}\nabla w_{\beta}\h{1pt}|^2 \h{2pt}\geq\h{2pt}\int_{B_1} | \h{1pt}\nabla w_{\varphi} \h{1pt}|^2 \h{2pt}=\h{2pt} \int_{B_1} |\h{0.5pt} \nabla \varphi \h{0.5pt}|^2 + \frac{1}{\rho^2} \sin^2 \varphi \h{2pt}=\h{2pt}8 \pi, \h{10pt}\text{for any $\beta$ satisfying (6.10).}
\end{eqnarray}
\textbf{Step 2. Lifting and Extension.} On $B_1^+$, we define $u^\ddagger_1 := u^\dagger_1$, $u^\ddagger_2 := u^\dagger_2$, $u^\ddagger_3 := \big|\h{2pt} u^\dagger_3 \h{2pt}\big|$. By (6.7) and the almost non-negativity of $u^\ddagger_3$ on $B_1^+$, there exists a unique angular function $\alpha_+$ with values in $[\h{0.5pt}0, \pi\h{0.5pt}/\h{0.5pt}2 \h{1pt}]$ so that   \begin{eqnarray} v_1^\ddagger := 1 - \frac{4}{\sqrt{3}} \h{1pt} u^\ddagger_1 =  \cos 2 \alpha_+,\h{20pt}v_2^\ddagger := \frac{2}{\sqrt{3}}\h{1pt}u^\ddagger_3 = \sin 2 \alpha_+ \h{20pt} \text{a.\h{0.5pt}e. in $B_1^+$.} \end{eqnarray}Moreover  $\alpha_+ \in H^1\big(B_1^+; [\h{1pt}0, \pi\h{1pt}/ \h{1pt}2 ]\h{1pt}\big)$. For all $(x_1, x_2, z) \in B_1^-$, we define $\alpha_-(x_1, x_2, z) = \pi - \alpha_+(x_1, x_2, - z)$. With $\alpha_+$ and $\alpha_-$, we introduce the angular function $\alpha$ on $B_1$ which equals to $\alpha_+$ and $\alpha_-$ on $B_1^+$ and $B_1^-$, respectively. By Lemma 6.4 and (6.12), it holds $\cos 2\alpha_+ = - 1$ on $T \sim \big\{ \h{0.5pt}0\h{0.5pt}\big\}$. Since $\alpha_+ \in [\h{1pt}0, \pi\h{1pt}/ 2 \h{1pt} ]$, we have $\alpha_+ = \pi\h{1pt}\big/\h{1pt}2 $ on $T \sim \big\{ \h{0.5pt}0 \h{0.5pt}\big\}$. Therefore the trace of $\alpha_+ - \pi\h{1pt}\big/\h{1pt}2$ equals to $0$ on $T$. Note that $\alpha - \pi\h{1pt}\big/ \h{1pt}2$ is the odd extension of $\alpha_+ - \pi\h{1pt}\big/ \h{1pt}2$ from $B_1^+$ to $B_1$. It  then turns out $\alpha - \pi\h{1pt}\big/ \h{1pt}2 \in H^1\big(B_1\big)$ by the $H^1$-regularity of $\alpha_+$ on $ B_1^+$. Equivalently we have the $H^1$-regularity of $\alpha$ on $B_1$. Utilizing $u^\ddagger = U^*$ on $\p^+ B_1$ and noticing the range of $\alpha_+$, we get $\alpha_+ = \varphi $ on $\p^+ B_1$, which in turn induces $\alpha = \varphi$ on $\p B_1$.
\vspace{0.5pc}
\\
\textbf{Step 3. Upper Bound and Completion of Proof.} The angular function $\alpha$ constructed in Step 2 satisfies (6.10). Then we have a $\mathbb{S}^2$-valued map $w_{\alpha} = (\sin \alpha \cos \theta, \sin \alpha \sin \theta, \cos \alpha)$. The Dirichlet energy of $w_{\alpha}$ can be calculated by \begin{eqnarray*} \int_{B_1} \big|\h{1pt}\nabla w_{\alpha}\h{1pt}\big|^2 =  \int_{B_1} |\h{0.5pt} \nabla \alpha \h{0.5pt}|^2 + \frac{1}{\rho^2} \sin^2 \alpha = 2 \int_{B_1^+} |\h{0.5pt} \nabla \alpha_+ \h{0.5pt}|^2 + \frac{1}{\rho^2} \sin^2 \alpha_+.  \end{eqnarray*}In light of (6.12), it holds $$\displaystyle \int_{B_1^+} |\h{0.5pt} \nabla \alpha_+ \h{0.5pt}|^2 + \frac{1}{\rho^2} \sin^2 \alpha_+ = \dfrac{1}{4} \h{1pt}F\left[\h{1pt}v^{\ddagger}\h{1pt}\right],\h{20pt}\text{ where $v^\ddagger = \left( v^\ddagger_1, v^\ddagger_2\right)$.}$$ Now we extend $u^\ddagger := \left(u^\ddagger_1, u^\ddagger_2, u^\ddagger_3\right)$ to $B_1$ so that the extension, still denoted by $u^\ddagger$, is $\mathscr{R}$-axially symmetric. Then by (6.8), the last equality can be reduced to $\displaystyle \int_{B_1^+} |\h{0.5pt} \nabla \alpha_+ \h{0.5pt}|^2 + \frac{1}{\rho^2} \sin^2 \alpha_+ = \dfrac{1}{6} \h{1pt}E\left[\h{1pt}u^{\ddagger}\h{1pt}\right]$.   Utilizing Lemma 6.2 and the fact $U^* \in \mathscr{F}_1$, we have  $ E\big[\h{0.5pt}u^{\ddagger}\h{0.5pt}\big] \leq E\big[\h{0.5pt}u^{\dagger}\h{0.5pt}\big] \h{2pt}\leq\h{2pt}E\big[\h{1pt}U^{*}\h{0.5pt}\big] \h{2pt}=\h{2pt}24 \pi$. The above arguments infer $\displaystyle \int_{B_1} \big|\h{1pt}\nabla w_{\alpha}\h{1pt}\big|^2 \leq 8 \h{1pt}\pi$. Using this upper bound and (6.11) then yields $\displaystyle \int_{B_1} \big|\h{1pt}\nabla w_{\alpha}\h{1pt}\big|^2 = 8 \h{1pt}\pi$. Equivalently $w_{\alpha}$ saturates the minimum Dirichlet energy within the space $\bigg\{ u \in H^1\big(B_1; \mathbb{S}^2\big) : u(x) = x \h{5pt}\text{on $\p B_1$} \bigg\}$. Applying Theorem 7.1 in [11] then yields $w_{\alpha} = x \h{1pt}\big/ \h{1pt}|\h{0.5pt}x \h{0.5pt}|$ on $B_1$. That is  $\alpha = \varphi$ on $B_1$. This result show that $u^\ddagger = U^*$ on $B_1$. Using $P(u^\star_\infty) = 1/3$ and the regularity result in  Lemma 6.4, we then obtain $u^\star_\infty = U^*$ on $B_1$. The proof is finished.
\end{proof}\vspace{0.2pc}
\noindent \textbf{ACKNOWLEDGMENTS} The author is partially supported by RGC grant of Hong Kong No. 14306414.\vspace{0.8pc}\\
\noindent \textbf{REFERENCES}\vspace{0.2pc}
\begin{description}
\item{[1].} Alama, S., Bronsard, L. and Lamy, X.: \textit{Minimizers of the Landau-de Gennes Energy Around a Spherical Colloid Particle,} Arch. Ration. Mech. Anal. \textbf{222}, 427-450 (2016);
\item{[2].} Alper, O., Hardt, R. and Lin, F.-H.: \textit{Defects of liquid crystals with variable degree of orientation,} Calc. Var. 56:128 (2017);
\item{[3].} Ambrosio, L.: \textit{Existence of minimal energy configurations of nematic liquid crystals with variable degree of orientation,} Manuscripta Math. \textbf{68}, 215-228 (1990);
\item{[4].} Ambrosio, L.: \textit{Regularity of solutions of a degenerate elliptic variational problem,} Manuscripta Math. \textbf{68}, 309-326 (1990);
\item{[5].} Ambrosio, L. and Virga, E.: \textit{A boundary value problem for nematic liquid crystals with a variable degree of orientation,} Arch. Rational Mech. Anal. \textbf{114}, 335-347 (1991);
\item{[6].} An, D., Wang, W. and Zhang, P. W.: \textit{On equilibrium configurations of nematic liquid crystals droplet with anisotropic elastic energy,} Res. Math. Sci. 4:7 (2017);
\item{[7].} Ball, J. M. and Majumdar, A.: \textit{Nematic liquid crystals: from Maier-Saupe to a continuum theory,} Proceedings of the European Conference on Liquid Crystals, Colmar, France: 19-24 (2009);
\item{[8].} Ball, J. M. and Zarnescu, A.: \textit{Orientability and energy minimization in liquid crystal models,} Arch. Ration. Mech. Anal. \textbf{202}, 493-535 (2011);
\item{[9].} Bauman, P., Park, J. and Phillips, D.: \textit{Analysis of Nematic Liquid Crystals with Disclination Lines,} Arch. Ration. Mech. Anal. \textbf{205}, 795-826 (2012);
\item{[10].} Bauman, P. and Phillips, D.: \textit{Regularity and the behavior of eigenvalues for minimizers of a constrained $\mathrm{Q}$-tensor energy for liquid crystals,} Calc. Var. 55:81 (2016);
\item{[11].} Brezis, H., Coron, J.-M. and Lieb, E. H., \textit{Harmonic maps with defects,} Commun. Math. Phys. \textbf{107} (4), 649-705 (1986);
\item{[12].} Caffarelli, L. A.: \textit{Further regularity for the Signorini problem,} Comm. Partial Differential Equations \textbf{4}, no. 9, 1067-1075 (1979);
\item{[13].} Canevari, G.: \textit{Biaxiality in the asymptotic analysis of a $2$D Landau-de Gennes model for liquid crystals,} ESAIM: COCV \textbf{21}, 101-137 (2015);
\item{[14].} Canevari, G.: \textit{Line Defects in the Small Elastic Constant Limit of a Three-Dimensional Landau-de Gennes Model,} Arch. Ration. Mech. Anal. \textbf{223}, 591-676 (2017);
\item{[15].} Canevari, G., Ramaswamy, M. and Majumdar, A.: \textit{Radial symmetry on three-dimensional shells in the Landau-de Gennes theory,} Phys. D \textbf{314}, 18-34 (2016);
\item{[16].} Chiccoli, C., Pasini, P., Semeria, F., Sluckin, T. J. and Zannoni, C.: \textit{Monte Carlo Simulation of the Hedgehog Defect Core in Spin Systems,} J. Phys. II France \textbf{5}, 427-436 (1995);
\item{[17].} De Gennes, P. G. and Prost, J.: \textit{The Physics of Liquid Crystals,} 2nd ed., Oxford University Press, Oxford (1995);
\item{[18].} Duzaar, F. and Grotowski, F.: \textit{Energy minimizing harmonic maps with an obstacle at the free boundary,} Manuscripta Math. \textbf{83}, 291-314 (1994);
\item{[19].} Duzaar, F. and Steffen, K.: \textit{A partial regularity theorem for harmonic maps at a free boundary,} Asymptotic Analysis \textbf{2}, 299-343 (1989);
\item{[20].} Duzaar, F. and Steffen, K.: \textit{An optimal estimate for the singular set of harmonic mapping in the free boundary,} J. reine angew. Math. \textbf{401}, 157-187 (1989);
\item{[21].} Evans, L. C.: \textit{Partial Regularity for Stationary Harmonic Maps into Spheres,} Arch. Ration. Mech. Anal. \textbf{116}, 101-113 (1991);
\item{[22].} Evans, L. C., Kneuss, O. and Tran, H.: \textit{Partial regularity for minimizers of singular energy functionals, with application to liquid crystal models,} Transaction of AMS \textbf{368}(5), 3389-3413 (2016);
\item{[23].} Fratta, G. D., Robbins, J. M., Slastikov, V. and Zarnescu, A.: \textit{Half-integer point defects in the $\mathrm{Q}$-tensor theory of nematic liquid crystals,} Journal of Nonlinear Science \textbf{26} (1), 121-140 (2016);
\item{[24].} Gartland, E. C. and Mkaddem, S.: \textit{Instability of radial hedgehog configurations in nematic liquid crystals under Landau-de Gennes free-energy models,} Phys. Rev. E \textbf{59}, 563-567 (1999);
\item{[25].} Gartland, E. C. and Mkaddem, S.: \textit{Fine structure of defects in radial nematic droplets,} Phys. Rev. E \textbf{62}, 6694-6705 (2000);
\item{[26].} Gilbarg, D. and Trudinger, N. S.: \textit{Elliptic Partial Differential Equations of Second Order,} Springer-Verlag Berlin Heidelberg 2001;
\item{[27].} Han, Q. and Lin, F.-H.: \textit{Elliptic Partial Differential Equations,} 2nd Ed., Courant Lecture Notes 1, American Mathematical Society and Courant Institute of Mathematical Sciences (2011);
\item{[28].} Hardt, R., Kinderlehrer, D. and Lin, F.-H., \textit{Existence and partial regularity of static liquid crystal configurations,} Comm. Math. Phys. \textbf{105}, 547-570 (1986);
\item{[29].} Hardt, R., Kinderlehrer, D. and Lin, F.-H., \textit{The variety of configurations of static liquid crystals,}  pp. 115-132 in: Progress in Nonlinear Differential Equations and their Applications,
Vol. \textbf{4}, Birkh\"{a}user, 1990;
\item{[30].} Hardt, R. and Lin, F.-H.: \textit{Partially constrained boundary conditions with energy minimizing mapping,} Comm. Pure Appl. Math. \textbf{XLII} 309-334 (1989);
\item{[31].} Hardt, R., Lin, F.-H. and Poon, C.-C.: \textit{Axially symmetric harmonic maps minimizing a relaxed energy,} Comm. Pure Appl. Math. \textbf{XLV}, 417-459 (1992);
\item{[32].} Henao, D., Majumdar, A. and Pisante, A.: \textit{Uniaxial versus biaxial character of nematic equilibria in three dimensions,} Calc. Var. 56:55 (2017);
\item{[33].} Hu, Y. C., Qu, Y. and Zhang, P. W.:  \textit{On the disclination lines of nematic liquid crystals,} Commun. Comput. Phys. \textbf{19} 354-379 (2016);
\item{[34].} Ignat, R., Nguyen, L., Slastikov, V. and Zarnescu, A.: \textit{Stability of the melting hedgehog in the Landau-de Gennes theory of nematic liquid crystals,} Arch. Ration. Mech. Anal. \textbf{215}, 633-673 (2015);
\item{[35].} Ignat, R., Nguyen, L., Slastikov, V. and Zarnescu, A.: \textit{Instability of point defects in a two-dimensional nematic liquid crystal model,} Ann. I. H. Poincar\'{e} - AN \textbf{33}, 1131-1152 (2016);
\item{[36].} Kralj, S. and Virga, E.: \textit{Universal fine structure of nematic hedgehogs,} J. Phys. A Math. Gen. \textbf{24}, 829-838 (2001);
\item{[37].} Lamy, X.: \textit{Some properties of the nematic radial hedgehog in the Landau-de Gennes theory,} J. Math. Anal. Appl. \textbf{397}, 586-594 (2013);
\item{[38].} Lavrentovich, O. D. and Terent'ev, E. M.:\textit{Phase transition altering the symmetry of topological point defects (hedgehogs) in a nematic liquid crystal,} Zh. Eksp. Teor. Fiz. \textbf{91}, 2084-2096 (1986);
\item{[39].} Lemaire, L. and Wood, J. C.: \textit{Jacobi fields along harmonic $2$-spheres in $3$- and $4$-spheres are not all integrable,} Tohoku Math. J. \textbf{61}, 165-204 (2009);
\item{[40].} Lin, F.-H.: \textit{On nematic liquid crystals with variable degree of orientation,} Comm. Pure Appl. Math. \textbf{44}, 453-468 (1991);
\item{[41].} Lin, F.-H. and Wang, C.-Y.: \textit{The Analysis Of Harmonic Maps And Their Heat Flows,} World Scientific Publishing Co. Pte. Ltd (2008);
\item{[42].} Lin, F.-H. and Wang, C.-Y.: \textit{Recent developments of analysis for hydrodynamic flow of nematic liquid crystals,} Philos Trans A Math Phys Eng Sci., \textbf{372}(2029): 20130361 (2014);
\item{[43].} Luckhaus, S.: \textit{Partial H\"{o}lder continuity for minima of certain energies among maps into a Riemannian manifold,} Indiana Univ. Math. J. \textbf{37}, 349-367 (1988);
\item{[44].} Luckhaus, S.: \textit{Convergence of Minimizers for the $p$-Dirichlet Integral,} Math. Z. \textbf{213}, 449-456 (1993);
\item{[45].} Majumdar, A. and Zarnescu A.: \textit{Landau-de Gennes theory of nematic liquid crystals: the Oseen-Frank limit and beyond,} Arch. Ration. Mech. Anal. \textbf{196}, 227-280 (2010);
\item{[46].} Majumdar, A.: \textit{Equilibrium order parameters of liquid crystals in the Landau-de Gennes theory,} Eur. J. Appl. Math. \textbf{21}, 181-203 (2010);
\item{[47].} Majumdar, A.: \textit{The radial-hedgehog solution in Landau-de Gennes' theory for nematic liquid crystals,} Eur. J. Appl. Math. \textbf{23}, 61-97 (2012);
\item{[48].} Morrey, Jr. C. B.: \textit{Multiple integrals in the calculus of variations,} Grundlehren d.math.Wissenschaften in Einzeldarst., \textbf{130}, Springer Verlag, Berlin-Heidelberg-New York (1966);
\item{[49].} M\"{u}ller, F. and Schikorra, A.: \textit{Boundary regularity via Uhlenbeck-Rivi$\grave{\text{e}}$re decomposition,} Analysis International mathematical journal of analysis and its applications, \textbf{29} (2),  199-220 (2009);
\item{[50].} Ne\v{c}as, J.: \textit{Direct Methods in the Theory of Elliptic Equations,} Springer Monographs in Mathematics, Springer-Verlag Berlin Heidelberg (2012);
\item{[51].} Parker, T.: \textit{Bubble tree convergence for harmonic maps,} J. Differ. Geom. \textbf{44}, 595-633 (1996);
\item{[52].} Penzenstadler, E. and Trebin, H.-R.: \textit{Fine structure of point defects and soliton decay in nematic liquid crystals,} J. Phys. France \textbf{50}, 1027-1040 (1989);
\item{[53].} Petrosyan, A., Shahgholian, H. and Uraltseva, N.: \textit{Regularity of Free Boundaries in Obstacle-Type Problems,} Graduate Studies in Mathematics \textbf{136}, American Mathematical Society (2012);
\item{[54].} Rivi$\grave{\text{e}}$re, T. and Struwe, M.: \textit{Partial Regularity for Harmonic Maps and Related Problems,} Comm. Pure Appl. Math. \textbf{LXI}, 451-463 (2008);
\item{[55].} Rosso, R. and Virga, E. G.: \textit{Metastable nematic hedgehogs,} J. Phys. A \textbf{29}, 4247-4264 (1996);
\item{[56].} Schoen, R. and Uhlenbeck, K.: \textit{A regularity theory for harmonic maps,} J. Diff. Geom. \textbf{17}, 307-335 (1982);
\item{[57].} Schoen, R. and Uhlenbeck, K.: \textit{Boundary regularity and the Dirichlet problem of harmonic maps,} J. Diff. Geom. \textbf{18}, 253-268 (1983);
\item{[58].} Schoen, R. and Uhlenbeck, K.: \textit{Regularity of minimizing harmonic maps into the sphere,} Inventiones Math. \textbf{78}, 89-100 (1984);
\item{[59].} Schopohl, N. and Sluckin, T. J.: \textit{Hedgehog structure in nematic and magnetic systems,} J. Phys. France \textbf{49}, 1097-1101 (1988);
\item{[60].} Simon, L.: \textit{Theorems on Regularity and Singularity of Energy Minimizing Maps,} Lectures in Mathematics ETH Z\"{u}rich, Birkh\"{a}user Verlag 1996;
\item{[61].} Sonnet, A., Kilian, A. and Hess, S.: \textit{Alignment tensor versus director: Description of defects in nematic liquid crystals,} Phys. Rev. E \textbf{52}, 718-722 (1995).
\end{description}\vspace{2pc}
\noindent YONG YU\\
Department of Mathematics\\
The Chinese University of Hong Kong \\
E-mail: yongyu@math.cuhk.edu.hk
\end{document}